\documentclass[3p,sort&compress]{elsarticle}

\usepackage{amsthm,mathtools,amssymb,mathrsfs}
\newtheorem{theorem}{Theorem}
\newtheorem{proposition}[theorem]{Proposition}
\newtheorem{lemma}[theorem]{Lemma}
\newtheorem{corollary}[theorem]{Corollary}
\usepackage[UKenglish]{babel}
\usepackage{cmap}
\usepackage{hyperref,cleveref}
\usepackage{microtype}
\usepackage[adobe-utopia]{mathdesign}
\usepackage[T1]{fontenc}
\usepackage{siunitx}
\usepackage{xfrac}
\usepackage{enumitem}
\usepackage{booktabs}
\usepackage{adjustbox}
\usepackage{mathcomp}
\usepackage{relsize}
\usepackage{enumitem}
\usepackage{xcolor}

 \crefformat{footnote}{#2\footnotemark[#1]#3}

\newcommand{\reals}{\mathbb{R}}
\newcommand{\extreals}{\overline{\mathbb{R}}}
\newcommand{\posreals}{\reals_{> 0}}
\newcommand{\nnegreals}{\reals_{\geq 0}}
\newcommand{\nats}{\mathbb{N}}
\newcommand{\natz}{\mathbb{N}_{0}}

\newcommand{\genindica}[1]{\mathbf{1}_{#1}}
\newcommand{\indica}[1]{\mathbb{I}_{#1}}
\newcommand{\powset}[1]{\mathscr{P}(#1)}

\newcommand{\prev}{\mathrm{E}}

\newcommand{\dif}{\mathrm{d}}

\newcommand{\upprev}[1]{\overline{\mathrm{E}}_{#1}}
\newcommand{\upprevvovkk}{\overline{\mathrm{E}}_\mathrm{V}}
\newcommand{\lowprev}{\underline{\mathrm{E}}}
\newcommand{\lowprevvovkk}{\underline{\mathrm{E}}_\mathrm{V}}

\newcommand{\mupprev}{\overline{\mathrm{E}}_{\mathrm{meas}}}

\newcommand{\lupprev}[1]{\overline{\mathrm{Q}}_{#1}}
\newcommand{\extlupprev}[1]{\overline{\mathrm{Q}}{}_{#1}^{\hspace{0.6pt}\raisebox{1pt}{\scalebox{0.6}{\ensuremath \uparrow}}}}

\newcommand{\mupprevpreciseone}{\raisebox{0pt}[0.8em][0.6ex]{$\overline{\mathrm{E}}_{\raisebox{1.5pt}{\scriptsize $\mathrm{meas},p$}}^{\hspace*{1pt}\scalebox{0.7}{\raisebox{-1.5pt}{\ensuremath 1}}}$}}
\newcommand{\mupprevprecisetwo}{\raisebox{0pt}[0.8em][0.6ex]{$\overline{\mathrm{E}}_{\raisebox{1.5pt}{\scriptsize $\mathrm{meas},p$}}^{\hspace*{1pt}\scalebox{0.7}{\raisebox{-1.5pt}{\ensuremath 2}}}$}}

\newcommand{\sit}{x_{1:n}}

\newcommand{\tree}{\overline{\mathrm{Q}}}

\newcommand{\martingale}{\mathscr{M}}

\newcommand{\setofextsupmartb}{\overline{\mathbb{M}}_\mathrm{b}}

\newcommand{\situations}{\mathscr{X}^\ast}
\newcommand{\statespace}{\mathscr{X}}
\newcommand{\statespaceseq}[2]{\mathscr{X}_{#1:#2}}

\newcommand{\samplespace}{\Omega}

\newcommand{\setofgengambles}{\mathscr{L}}
\newcommand{\setofgenextvariables}{\overline{\mathscr{L}}}
\newcommand{\setofgenextvariablesb}{\overline{\mathscr{L}}_\mathrm{b}}

\newcommand{\setofupprev}{\overline{\mathbb{E}}}
\newcommand{\setofprev}{\mathbb{E}}

\newcommand{\setofgambles}{\mathbb{V}}
\newcommand{\setofextvariables}{\overline{\mathbb{V}}}
\newcommand{\setofextvariablesb}{\overline{\mathbb{V}}_{\mathrm{b}}}

\newcommand{\setoflimitsoffinmeasextb}{\overline{\mathbb{V}}_{\mathrm{b,lim}}}

\newcommand{\setoffingambles}{\mathbb{V}_{\mathrm{fin}}}

\title{A Particular Upper Expectation as Global Belief Model for Discrete-Time Finite-State Uncertain Processes}

\author{Natan T'Joens}
\ead{natan.tjoens@ugent.be}
\author{Jasper De Bock
}
\author{Gert de Cooman
}
\address{FLip, Ghent University, Technologiepark 125, 9052 Zwijnaarde, Belgium}

\begin{document}

\begin{abstract}
To model discrete-time finite-state uncertain processes, we argue for the use of a global belief model in the form of an upper expectation that is the most conservative one under a set of basic axioms. 
Our motivation for these axioms, which describe how local and global belief models should be related, is based on two possible interpretations for an upper expectation: a behavioural one similar to Walley's, and an interpretation in terms of upper envelopes of linear expectations.
We show that the most conservative upper expectation satisfying our axioms, that is, our model of choice, coincides with a particular version of the game-theoretic upper expectation introduced by Shafer and Vovk.
This has two important implications: it guarantees that there is a unique most conservative global belief model satisfying our axioms;
and it shows that Shafer and Vovk's model can be given an axiomatic characterisation and thereby provides an alternative motivation for adopting this model, even outside their game-theoretic framework.
Finally, we relate our model to the upper expectation resulting from a traditional measure-theoretic approach.
We show that this measure-theoretic upper expectation also satisfies the proposed axioms, which implies that it is dominated by our model or, equivalently, the game-theoretic model.
Moreover, if all local models are precise, all three models coincide.
\end{abstract}


\maketitle

\section{Introduction}\label{sec:intro}

There are various ways to describe discrete-time uncertain processes, such as Markov processes, mathematically.
For many, measure theory is and has been the preferred framework to describe the uncertain dynamics of such processes \cite{shiryaev2016probabilityPartI,shiryaev2019probabilityPartII,billingsley1995probability}.
Others have used martingales or a game-theoretic approach \cite{Shafer:2005wx,Vovk2019finance,williams1991probability}.
However, the common starting point for all these approaches are the local belief models.
They describe the dynamics of the process from one time instant to the next and they are typically learned from measurements or elicited from experts.
For instance, in some specific cases, one typically may have information about `the probability of throwing heads on the next coin toss', `the expected number of goods that are sold by a certain shop on a single day', `the probability of rain tomorrow', ...
In a measure-theoretic context, such local belief models are presented in the form of probability charges or probability measures on the local state space; in a game-theoretic context, sets of allowable bets are used.
The latter belong to a family of so-called imprecise probability models \cite{Walley:1991vk,Augustin:2014di,troffaes2014}, e.g. upper and lower previsions (or expectations), sets of desirable gambles, credal sets...  
Such imprecise probability models generalise traditional precise models, in the sense that they allow for indecision and for a representation that only expresses partial knowledge about the probabilities.
In order to allow for imprecision in a measure-theoretic context, we will often consider sets of probability measures (or charges) on the local state space. 
Then, if local state-spaces are assumed to be finite, as in our case, the game-theoretic and measure-theoretic local descriptions are mathematically equivalent.
This will be clarified in Sections~\ref{Sect: upprev} and~\ref{sect: Upper Expectations in Discrete-Time Uncertain Processes}.

In practice however, we are typically interested in more general inferences such as `the (lower and upper) expected number of tosses until the first tails is thrown', `the (lower and upper) probability of being out of stock on a given day', ...
It is less straightforward how we should directly learn about such more complicated inferences and, even if we could in principle do so, it is often not possible or feasible to gather sufficient information because of time and budget limitations.
Hence, the question arises: `How and to which extent, can we extend the information captured in the local models towards global information about the entire process?'. 

Measure theory relies on countable additivity to do this elegantly, and this results in a mathematically powerful, but rather abstract framework. 
The mathematical arguments are often encumbered by measurability constraints, and it is not always clear how to extend classical measure-theoretic notions to an imprecise probabilities setting.
On the other hand, Shafer and Vovk's game-theoretic framework defines global upper expectations constructively using the concept of a `supermartingale'.
Though it allows us to directly model imprecision and does not require measurability constraints, the involved upper expectations thus far lack a characterisation in terms of mathematical properties.
This impacts the generality of their model in the negative, since it can therefore only be motivated from a game-theoretic point of view.

We propose a different, third approach.
Our aim is to establish a global belief model, in the form of a conditional upper expectation, that extends the information contained in the local models by using a number of mathematical properties.
Notably, this model will not be bound to a single interpretation.
Indeed, its characterising properties can be justified starting from a number of different interpretations.
We consider and discuss two of the most significant ones in Section \ref{Sect: Search of a global model}; a behavioural interpretation that is more in line with a game-theoretic approach to uncertainty, and an interpretation in terms of upper envelopes of linear expectation operators.
We then define the desired global model as the most conservative---the unique model that does not contain any additional information---under this particular set of properties.
In other words, it is the so-called natural extension \cite{Walley:1991vk,troffaes2014,williams1975} under this set of properties.

Care should be taken here though. While the concept of natural---that is, most conservative---extension is well-established in the field of imprecise probabilities, it usually refers to the natural extension \emph{under coherence}. 
The advantage of this notion of coherence, and the related natural extension, is that it can be applied to general uncertainty models. 
However, 
the scope and properties of the resulting extension can sometimes be rather weak.
A first issue with natural extension under coherence is that it traditionally only allows us to work with \emph{bounded} variables. The problem of how to extend this method to a broader class of variables was addressed in \cite[Chapter 13]{troffaes2014}, yet, only for the class of \emph{unbounded real-valued} variables, whereas we wish to consider extended real-valued variables as well. There is a second, more profound issue though, which is that the natural extension under coherence lacks some basic continuity properties that we feel are desirable. 
For example, in our context of uncertain processes, natural extension under coherence would lead to a global upper expectation that does not necessarily satisfy upward monotone convergence, even if we were to restrict attention to bounded variables.\footnote{The experienced (or interested) reader can verify this using the following example (the notation and terminology is introduced further on in the paper). Interestingly, it shows that upward monotone convergence  fails even for increasing sequences of $n$-measurable gambles that converge to a limit that is itself a gamble.
Fix any $x\in\statespace{}$ and let $\lupprev{s}(f)\coloneqq f(x)$ for all $s\in\situations{}$; this means that in every situation, the next state will be $x$ with (lower) probability 1. 
Consider now the increasing sequence $\{f_n\}_{n\in\nats}$ of $n$-measurable gambles defined, for all $n\in\nats{}$, by $f_n(\omega) \coloneqq 0$ if $\omega_i = x$ for all $i\in\{1,\cdots,n\}$ and $f_n(\omega) \coloneqq 1$ for all other $\omega\in\Omega$.
Then it can be checked that the value of the natural extension under coherence is~$0$ for all gambles $f_n$, but that it is $1$---and therefore vacuous---for the limit $f\coloneqq\lim_{n\to+\infty}f_n$. This means that the (upper) probability of observing a state that is not $x$ in the first $n$ steps, for any $n$, is zero, but that the upper probability of ever observing such a state is $1$.
We are indebted to Matthias C. M. Troffaes for this example---or rather, a very similar one that inspired it. He presented it in a talk at WPMSIIP 2013 in Lugano, and recently also discussed it with us in private communication; it seems to be unpublished though.}
Furthermore, and related to this, the corresponding inferences are often very conservative, or even vacuous.
If we would only be interested in \emph{bounded} variables and uncertain processes with a \emph{finite} time horizon, all of these issues would not come to the fore, and the natural extension under coherence would then be a valuable model; see for example \cite{deCooman:2008km}.
However, in our---vastly---more general setting, where we will be looking at \emph{extended real-valued} variables and uncertain processes with \emph{infinite} time horizons, we find this model unsatisfactory for the reasons described above.
We will therefore not be considering the natural extension under coherence (alone); instead, we will consider the natural extension under a different set of axioms.
Unlike the coherence axioms, our axioms are specifically designed with the purpose of extending local models in the context of uncertain processes (with infinite time horizons). As we will see, the resulting natural---most conservative---extension yields more satisfying results when it comes to the mathematical properties of the resulting global model.

In Section \ref{Sect: Axiomatisation}, we give an alternative characterisation of our most conservative model by showing that it is the unique operator that satisfies a specific set of properties.
This alternative characterisation is particularly interesting from a technical point of view, since it allows us to establish our main result in Section~\ref{Sect: Game-theoretic upper exp}: that our model coincides with a version of the game-theoretic upper expectation introduced by Shafer and Vovk \cite{Shafer:2005wx,Vovk2019finance}.
On the one hand, this serves as an additional motivation for our model.
On the other hand, it establishes a concrete axiomatisation for the game-theoretic upper expectation.
Section~\ref{Sect: additional properties} aims at strengthening the relevance of the most conservative model that results from our axioms by listing its most salient properties.

In the final section, we couch our findings in a more measure-theoretic language.
As a first step, we put forward a sensible way of applying measure theory in our imprecise probabilities setting.
To do this, we draw inspiration from earlier work done in our group \cite{8535240}.
Subsequently, we show that the measure-theoretic upper expectation this results in, satisfies the axioms presented in Section~\ref{Sect: Search of a global model} and is therefore always dominated by the most conservative one satisying these axioms---which is our model.
Moreover, the fact that this model, entirely based on measure-theoretic principles, also satisfies our proposed axioms, is to be seen as an additional motivation to adopt them.
Furthermore, in a context where all local models are precise, meaning that every single one of them can be represented by a single probability mass function on the local state space, both upper expectations---and hence, also the game-theoretic one---turn out to be one and the same model.

As a final remark, we want to mention that our current work extends upon an earlier conference contribution \cite{Tjoens2019NaturalExtensionISIPTA}.
The added material is mainly gathered in the penultimate section; it concerns the study of how our model is related to a possible measure-theoretic model.
Apart from that, another important difference with the conference contribution lies in how local models are extended in order to become consistent with the game-theoretic approach; see the first part of Section~\ref{Sect: Game-theoretic upper exp} and compare this with \cite[Section~2]{Tjoens2019NaturalExtensionISIPTA}.

\section{Upper Expectations}\label{Sect: upprev}

We denote the set of all natural numbers, without \(0\), by \(\nats\), and let \(\natz \coloneqq \nats \cup \{0\}\). 
The set of extended real numbers is denoted by \(\extreals \coloneqq \reals \cup \{+ \infty, - \infty\}\) and is endowed with the usual order and the usual order topology (corresponding to the two-point compactification of \(\reals{}\)).  
The set of positive real numbers is denoted by \(\posreals\) and the set of non-negative real numbers by \(\nnegreals\). 
We also adopt the usual conventions for the addition between the reals and \(+\infty\) and \(-\infty\), and the conventions that \(+\infty-\infty = -\infty+\infty = +\infty\) and \(0\cdot(+\infty) = 0\cdot(-\infty) = 0\).


Informally, we consider a subject who is uncertain about the value that some variable \(Y\) assumes in a non-empty set \(\mathscr{Y}\).
More formally, we call any map on \(\mathscr{Y}\) a \emph{variable}\/ ; our informal \(Y\) is a special case: it corresponds to the identity map on \(\mathscr{Y}\). 
A subject's uncertainty about the unknown value of \(Y\) can then be represented by an \emph{upper expectation} \(\upprev{}\): an extended real-valued map defined on some subset \(\mathscr{D}\) of the set \(\setofgenextvariables{}(\mathscr{Y})\) of all extended real-valued variables on \(\mathscr{Y}\). 
An element \(f\) of \(\setofgenextvariables{}(\mathscr{Y})\) is simply called an \emph{extended real variable}. 
We say that \(f\) is bounded below if there is some real \(c\) such that \(f(y) \geq c\) for all \(y \in \mathscr{Y}\), and we say that \(f\) is bounded above if \(-f\) is bounded below. 
We call a sequence \(\{f_n\}_{n \in \natz{}}\) of extended real variables \emph{uniformly bounded below} if there is a real \(c\) such that \(f_n(y) \geq c\) for all \(n \in \natz{}\) and \(y \in \mathscr{Y}\).
An important role will be reserved for elements \(f\) of \(\smash{\setofgenextvariables{}(\mathscr{Y})}\) that are bounded, meaning that they are bounded above \emph{and} below. 
Bounded real-valued variables on \(\mathscr{Y}\) are called \emph{gambles} on \(\mathscr{Y}\), and we use \(\setofgengambles{}(\mathscr{Y})\) to denote the set of all of them. 
The set of all bounded below elements of \(\setofgenextvariables{}(\mathscr{Y})\) is denoted by \(\setofgenextvariablesb{}(\mathscr{Y})\).

Consider now the special case that \(\upprev{}\) is defined on at least the set of all bounded real-valued variables; so \(\setofgengambles{}(\mathscr{Y}) \subseteq \mathscr{D}\).
Then we call \(\upprev{}\) \emph{boundedly coherent} \cite{Walley:1991vk,troffaes2014,Augustin:2014di} if it satisfies the following three coherence axioms:
\begin{enumerate}[leftmargin=*,ref={\upshape{C\arabic*}},label={\upshape{}C\arabic*}.,itemsep=3pt, series=sepcoherence ]
\item \label{sep coherence 1} 
\(\upprev{}(f) \leq \sup f\) for all \(f \in \setofgengambles{}(\mathscr{Y})\);  \hfill [upper bound]
\item \label{sep coherence 2} 
\(\upprev{}(f+g) \leq \upprev{}(f) + \upprev{}(g)\) for all \(f,g \in \setofgengambles{}(\mathscr{Y})\); \hfill [sub-additivity]
\item \label{sep coherence 3}
\(\upprev{}(\lambda f) = \lambda \upprev{}(f)\) for all \(\lambda \in \nnegreals{}\) and \(f \in \setofgengambles{}(\mathscr{Y})\). \hfill [non-negative homogeneity]
\end{enumerate}

\noindent
In the particular case where \(\mathscr{D}=\setofgengambles{}(\mathscr{Y})\), we will simply say that \(\upprev{}\) is \emph{coherent}.
For any (boundedly) coherent $\upprev{}$, if we let \(\lowprev{}\) be the conjugate lower expectation defined by \(\lowprev{}(f) \coloneqq -\upprev{}(-f)\) for all \(f \in -\mathscr{D}\), then the following additional properties follow from \ref{sep coherence 1}--\ref{sep coherence 3}:

\begin{enumerate}[leftmargin=*,ref={\upshape{C\arabic*}},label={\upshape{}C\arabic*}.,itemsep=3pt,resume=sepcoherence]
\item \label{sep coherence 4}
\(f \leq g \Rightarrow \upprev{}(f) \leq \upprev{}(g)\) for all \(f, g \in \setofgengambles(\mathscr{Y})\); \hfill [monotonicity]
\item \label{sep coherence 5}
\(\inf f \leq \lowprev(f) \leq \upprev{}(f) \leq \sup f\) for all \(f \in \setofgengambles(\mathscr{Y})\); \hfill [bounds]
\item \label{sep coherence 6}
\(\upprev{}(f + \mu) = \upprev{}(f) + \mu\) for all \(\mu \in \reals{}\) and all \(f \in \setofgengambles(\mathscr{Y})\); \hfill [constant additivity]
\item \label{sep coherence 7}
for any $f\in\setofgengambles{}(\mathscr{Y})$ and any sequence \(\{f_n\}_{n \in \natz}\) in \(\setofgengambles{}(\mathscr{Y})\): \hfill [uniform continuity]
\begin{equation*}
\lim_{n \to +\infty} \sup{\vert f - f_n \vert} = 0~\Rightarrow\lim_{n \to +\infty}\upprev{}(f_n) = \upprev{}(f).
\end{equation*}
\end{enumerate}

\begin{proof}
We only prove \ref{sep coherence 5}, which clearly implies that \(\upprev{}\) is real-valued on \(\setofgengambles{}(\mathscr{Y})\), and the remaining properties then follow from the standard argumentation in~\cite[Section 2.6.1.]{Walley:1991vk}.

First, note that \(\upprev{}(0) = 0\) because of \ref{sep coherence 3} and our convention that \(0\cdot(+\infty) = 0\cdot(-\infty) = 0\).
Therefore, for all \(f\in\setofgengambles{}(\mathscr{Y})\), it follows from \ref{sep coherence 2} that \(0 \leq \upprev{}(f) + \upprev{}(-f)\), or equivalently, that \(-\upprev{}(-f)\leq \upprev{}(f)\). 
Applying \ref{sep coherence 1} to both sides, we find that \(\inf f=-\sup(-f)\leq-\upprev{}(-f)\leq\upprev{}(f)\leq\sup f\). 
The result now follows readily from the definition of \(\lowprev\).
\end{proof}

A gamble \(f\) is typically interpreted as an uncertain reward or gain that depends on the value that \(Y\) takes in \(\mathscr{Y}\); if \(Y\) takes the value \(y\), the (possibly negative) gain is \(f(y)\).
Then, according to Walley's behavioural interpretation \cite{Walley:1991vk}, the upper expectation \(\upprev{}(f)\) of a gamble \(f\) is a subject's infimum selling price for the gamble \(f\), implying that, for any \(\alpha > \upprev{}(f)\), the subject is willing to accept the gamble \(\alpha - f\).\footnote{\label{note1}Strictly speaking, the interpretation as infimum selling price only guarantees that there are \(\alpha > \upprev{}(f)\) \emph{arbitrarily close to} $\upprev{}(f)$ for which our subject accepts \(\alpha - f\) (by definition of the infimum). Our claim that this is true for \emph{all} \(\alpha > \upprev{}(f)\) is based on an additional implicit assumption that if a subject is willing to sell a gamble for some price, then she is also willing to sell that gamble for any higher price. 
}
Axioms \ref{sep coherence 1}--\ref{sep coherence 3} are then called rationality axioms, since they ensure that these selling prices are chosen rationally. 
Alternatively, any coherent upper expectation on \(\setofgengambles{}(\mathscr{Y})\) can equivalently be represented by a set of \emph{linear expectations} on \(\setofgengambles{}(\mathscr{Y})\): coherent upper expectations on \(\setofgengambles{}(\mathscr{Y})\) that are self-conjugate, meaning that \(\upprev{}(f) = - \upprev{}(-f)\) for all \(f \in \setofgengambles{}(\mathscr{Y})\).
More precisely, it follows from \cite[Section 3.3.3]{Walley:1991vk} that any coherent upper expectation \(\upprev{}\) on \(\setofgengambles{}(\mathscr{Y})\) is the upper envelope of the set of all linear expectations on \(\setofgengambles{}(\mathscr{Y})\) that are dominated by \(\upprev{}\); so 
\(\upprev{}(f) = \sup \bigl\{\prev{}(f) \colon (\forall g \in \setofgengambles{}(\mathscr{Y})) \, \prev{}(g) \leq \upprev{}(g) \bigr\}\), where \(\prev{}\) ranges over the linear expectations on \(\setofgengambles{}(\mathscr{Y})\).
Moreover, according to \cite[Theorem 8.15]{troffaes2014}, linear expectations on \(\setofgengambles{}(\mathscr{Y})\) are in a one-to-one relation with \emph{probability charges} on the power set \(\powset{\mathscr{Y}}\) of \(\mathscr{Y}\), being maps \(p \colon \powset{\mathscr{Y}} \to \nnegreals{}\) that are finitely additive and where \(p(\mathscr{Y}) = 1\) and \(p(\emptyset) = 0\). 
Such probability charges are more general than the more conventional \emph{probability measures}, which additionally require \(\sigma\)-additivity and typically also that the domain should be restricted to a \(\sigma\)-algebra on \(\mathscr{Y}\).
However, when \(\mathscr{Y}\) is finite, this distinction disappears, and we can simply limit ourselves to working with probability mass functions on $\mathscr{Y}$; these can be seen as probability charges or measures restricted to the domain of all singletons.
In that case, we have that \(\prev{}(f) = \sum_{y \in \mathscr{Y}} f(y) p(y)\) for any linear expectation \(\prev{}\) on \(\setofgengambles{}(\mathscr{Y})\) and any \(f \in \setofgengambles{}(\mathscr{Y})\), where \(p\) is the unique probability mass function on \(\mathscr{Y}\) defined by \(p(y) \coloneqq \prev{}(\genindica{y})\) for all \(y \in \mathscr{Y}\).
Here, we used \(\genindica{A}\) to denote the \emph{indicator} of \(A \subseteq \mathscr{Y}\): the gamble on \(\mathscr{Y}\) assuming the value \(1\) on \(A\) and \(0\) elsewhere, and where we will not distinguish in notation between \(y\) and \(\{y\}\).
Hence, if \(\mathscr{Y}\) is finite, any coherent upper expectation \(\upprev{}\) on \(\setofgengambles{}(\mathscr{Y})\) can alternatively be represented by the corresponding set of probability mass functions 
\begin{align}\label{Eq: link upprev and mass functions}
\mathbb{P}_{\mbox{\tiny$\upprev{}$}} \coloneqq \Big\{p\in\mathbb{P} \colon \big(\forall f \in \setofgengambles{}(\mathscr{Y})\big) \sum_{y \in \mathscr{Y}} f(y) p(y) \leq \upprev{}(f)\Big\},
\end{align}
where $\mathbb{P}$ denotes the set of all probability mass functions on $\mathscr{Y}$.
It then follows from \cite[Theorem 2]{Williams:2007eu}\footnote{More specifically, Equation~\eqref{Eq: second link upprev and mass functions} does not follow from \cite[Theorem 2]{Williams:2007eu} itself, but rather from the last paragraph of its proof.} that 
\begin{align}\label{Eq: second link upprev and mass functions}
\upprev{}(f) = \max \biggl\{ \sum_{y \in \mathscr{Y}} f(y) p(y) \colon p \in \mathbb{P}_{\mbox{\tiny$\upprev{}$}} \biggr\} \text{ for all } f \in \setofgengambles{}(\mathscr{Y}).
\end{align}
This alternative representation of a coherent upper expectation will be our starting point in Section~\ref{sect: relation to measure theory}, where we aim to establish a measure-theoretic global belief model in the form of an upper expectation.

It follows from the discussion above that, in general, (boundedly coherent) upper expectations can be interpreted in at least two possible ways: in a direct behavioural manner in terms of selling prices for gambles, or as a supremum over---an upper envelope of---a set of linear expectations.
As will be discussed in Section~\ref{Sect: Game-theoretic upper exp}, the behavioural interpretation allows us to motivate the game-theoretic approach proposed by Shafer and Vovk.
On the other hand, an interpretation in terms of upper envelopes of linear expectations seems more natural from a measure-theoretic point of view, where probability measures or mass functions are regarded as the primary objects for describing uncertainty.
Linear expectations are then typically obtained by integration and limiting the domain to measurable functions.
In this paper, we will not allow ourselves to be bound to any of these two interpretations.
Instead, we will motivate the defining properties of our proposed global belief model in terms of either of these interpretations.

\section{Upper Expectations in Discrete-Time Uncertain Processes}\label{sect: Upper Expectations in Discrete-Time Uncertain Processes}

We consider a \emph{discrete-time uncertain process}\/ : a sequence \(X_1, X_2, ..., X_n , ...\) of uncertain states, where the state \(X_k\) at each discrete time \(k \in \nats\) takes values in a fixed non-empty set \(\statespace{}\), called the \emph{state space}.
We will assume that this state space $\statespace{}$ is \emph{finite}.
Let a \emph{situation} \(\sit{}\) be any string \((x_1,...,x_n) \in \statespace{}_{1:n} \coloneqq \statespace{}^n\) of possible state values with finite length \(n \in \natz\).
In particular, the unique empty string \(x_{1:0}\), denoted by \(\Box\), is called the \emph{initial situation}: \(\statespace{}_{1:0} \coloneqq \{\Box\}\).
We denote the set of all situations by \(\statespace{}^\ast \coloneqq \cup_{n \in \natz{}} \mathscr{X}_{1:n}\).
We will also use the generic notations \(s\) and \(t\) to denote situations.
Furthermore, when we write \(x_{1:n} \in \situations{}\), we implicitly assume that \(n \in \natz\).

To model our uncertainty about the dynamics of an uncertain process, we associate, with every situation \(x_{1:n} \in \situations{}\), a coherent upper expectation \(\lupprev{x_{1:n}}\) on \(\setofgengambles{}(\statespace{})\).
This upper expectation expresses a subject's beliefs about the uncertain value of the next state \(X_{n + 1}\), when she has observed that \(X_1 = x_1, X_2 = x_2,\cdots, X_{n-1} = x_{n-1}\) and \(X_n = x_n\).
Hence, it gives information about how the process changes from one time instant to the next.
We will therefore also refer to \(\lupprev{x_{1:n}}\) as the \emph{local} model or the \emph{local} upper expectation associated with \(x_{1:n}\).
We gather all local models in an \emph{imprecise probability tree}\/ : a function \(\tree{}_{\bullet}\) that maps any situation \(s \in \situations{}\) to its corresponding coherent upper expectation \(\lupprev{s}\).
Hence, such an imprecise probability tree \(\tree{}_{\bullet}\), which we will also simply denote by \(\tree{}\), represents the dynamics of the uncertain process as a whole.

As discussed in the previous section, we consider two possible ways of interpreting the coherent upper expectations \(\lupprev{s}\): in terms of upper envelopes of linear expectations and in terms of betting behaviour.
The latter is, by nature, more oriented towards a game-theoretic approach, whereas the interpretation in terms of upper envelopes of linear expectations will seem more natural for a reader with a measure-theoretic background.
In fact, since \(\statespace{}\) is assumed to be finite, the coherent upper expectation \(\lupprev{s}\), for any \(s \in \situations{}\), can be represented equivalently by the set \(\mathbb{P}_{s}\coloneqq\mathbb{P}_{\mbox{\tiny$\lupprev{s}$}}\) of probability mass functions on \(\statespace{}\), related to it by Equations~\eqref{Eq: link upprev and mass functions} and~\eqref{Eq: second link upprev and mass functions}.
In measure-theoretic contexts, this representation in terms of \(\mathbb{P}_s\) is then typically preferred over the representation in terms of \(\lupprev{s}\).

The mathematical equivalence between the different approaches---the game-theoretic, the measure-theoretic and ours---on a local level is crucial, because only then, we can meaningfully compare the differences of their associated uncertainty models on a global level. 
Indeed, though it is locally only a matter of interpretational differences,
it is not immediately clear how these different approaches relate on a global level, where we look at the entire dynamics of the process and no longer only at the transition from one time instant to the next. 
Each approach will extend the local models using different concepts, assumptions and methods, resulting in possibly different global models.
Our upper expectation approach will use a limited set of simple properties, that, as we will argue, are desirable for a global model to have, under both of the interpretations we consider. 
Before we proceed to do so, we finish this section by introducing some further notation about uncertain processes.

An infinite sequence \(x_1 x_2 x_3 \cdots\) of state values is called a \emph{path}, which we denote by \(\omega \coloneqq x_1 x_2 x_3 \cdots\). 
We gather all paths in the \emph{sample space} \/ \(\Omega \coloneqq \statespace{}^\nats\).
For any path \(\omega \coloneqq x_1 x_2 \cdots \in \Omega\), the situation \(x_{1:n} \coloneqq x_1x_2 \cdots x_n\) that consists of its first \(n\) state values is denoted by \(\omega^n \in \statespaceseq{1}{n}\). 
The state value \(x_n\) at time \(n\) is denoted by \(\omega_n \in \statespace\).
An \emph{event} \(A \subseteq \Omega\) is a collection of paths, and in particular, the \emph{cylinder event} \(\Gamma(\sit{}) \coloneqq \{\omega \in \Omega\colon \omega^n = \sit{}\}\) of some situation \(\sit{} \in \situations{}\), is the set of all paths \(\omega \in \Omega\) that `go through' the situation \(\sit{}\).

A variable on \(\Omega\) is called a \emph{global} variable and we gather all extended real(-valued) global variables in the set \(\setofextvariables{} \coloneqq \setofgenextvariables{}(\Omega)\).
Similarly, we let \(\setofextvariablesb{} \coloneqq \setofgenextvariablesb{}(\Omega)\) and \(\setofgambles{} \coloneqq \setofgengambles{}(\Omega)\).
For any natural \(k \leq \ell\), we denote by \(X_{k:\ell}\) the global variable defined by \(X_{k:\ell}(\omega)\coloneqq(\omega_k,...,\omega_\ell)\) for all \(\omega \in \Omega\).
As such, the state \(X_k \coloneqq X_{k:k}\) at time \(k\) can also be regarded as a global variable.
Moreover, for any natural \(k \leq \ell\) and any map \(f \colon \statespace{}^{\ell - k +1} \to \extreals{}\), we will write \(f(X_{k:\ell})\) to denote the global extended real variable defined by \(f(X_{k:\ell}) \coloneqq f \circ X_{k:\ell}\).
We call a global extended real variable \(f\) \emph{\(n\)-measurable} for some \(n \in \natz{}\), if it only depends on the initial \(n\) state values; so \(f(\omega_1) = f(\omega_2)\) for any two paths \(\omega_1\) and \(\omega_2\) such that \(\omega_1^n = \omega_2^n\).
We will then use the notation \(f(\sit{})\) for its constant value \(f(\omega)\) on all paths \(\omega \in \Gamma(\sit{})\).
Clearly, the indicator \(\indica{\sit{}} \coloneqq \genindica{\Gamma(\sit{})}\) of the cylinder event \(\Gamma(\sit{})\) for \(\sit{} \in \situations{}\) is an \(n\)-measurable variable.
Finally, we call any \(f \in \setofextvariables{}\) \emph{finitary} if it is \(n\)-measurable for some \(n \in \natz{}\).
We gather all finitary \emph{gambles} in \(\setoffingambles{}\).

\section{In Search of a Global Belief Model}\label{Sect: Search of a global model}

Any extended real-valued map \(\upprev{} \colon \setofextvariables{} \times \situations{} \to \extreals{} \colon (f,s) \mapsto \upprev{}(f \vert s)\) will be called a \emph{global upper expectation}.
The corresponding \emph{global lower expectation} \(\lowprev{} \colon \setofextvariables{} \times \situations{} \to \extreals{} \colon (f,s) \mapsto \lowprev{}(f \vert s)\) is then defined by  conjugacy: $\lowprev{}(f\vert s) \coloneqq -\upprev{}(-f\vert s)$ for all $f\in\setofextvariables{}$ and all $s\in\situations{}$. These lower and upper expectations also give rise to upper and lower probabilities. In particular, for any event $A\subseteq\Omega$ and any $s\in\situations{}$, $\overline{\mathrm{P}}(A \vert s)\coloneqq\upprev{}(\genindica{A}\vert s)$ is the \emph{upper probability} of $A$ conditional on $s$ and, similarly,  $\underline{\mathrm{P}}(A \vert s)\coloneqq\lowprev{}(\genindica{A}\vert s)$ is the \emph{lower probability} of $A$ conditional on $s$.
Throughout this work, however, we will focus almost entirely on (global) upper expectations, and will regard (global) lower expectations and upper and lower probabilities as derived notions.

Given an imprecise probability tree \(\tree{}\) that associates a local upper expectation \(\lupprev{s}\) with every situation \(s \in \situations{}\), we aim to define a global upper expectation \(\upprev{}\) that extends the information included in these local models in a `rational' manner.
To do so, we impose the following properties, where we adopt the notation that \(\upprev{}(f \vert X_{1:n}) \coloneqq \upprev{}(f \vert \cdot) \circ X_{1:n}\) for all \(f\in\setofextvariables{}\) and all \(n\in\natz\), and where limits of variables are intended to be taken pointwise.

\begin{enumerate}[leftmargin=*,ref={\upshape P\arabic*},label={\upshape P\arabic*}.,itemsep=3pt,series=Properties]
\item \label{P compatibility}
\(\upprev{} (f(X_{n+1}) \vert \sit{}) = \lupprev{\sit{}}(f)\) \, for all \(f \in \setofgengambles{}(\statespace{})\) and all \(\sit{} \in \situations{}\).
		
\item \label{P conditional} 
\(\upprev{} (f \vert s) = \upprev{} (f \, \indica{s} \vert s)\) \, for all \(f \in \setoffingambles{}\) and all \(s \in \situations{}\).
\item \label{P iterated} 
\(\upprev{}(f \vert X_{1:n}) \leq \upprev{}(\upprev{}(f \vert X_{1:n+1}) \vert X_{1:n})\) \, for all \(f \in \setoffingambles{}\) and all \(n \in \natz\). 
\item \label{P monotonicity}  
\(f \leq g \Rightarrow \upprev{}(f \vert s) \leq \upprev{}(g \vert s)\) \, for all \(f,g \in \setofextvariables{}\) and all \(s \in \situations{}\).
\item \label{P continuity}
For any sequence \(\{f_n\}_{n \in \natz}\) of finitary gambles that is uniformly bounded below and any \(s \in \situations{}\):
\begin{equation*}
\lim_{n \to +\infty} f_n = f \Rightarrow 
\limsup_{n \to +\infty} \upprev{}(f_n \vert s) \geq \upprev{}(f \vert s).
\end{equation*} 
\end{enumerate}
\noindent
To motivate \ref{P compatibility}--\ref{P continuity}, we need to attach some interpretation to \(\upprev{}\).
We will consider two particular ones, similar to what we have done for the (boundedly) coherent upper expectations in Section~\ref{Sect: upprev}.

We start from the interpretation of a global \emph{gamble} \(f \in \setofgambles{}\) as an uncertain reward depending on the uncertain path \(\omega\) that the uncertain process takes in \(\Omega\).
However, it is not clear what this means operationally if the gamble \(f\) depends on the entire length of the path.
Indeed, the gamble \(f\) depends on an infinite number of subsequent state values, so there is no point in time when we can determine the actual reward linked to the gamble~\(f\).
The same interpretational problem arises when considering unbounded or extended real variables (on a general set \(\mathscr{Y}\) such as \(\Omega\)).
The simple interpretation of an uncertain reward does not suffice there because the reward itself can be unbounded or infinite, which is unrealistic---or even meaningless---in operational practice.
For this reason, we prefer to only attach a direct operational meaning to the value \(\upprev{}(f \vert s)\) of a global upper expectation \(\upprev{}\) for a \emph{finitary gamble} \(f \in \setoffingambles{}\) conditional on a situation \(s \in \situations{}\).
Such finitary gambles can be given an operationally meaningful behavioural interpretation as uncertain rewards because they take real values\footnote{One could also question the meaning of a direct behavioural interpretation for some specific (real-valued) gambles. For instance, how do we exchange money if the gamble's value is equal to \(\pi\)? In this case, we can only give an indirect interpretation in terms of simpler gambles that are `sufficiently' close.} and only depend on the state at a finite number of time instances.

We distinguish the following two ways for interpreting the global upper expectation \(\upprev{}(f \vert s)\) of a finitary gamble \(f \in \setoffingambles{}\) conditional on a situation \(s \in \situations{}\):
\begin{itemize}[leftmargin=*]
\item
\textbf{Behavioural interpretation.} It is a subject's infimum selling price for \(f\) \emph{contingent} on the event \(\Gamma(s)\), implying that,\footnote{The comment in Footnote~\ref{note1}, applies here as well, suitable adapted to the conditional setting.} for any \(\alpha > \upprev{}(f \vert s)\), she is willing to accept the uncertain reward associated with the gamble \(\indica{s} (\alpha - f)\).
\item
\textbf{Interpretation as an upper envelope.} 
It is the supremum value of \(\prev{}(f\vert s)\), where \(\prev{}\) belongs to some given set~\(\setofprev{}\) of conditional linear expectation operators: \(\upprev{}(f\vert s)=\sup \{\prev{}(f \vert s) \colon \prev{} \in \setofprev{}\}\).
\end{itemize} 

Since \ref{P compatibility}--\ref{P iterated} only apply to finitary gambles, a direct justification for these axioms can be given quite easily in each of the above interpretations.
Property \ref{P compatibility} requires that the global model \(\upprev{}\) should be compatible with the local models \(\lupprev{s}\).
The desirability of this property is self-evident, no matter which interpretation is used.
Property \ref{P conditional} requires that the upper expectation of a finitary gamble \(f\) conditional on \(s\) should only depend on the value of \(f\) on the paths \(\omega \in \Gamma(s)\).
This property is clearly desirable when using the behavioural interpretation, because \(\indica{s} (\alpha - f)\) only depends on the restriction of \(f\) to \(\Gamma(s)\).
It is also quite evident that this property is desirable for an upper envelope of conditional linear expectations, because \ref{P conditional} should in particular also hold for such conditional linear expectations themselves.
Similarly, that \ref{P iterated} should hold under the upper envelope interpretation, is motivated by the fact that conditional linear expectations satisfy \ref{P iterated} with equality---then better known as the law of iterated expectations; an upper envelope of conditional expectations is therefore guaranteed to satisfy \ref{P iterated}, under the assumption that \ref{P monotonicity} holds for conditional linear expectations.
In order to see that property \ref{P iterated} is also desirable according to the behavioural interpretation, one requires a conditional version of the notion of coherence that we discussed in Section~\ref{Sect: upprev}~\cite{Walley:1991vk,williams1975,troffaes2014}.\footnote{There are several versions of conditional coherence \cite{Walley:1991vk,williams1975,troffaes2014}, however, in the case where variables take values in a finite set, all these different versions are mathematically equivalent.} 
Explaining in more detail why this is the case, and what this conditional notion of coherence exactly entails, would however lead us too far. 

Having attached an interpretation to finitary gambles and their conditional upper expectations, we now proceed to do the same for more general variables.
We have already argued that no direct operational meaning can be given to such variables. 
However, this should not be taken to imply that an uncertainty model should not be able to deal with them.
In fact, they can serve as useful abstract idealisations of (sequences of) variables that can be given a direct operational meaning.
In particular, we will regard any extended real variable \(f\) that is bounded below and that can be written as the pointwise limit \(\lim_{n \to +\infty} f_n\) of some sequence of finitary gambles \(\{f_n\}_{n \in \natz{}}\), as an abstract idealisation of \(f_n\) for \emph{large} \(n\). 
We gather these limits in the set
\begin{align*}
\setoflimitsoffinmeasextb{} \coloneqq 
\biggl\{f \in \setofextvariablesb{} \colon f = \lim_{n \to +\infty} f_n \text{ for } \text{some sequence } \{f_n\}_{n \in \natz{}} \text{ in } \setoffingambles{} \biggr\}.
\end{align*}
Since \ref{P continuity} applies to precisely these kinds of variables, this axiom can be justified by extending the above idealisation from the variables \(f\) to their upper expectations \(\upprev{}(f\vert s)\). 
Basically, since \(f\) is an abstract idealisation of \(f_n\) for large \(n\), \(\upprev{}(f\vert s)\) should be an abstract idealisation of \(\upprev{}(f_n\vert s)\) for large \(n\). The practical benefit of this is that we can then use \(\upprev{}(f\vert s)\) to reason about \(\upprev{}(f_n\vert s)\) for a generic large value of \(n\), without having to specify the specific value of \(n\).
The problem, however, is that the sequence \(\smash{\bigl\{\upprev{}(f_n \vert s) \bigr\}_{n \in \natz{}}}\) may not converge. 
What we then do know for sure, however, is that as \(n\) approaches infinity, \(\upprev{}(f_n\vert s)\) will oscillate between the limit superior and inferior of the sequence \(\smash{\big\{\upprev{}(f_n \vert s) \big\}_{n \in \natz{}}}\). 
Since we want \(\upprev{}(f\vert s)\) to serve as an idealisation of \(\upprev{}(f_n\vert s)\) for generic large values of \(n\), \(\upprev{}(f\vert s)\) should therefore definitely not exceed the limit superior, as this would result in an unwarranted loss of `information', in the sense that $\upprev{}(f\vert s)$ would be too high---this will be clarified shortly.
We therefore impose~\ref{P continuity}.

The final property that we impose is \ref{P monotonicity}, which states that \(\upprev{}\) should be monotone. 
For finitary gambles~\(f\) and~\(g\), this follows easily from either of our two different interpretations for \(\upprev{}(f\vert s)\) and \(\upprev{}(g\vert s)\). 
Under a behavioural interpretation, since the reward associated with \(g\) is guaranteed to be at least as high as that of \(f\), the same should be true for a subject's infimum selling prices for these two gambles. 
Under an interpretation in terms of upper envelopes of expectations, monotonicity of the envelope is implied by the monotonicity of each of the individual expectations. 
If \(f\) and \(g\) are more general variables---so not necessarily finitary gambles---the motivation for \ref{P monotonicity} is that, still, higher rewards---even if abstract and idealized---should correspond to higher upper expectations. It is also worth noting that the combination of \ref{P monotonicity} and \ref{P continuity} implies that \(\upprev{}\) is continuous with respect to non-decreasing sequences of finitary gambles.
In a measure-theoretic context, this kind of continuity is usually obtained as a consequence of the assumption of \(\sigma\)-additivity \cite{shiryaev2016probabilityPartI,royden2010real}.

We will show in Section \ref{Sect: Game-theoretic upper exp} that \ref{P compatibility}--\ref{P continuity} are consistent, in the sense that if the local models \(\lupprev{s}\) are coherent, there always is at least one global upper expectation \(\upprev{}\) satisfying \ref{P compatibility}--\ref{P continuity}.
However, there may be more than one global upper expectation \(\upprev{}\) satisfying \ref{P compatibility}--\ref{P continuity}. In that case, the best thing to do, we believe, is to choose the most conservative model among those that satisfy \ref{P compatibility}--\ref{P continuity}, as choosing any other would mean adding `information' that is not simply implied by our axioms.
We will denote this most conservative global upper expectation by \(\upprev{\mathrm{A}}\) and let \(\lowprev{}_{\hspace*{1pt}\raisebox{0pt}{\scriptsize $\mathrm{A}$}}\) be its conjugate lower expectation. 
As we will see in Section~\ref{Sect: Game-theoretic upper exp}, \(\upprev{\mathrm{A}}\) is guaranteed to exist, and furthermore coincides with a particular version of the game-theoretic upper expectation defined by Shafer and Vovk \cite{Shafer:2005wx,shafer2012}.

Of course, in order for our definition of \(\upprev{\mathrm{A}}\) to make sense, we need to know what it means for an upper expectation \(\upprev{}{'}\) to be more conservative than some other upper expectation \(\upprev{}\). 
We here take this to mean that \(\upprev{}{'}\) is higher than \(\upprev{}\). So higher upper expectations are more conservative, or less informative. 
This is why, in our motivation for \ref{P continuity}, we said that a too high \(\upprev{}(f\vert s)\) corresponds to an unwarranted loss of `information'. 
That it is indeed reasonable to regard higher expectations as more conservative, can again be motivated using either of the two interpretations that we considered before.
Under the behavioural interpretation, higher upper expectations mean higher selling prices, which is clearly more conservative. 
Using an interpretation in terms of upper envelopes of expectations, higher upper expectations correspond to larger sets of expectations, so weaker constraints on expectations, which is again less informative and hence more conservative.

Having said all this, we would like to stress that---despite our extensive use of them to motivate our axioms---none of the results that we are about to develop hinge on any particular interpretation for upper expectations. 
Other interpretations could also be adopted, or perhaps even no interpretation at all. All that is needed is that we agree on \ref{P compatibility}--\ref{P continuity} and on the fact that higher upper expectations are more conservative.

\section{An Axiomatisation of \(\upprev{\mathrm{A}}\)}\label{Sect: Axiomatisation}
For a given imprecise probability tree \(\tree{}\), let \(\setofupprev{}_{1-2}(\tree{})\) denote the set of all global upper expectations satisfying \ref{P compatibility}--\ref{P conditional}, and similarly for \(\setofupprev{}_{1-4}(\tree{})\) and \(\setofupprev{}_{1-5}(\tree{})\).
In this section, we introduce sufficient conditions for a global upper expectation \(\upprev{}\) to be the most conservative among all upper expectations in \(\setofupprev{}_{1-5}(\tree{})\).
We start by considering the domain of finitary gambles and then, step by step, extend the domain and introduce additional conditions on \(\upprev{}\), in such a way that it becomes the most conservative on this extended domain.

For any situation \(\sit{} \in \situations\) and any \((n+1)\)-measurable (global) gamble \(f\), we use \(f(\sit{} \cdot)\) 
to denote the gamble on~\(\statespace{}\) that assumes the value \(f(x_{1:n+1})\) in \(x_{n+1} \in \statespace{}\), and then use \(f(x_{1:n}X_{n+1})\) as a shorthand for the global gamble \(f(x_{1:n}\cdot) \circ X_{n+1}\).
The following lemma establishes compatibility with the local models in a stronger way than \ref{P compatibility} does.

\begin{lemma}\label{Lemma: compatibility}
Consider any \(\upprev{} \in \setofupprev{}_{1-2}(\tree{})\).
Then, for any situation \(\sit{} \in \situations\) and any \((n+1)\)-measurable gamble \(f\),
\begin{equation*}
		\upprev{} (f \vert \sit{}) = \lupprev{\sit{}}(f(\sit{} \cdot)).
\end{equation*}
\end{lemma}
\begin{proof}
Fix any \(\sit{} \in \situations\) and any \((n+1)\)-measurable gamble \(f\).
Note that \(f(\sit{} X_{n+1}) \indica{\sit{}} = f \indica{\sit{}}\) and hence, because of \ref{P conditional}, \(\upprev{} (f \vert \sit{}) = \upprev{} (f \indica{\sit} \vert \sit{}) = \upprev{} (f(\sit{} X_{n+1}) \indica{\sit} \vert \sit{}) = \upprev{} (f(\sit{} X_{n+1}) \vert \sit{})\). 
\ref{P compatibility} therefore implies that, indeed, \(\upprev{} (f \vert \sit{}) = \lupprev{\sit{}}(f(\sit{} \cdot))\).
\end{proof}

\noindent
To make sure that some particular upper expectation \(\upprev{} \in \setofupprev{}_{1-4}(\tree{})\) is a most conservative upper expectation on the domain of all finitary gambles, we impose on it the following property, known as \emph{the law of iterated upper expectations} \cite{DECOOMAN201618,Shafer:2005wx}: 

\begin{enumerate}[leftmargin=*,itemsep=6pt,ref={\upshape P\arabic*\(^{=}\)},label={\upshape P\arabic*\(^{=}\)}.,start=3]
\item \label{P iterated general} 
\(\upprev{}(f \vert X_{1:n}) = \upprev{}(\upprev{}(f \vert X_{1:n+1}) \vert X_{1:n})\) \, for all \(f \in \setoffingambles{}\) and all \(n\in\natz\). 
\end{enumerate}

\begin{proposition}\label{proposition: largest on n-measurables} 
Consider any \(\upprev{} \in \setofupprev{}_{1-4}(\tree{})\) that satisfies \ref{P iterated general}.
Then, for any \(\upprev{}{'} \in \setofupprev{}_{1-4}(\tree{})\), we have that 
\begin{align*}
\upprev{}(f \vert s) \geq \upprev{}{'}(f \vert s) \text{ for all } f \in \setoffingambles{} \text{ and all } s \in \situations{}.
\end{align*}
\end{proposition}
\begin{proof}
For any \(p \in \natz{}\) and any \((p + 1)\)-measurable gamble \(g\), we let \(\lupprev{X_{1:p}}(g)\) be the \(p\)-measurable gamble defined by
\(\smash{\lupprev{X_{1:p}}(g)(\omega) \coloneqq \lupprev{\omega^{p}}(g(\omega^{p} \cdot))} \text{ for all } \omega \in \Omega\).
Note that \(\smash{\lupprev{X_{1:p}}(g)}\) is indeed a gamble because of (bounded) coherence [\ref{sep coherence 5}] and the fact that \(g\) is a gamble.
Then, because both the operators \(\upprev{}\) and \(\upprev{}{'}\) satisfy \ref{P compatibility} and \ref{P conditional}, it follows from Lemma~\ref{Lemma: compatibility} that
\begin{align}\label{Eq: proposition: largest on n-measurables}
\upprev{}{'}( g \vert X_{1:p}) = \lupprev{X_{1:p}}(g) = \upprev{}( g \vert X_{1:p}).
\end{align}
Now fix any \(f \in \setoffingambles{}\) and any \(x_{1:m} \in \situations{}\).
Since \(f\) is finitary, it is \(n\)-measurable for some \(n \in \natz{}\).
We can assume that \(m+2 < n\) without loss of generality, because \(f\) is obviously also \(p\)-measurable for every \(p \geq n\).
Now, it follows from Equation~\eqref{Eq: proposition: largest on n-measurables} and the fact that $\lupprev{X_{1:n-1}}(f)$ is an $(n-1)$-measurable gamble, that \(\upprev{}{'}( f \vert X_{1:n-1}) = \upprev{}( f \vert X_{1:n-1})\) is an $(n-1)$-measurable gamble.
Subsequently, we can apply Equation~\eqref{Eq: proposition: largest on n-measurables} once more, where \(\smash{\upprev{}{'}( f \vert X_{1:n-1}) = \upprev{}( f \vert X_{1:n-1})}\) now takes the role of $g$, in order to find that \(\upprev{}{'}(\upprev{}{'}( f \vert X_{1:n-1}) \vert X_{1:n-2}) = \upprev{}(\upprev{}( f \vert X_{1:n-1}) \vert X_{1:n-2})\) is an $(n-2)$-measurable gamble.
Then we can continue in the same way, to finally obtain that 
\begin{align*}
\upprev{}{'}( \upprev{}{'}( \cdots  \upprev{}{'}(f \vert X_{1:n-1}) \cdots \vert X_{1:m+1}) \vert x_{1:m})
= \upprev{}( \upprev{}( \cdots  \upprev{}(f \vert X_{1:n-1}) \cdots \vert X_{1:m+1}) \vert x_{1:m}).
\end{align*}
By repeatedly applying \ref{P iterated} and \ref{P monotonicity}, we infer that $\upprev{}{'}( \upprev{}{'}( \cdots  \upprev{}{'}(f \vert X_{1:n-1}) \cdots \vert X_{1:m+1}) \vert x_{1:m}) 
\geq \cdots \geq \upprev{}{'}(\upprev{}{'}(f \vert X_{1:m+1}) \vert x_{1:m}) 
\geq \upprev{}{'}(f \vert x_{1:m})$.
Hence, plugging this back into the previous expression, we have that 
\begin{align*}
\upprev{}{'}(f \vert x_{1:m}) \leq
 \upprev{}( \upprev{}( \cdots  \upprev{}(f \vert X_{1:n-1}) \cdots \vert X_{1:m+1}) \vert x_{1:m}).
\end{align*}
On the other hand, since $\upprev{}$ satisfies \ref{P iterated general}, we infer that
$\upprev{}( \upprev{}( \cdots  \upprev{}(f \vert X_{1:n-1}) \cdots \vert X_{1:m+1}) \vert x_{1:m}) 
= \cdots
= \upprev{}(\upprev{}(f \vert X_{1:m+1}) \vert x_{1:m})$
$= \upprev{}(f \vert x_{1:m})$.
So indeed, we find that $\upprev{}{'}(f \vert x_{1:m})\leq\upprev{}(f \vert x_{1:m})$.
\end{proof}

\noindent
Next, we consider the domain \(\setoflimitsoffinmeasextb{} \subset \setofextvariables{}\) of all extended real variables that are bounded below and that can be written as the pointwise limit of a sequence of finitary gambles.
The following condition, together with \ref{P iterated general}, turns out to be sufficient for an upper expectation \(\upprev{}\) to be a most conservative one on the domain of \(\setoflimitsoffinmeasextb{}\) amongst all upper expectations in \(\setofupprev{}_{1-5}(\tree{})\).

\begin{enumerate}[leftmargin=*, itemsep = 6pt,ref={\upshape P\arabic*},label={\upshape P\arabic*}.,resume=Properties]
\item \label{P continuity n-meas} 
For any \(f \in \setoflimitsoffinmeasextb{}\) and any \(s \in \situations{}\), there is some sequence \(\{f_n\}_{n \in \natz{}}\) of \(n\)-measurable gambles that is uniformly bounded below and that converges pointwise to \(f\), such that moreover
\(\limsup_{n \to +\infty} \upprev{}(f_n \vert s) \leq \upprev{}(f \vert s)\).
\end{enumerate}
Note that, for an upper expectation $\upprev{}\in\setofupprev{}_{1-5}(\tree{})$, the inequality \(\limsup_{n \to +\infty} \upprev{}(f_n \vert s) \leq \upprev{}(f \vert s)\) in \ref{P continuity n-meas} is in fact an equality because of \ref{P continuity}.

\noindent

\begin{proposition}\label{Prop: largest on gamble limits of n-measurables}
Consider any \(\upprev{} \in \setofupprev{}_{1-5}(\tree{})\) that satisfies {\normalfont\ref{P iterated general}} and {\normalfont\ref{P continuity n-meas}}.
Then, for all \(\upprev{}{'} \in \setofupprev{}_{1-5}(\tree{})\), we have that 
\begin{align*}
\upprev{}(f \vert s) \geq \upprev{}{'}(f \vert s) \text{ for all } f \in \setoflimitsoffinmeasextb{} \text{ and all } s \in \situations{}.
\end{align*}
\end{proposition}
\begin{proof}
Fix any \(f \in \setoflimitsoffinmeasextb{}\), any \(s \in \situations{}\) and any \(\upprev{}{'} \in \setofupprev{}_{1-5}(\tree{})\).
According to \ref{P continuity n-meas}, there is some sequence \(\{f_n\}_{n \in \natz{}}\) of \(n\)-measurable gambles that is uniformly bounded below and that converges pointwise to \(f\) in such a way that \(\limsup_{n \to +\infty} \upprev{}(f_n \vert s) \leq \upprev{}(f \vert s)\) and therefore, due to \ref{P continuity}, also that
\begin{equation}\label{Eq: prop largest on Vblim 1}
\limsup_{n \to +\infty} \upprev{}(f_n \vert s) = \upprev{}(f \vert s).\vspace{-2pt}
\end{equation}
Because all \(f_n\) are finitary gambles and both \(\upprev{}\) and \(\upprev{}{'}\) are upper expectations in \(\setofupprev{}_{1-4}(\tree{})\), with \(\upprev{}\) additionally satisfying \ref{P iterated general}, we can apply Proposition \ref{proposition: largest on n-measurables} to find that 
\begin{equation}\label{Eq: prop largest on Vblim 2}
\limsup_{n \to +\infty} \upprev{}{'}(f_n \vert s) 
\leq \limsup_{n \to +\infty} \upprev{}(f_n \vert s).
\end{equation}
Furthermore, since \(\{f_n\}_{n \in \natz{}}\) is a sequence of finitary gambles that is uniformly bounded below and that converges pointwise to \(f\), \ref{P continuity} implies that
$\upprev{}{'}(f \vert s) 
\leq \limsup_{n \to +\infty} \upprev{}{'}(f_n \vert s)$.
Combining this with Equations \eqref{Eq: prop largest on Vblim 1} and \eqref{Eq: prop largest on Vblim 2}, we find that, indeed, \(\upprev{}{'}(f\vert s) \leq \upprev{}(f\vert s)\). 
\end{proof}

\noindent
Finally, we consider the entire domain \(\setofextvariables{}\).
Now, in order for an upper expectation on \(\setofextvariables{}\) to be the most conservative one, it suffices to additionally impose the following property, as we will show presently.

\begin{enumerate}[leftmargin=*,ref={\upshape P\arabic*},label={\upshape P\arabic*}.,resume=Properties]
\item \label{P natural extension}  
For any \(f \in \setofextvariables{}\) and any \(s \in \situations{}\),
\begin{align*}
	\upprev{} (f \vert s) &= \inf{\Big\{ \upprev{}(g \vert s) \colon g \in \setoflimitsoffinmeasextb{} \text{ and }  g \geq f  \Big\}}.
\end{align*}
\end{enumerate}

\begin{theorem}\label{Theorem: smallest vovk general}
Consider any \(\upprev{} \in \setofupprev{}_{1-5}(\tree{})\) that satisfies \ref{P iterated general}, \ref{P continuity n-meas} and \ref{P natural extension}.
Then, for all \(\upprev{}{'} \in \setofupprev{}_{1-5}(\tree{})\), we have that 
\begin{align*}
\upprev{}(f \vert s) \geq \upprev{}{'}(f \vert s) \text{ for all } f \in \setofextvariables{} \text{ and all } s \in \situations{}.
\end{align*}
\end{theorem}
\begin{proof}
Fix any \(f \in \setofextvariables{}\), any \(s \in \situations{}\) and any \(\upprev{}{'} \in \setofupprev{}_{1-5}(\tree{})\).
According to \ref{P natural extension}, we have that
\begin{align*}
	\upprev{}(f \vert s) &= \inf{\Big\{ \upprev{}(g \vert s) \colon g \in \setoflimitsoffinmeasextb{} \text{ and }  g \geq f  \Big\}}.
\end{align*}
Then, using Proposition \ref{Prop: largest on gamble limits of n-measurables}, we get
\begin{align*}
\inf{\Big\{ \upprev{}{'}(g \vert s) \colon g \in \setoflimitsoffinmeasextb{} \text{ and }  g \geq f  \Big\}} 
\leq \inf{\Big\{ \upprev{}(g \vert s) \colon g \in \setoflimitsoffinmeasextb{} \text{ and }  g \geq f  \Big\}} = \upprev{} (f \vert s).
\end{align*}
Since, for any \(g \in \setoflimitsoffinmeasextb{}\) such that \(f \leq g\), we have that \(\upprev{}{'}(f\vert s) \leq \upprev{}{'}(g\vert s)\) because $\upprev{}{'}$ satisfies \ref{P monotonicity}, it follows that
\begin{align*}
\upprev{}{'}(f \vert s) \leq \inf{\Big\{ \upprev{}{'}(g \vert s) \colon g \in \setoflimitsoffinmeasextb{} \text{ and }  g \geq f  \Big\}} \leq \upprev{} (f \vert s),
\end{align*}
which completes the proof.
\end{proof}

\noindent
We conclude that, according to Theorem \ref{Theorem: smallest vovk general}, if there is some upper expectation \(\upprev{} \in \setofupprev{}_{1-5}(\tree{})\) that satisfies \ref{P iterated general}, \ref{P continuity n-meas} and \ref{P natural extension}, it must be the unique most conservative upper expectation in \(\setofupprev{}_{1-5}(\tree{})\), and it must therefore be the global belief model \(\upprev{\mathrm{A}}\) that we are after.
So it now only remains to show that there is at least one---then necessarily unique---global upper expectation in \(\setofupprev{}_{1-5}(\tree{})\) that additionally satisfies \ref{P iterated general}, \ref{P continuity n-meas} and \ref{P natural extension}.

\section{Game-Theoretic Upper Expectations}\label{Sect: Game-theoretic upper exp}

In this section, we show that a particular version of the game-theoretic upper expectation developed by Shafer and Vovk \cite{Shafer:2005wx,Vovk2019finance} belongs to \(\setofupprev{}_{1-5}(\tree{})\) and furthermore satisfies \ref{P iterated general}, \ref{P continuity n-meas} and \ref{P natural extension}, which will then---because of Theorem~\ref{Theorem: smallest vovk general}---imply the existence of the unique \(\upprev{\mathrm{A}}\) we are in pursuit of.
This game-theoretic upper expectation relies on the concept of a \emph{supermartingale},\footnote{Shafer and Vovk drew inspiration from Jean Ville, whose work \cite{ville1939etudecritique} constitutes the first major steps towards a theory of probability where (super)martingales, instead of probability measures, function as the fundamental, primary objects.} which is a capital process---the evolution of a subject's capital---that is obtained by betting against a `forecasting system'. 
The forecasting system---called `Forecaster' in Shafer and Vovk's framework---determines for each situation a collection of allowable bets that a subject---called `Skeptic' in Shafer and Vovk's framework---can choose from.
These allowable bets define Forecaster's uncertainty model and are also captured by our notion of an imprecise probability tree.
We should point out though that Shafer and Vovk also allow for more general settings where local models need not be boundedly coherent, where state spaces can be infinite and where Forecaster his commitments---that is, the allowable bets for Skeptic---in each situation need not be announced beforehand, but may also depend on previous moves by Skeptic. 
In that respect, our discussion here will be more particular and limited.

As before, we start from an imprecise probability tree that specifies a local coherent upper expectation \(\lupprev{s}\) on \(\setofgengambles{}(\statespace{})\) for every situation \(s \in \situations{}\).
However, in their most recent work~\cite[Part~II]{Vovk2019finance}, Shafer and Vovk start from local models $\upprev{s}$ that are defined on the extended domain $\setofgenextvariables{}(\statespace{})$ of all extended real local variables, and they require these local models $\upprev{s}$ to satisfy a modified version of the bounded coherence axioms \ref{sep coherence 1}--\ref{sep coherence 3}, generalised to extended real variables.
Concretely, for any $s\in\situations$, they impose the following properties on the upper expectation $\upprev{s}$ on $\setofgenextvariables{}(\statespace{})$ that models the corresponding local beliefs:
\begin{enumerate}[leftmargin=*,ref={\upshape E\arabic*},label={\upshape E\arabic*}.,series=vovklocalcoherence]
\item \label{vovk local coherence: const is const} 
$\upprev{s}(c)=c$ for all $c\in\reals{}$; 
\item \label{vovk local coherence: sublinearity} 
$\upprev{s}(f+g)\leq\upprev{s}(f) + \upprev{s}(g)$ for all $f,g\in\setofgenextvariables(\statespace)$;
\item \label{vovk local coherence: homog for positive lambda}
$\upprev{s}(\lambda f)=\lambda \upprev{s}(f)$ for all $\lambda\in\posreals$ and all $f\in\setofgenextvariables(\statespace)$;
\item \label{vovk local coherence: monotonicity}
$f \leq g \, \Rightarrow \, \upprev{s}(f)\leq\upprev{s}(g)$ for all $f, g\in\setofgenextvariables(\statespace)$.
\item \label{vovk local coherence: continuity wrt non-negative functions}
$\lim_{n\to+\infty} \upprev{s}(f_n)=\upprev{s}\left( \lim_{n\to+\infty} f_n \right)$ for any non-decreasing and non-negative sequence $\{f_n\}_{n\in\natz}$ in $\setofgenextvariables(\statespace)$.
\end{enumerate}
Hence, in order to place ourselves squarely in their framework, we need to extend the domain of our local models $\lupprev{s}$ from $\setofgengambles(\statespace)$ to $\setofgenextvariables(\statespace)$ in such a way that they satisfy \ref{vovk local coherence: const is const}--\ref{vovk local coherence: continuity wrt non-negative functions}.
For any $f\in\setofgenextvariables{}(\statespace{})$ and any $c\in\reals$, let $f^{\wedge c}$ and $f^{\vee c}$ be the variables in $\setofgenextvariables{}(\statespace{})$ respectively defined by $f^{\wedge c}(x) \coloneqq \min\{f(x),c\}$ and $f^{\vee c}(x) \coloneqq \max\{f(x),c\}$ for all $x\in\statespace$.
We propose the following two continuity properties for an upper expectation $\extlupprev{s}$ on $\setofgenextvariables(\statespace)$ that extends a coherent upper expectation $\smash{\lupprev{s}}$ on $\setofgengambles(\statespace)$: 
\begin{enumerate}[leftmargin=*,ref={\upshape{C\arabic*}},label={\upshape{}C\arabic*}.,itemsep=3pt, resume=sepcoherence ]
\item \label{upper cuts} \(\extlupprev{s}(f) = \lim_{c\to+\infty} \extlupprev{s}(f^{\wedge c})\) for all \(f\in\setofgenextvariablesb(\statespace{})\);
\item \label{lower cuts} \(\extlupprev{s}(f)=\lim_{c\to-\infty} \extlupprev{s}(f^{\vee c})\) for all \(f\in\setofgenextvariables(\statespace{})\);
\end{enumerate}
Properties \ref{upper cuts} and \ref{lower cuts} are respectively called \emph{continuity with respect to upper} and \emph{lower cuts} (or, alternatively, bounded above and below support).
According to \cite[Proposition~10~(ii)--(iii)]{Tjoens2020FoundationsARXIV} such an extension $\extlupprev{s}$ always exists and is moreover unique.
The upper expectation $\smash{\extlupprev{s}}$ is also an appropriate extension, in the sense that it satisfies \ref{vovk local coherence: const is const}--\ref{vovk local coherence: continuity wrt non-negative functions}; this follows from \cite[Proposition~10]{Tjoens2020FoundationsARXIV} together with \cite[Proposition~1]{Tjoens2020FoundationsARXIV}.\footnote{To see this, note that the definition of an upper expectation in Reference \cite{Tjoens2020FoundationsARXIV} is different from the one we adopt here.}
Hence, in order to adhere to Shafer and Vovk's approach, we can---and will---adopt this particular extension $\extlupprev{s}$ of $\lupprev{s}$ for each situation $s \in\situations{}$ in the following game-theoretic reasoning.

A map \(\martingale{}\colon\situations{}\to\extreals{}\) is called a \emph{supermartingale} if it satisfies \(\smash{\extlupprev{s}(\martingale{}(s \cdot))\leq\martingale{}(s)}\) for all \(s \in \situations{}\), where \(\martingale{}(s \cdot)\) 
denotes the variable in \(\setofgenextvariables{}(\statespace{})\) that assumes the value \(\martingale{}(s x)\) for all \(x \in \statespace{}\).
As mentioned earlier, a supermartingale \(\martingale{}\) can be interpreted as a capital process that results from betting against a forecasting system. 
In order for this to make sense, such an interpretation will only be attached to supermartingales that are \emph{bounded below}, meaning that there is some real \(c\) such that \(\martingale{}(s) \geq c\) for all \(s \in \situations{}\).
This represents the constraint that we can never borrow an infinite or even unbounded amount of money.
To see that a bounded below supermartingale $\martingale{}$ can indeed be interpreted as described above, we first need to know how a forecasting system allows us to play.
In our case, the allowable bets in any given situation \(s\in\situations{}\) are those variables \(f\in\setofgenextvariablesb{}(\statespace{})\) for which \(\extlupprev{s}(f) \leq 0\)---we will come back to this shortly.
Now, if \(\martingale{}(s)\) is real, then the condition that \(\extlupprev{s}(\martingale{}(s \cdot)) \leq \martingale{}(s)\) is equivalent to the condition that \(\smash{\extlupprev{s}(\martingale{}(s \cdot) - \martingale{}(s)) \leq 0}\) because \(\smash{\extlupprev{s}}\) satisfies constant additivity on the domain \(\setofgenextvariablesb{}(\statespace{})\); 
see \cite[Proposition~3(E8)]{Tjoens2020FoundationsARXIV}.
Hence, the incremental change \(\martingale{}(s \cdot) - \martingale{}(s)\) of the capital process \(\martingale{}\) in the situation~\(s\) indeed represents an allowable bet.
If on the other hand \(\martingale{}(s) = +\infty\) (the case where \(\martingale{}(s) = -\infty\) is impossible because we assumed \(\martingale{}\) to be bounded below), the condition that \(\extlupprev{s}(\martingale{}(s \cdot)) \leq \martingale{}(s)\) does not impose any constraints on \(\martingale{}(s \cdot)\).
The interpretation that the increment \(\martingale{}(s \cdot)-\martingale{}(s)\) represents an allowable bet is, in this case, somewhat questionable.
However, one could argue that, since the process $\martingale{}$ has reached the maximum possible value---that is, $+\infty$---in the situation $s$, it cannot increase any further, hence there is no need for any constraints on the process' future value.

That the allowable bets \(f\) in some situation \(s \in \situations{}\) are characterised by the condition that \(\extlupprev{s}(f) \leq 0\), agrees with the behavioural interpretation of (boundedly) coherent upper expectations.
To see this, we limit ourselves to the domain \(\setofgengambles{}(\statespace{})\) and recall that \(\extlupprev{s}(f)=\lupprev{s}(f)\), for any \(f \in \setofgengambles{}(\statespace{})\), can then be interpreted as a subject's infimum selling price for the gamble \(f\).
So the subject, being the forecasting system in this case, is willing to offer the gambles that have negative\footnote{In the interest of brevity, we ignore in our discussion of the interpretation the special case where we have zero upper expectation. See \cite{deCooman:2008km} for more details about this special case.} upper expectation.
We can then take the subject up on his commitments by accepting its offer and selecting one such gamble.
Moreover, since the upper expectations \(\lupprev{s}\) are assumed to be (boundedly) coherent, the associated sets of available gambles will also satisfy certain rationality axioms.
We refer to an earlier paper by one of us \cite{deCooman:2008km} for a more elaborate discussion of how Walley's behavioural approach can be related (in a slightly different context) to the game-theoretic one proposed by Shafer and Vovk.

Our definition of the global game-theoretic upper expectation---and the one considered by Shafer and Vovk---will be based on bounded below supermartingales, not only because it allows for a sensible interpretation as we have explained above, but also because of more technical reasons, which we discuss in \cite[Section~8]{Tjoens2020FoundationsARXIV}.
So let \(\setofextsupmartb{}\) be the set of all such bounded below supermartingales.
For any \(\martingale{} \in \setofextsupmartb{}\), we then use \(\liminf \martingale{}\) to denote the extended real global variable that assumes the value \(\liminf_{n \to +\infty} \martingale{}(\omega^n)\) in each \(\omega \in \Omega\).
We also write, for any situation $s\in\situations{}$ and any two \(f\) and \(g\) in \(\setofextvariables{}\), that \(f \geq_s g\) if \(f(\omega) \geq g(\omega)\) for all paths \(\omega \in \Gamma(s)\).
The global game-theoretic upper expectation \(\upprevvovkk{} \colon \setofextvariables{} \times \situations{} \to \extreals{}\) can now be defined as
\begin{align*}
\upprevvovkk{}(f \vert s) \coloneqq \inf \Big\{ \martingale{}(s) \colon \martingale{} \in \setofextsupmartb{} \text{ and } \liminf \martingale{} \geq_s f \Big\} \text{ for all } f \in \setofextvariables{} \text{ and all } s \in \situations{}.
\end{align*}
The corresponding game-theoretic lower expectation $\lowprevvovkk{}$ is again defined by conjugacy; so $\lowprevvovkk{}(f\vert s) \coloneqq - \upprevvovkk{}(- f \vert s)$ for all $f\in\setofextvariables{}$ and all $s\in\situations{}$.
For a concrete and detailed motivation and interpretation for the operator $\upprevvovkk{}$, we refer to \cite{Shafer:2005wx,Vovk2019finance,Tjoens2020FoundationsARXIV}.
Essentially, the upper expectation \(\upprevvovkk{}(f \vert x_{1:n})\) can once more be seen as an infimum selling price for a variable \(f \in \setofextvariables{}\) conditional on the fact that we have observed some history \(X_1 = x_1, X_2 = x_2, \cdots, X_n = x_n\) of the process.
It is based on the assumption that we should agree to sell $f$ for any starting capital \(\martingale{}(x_{1:n})\) such that, by appropriately betting against the forecasting system from \(x_{1:n}\) onwards, we manage to end up with a higher capital than what we would receive from \(f\), regardless of the path \(\omega \in \Omega\) taken by the process.
Indeed, it is reasonable to state that such a starting capital \(\martingale{}(x_{1:n})\) is worth more than the uncertain (possibly negative) payoff corresponding to \(f\).
The global game-theoretic upper expectation \(\upprevvovkk{}(f \vert x_{1:n})\) is the infimum over all such starting capitals.

As we will briefly discuss in Section~\ref{Sect: discussion}, the global game-theoretic upper expectation \(\upprevvovkk{}\) satisfies various, remarkably powerful properties.
In particular, it satisfies \ref{P compatibility}--\ref{P continuity}, \ref{P iterated general}, \ref{P continuity n-meas} and \ref{P natural extension}, which allows us to state our main result: that $\upprevvovkk{}$ is the unique most conservative upper expectation \(\upprev{\mathrm{A}}\) in \(\setofupprev{}_{1-5}(\tree{})\).
Our proofs for these results are heavily based on our work in \cite{Tjoens2020FoundationsARXIV}.
The setting there is entirely the same as the one described in this section, with the exception that we generally do not impose \ref{lower cuts} on the local models in \cite{Tjoens2020FoundationsARXIV}.
It will soon be clarified why we have made this distinction; for the moment, it suffices to see that the setting in \cite{Tjoens2020FoundationsARXIV} is more general and therefore, that the results in \cite{Tjoens2020FoundationsARXIV} obviously remain valid here.

\begin{proposition}\label{Prop: vovk satisfies P1-P7}
\(\upprevvovkk{}\) is an element of \/ \(\setofupprev{}_{1-5}(\tree{})\) and furthermore satisfies \ref{P iterated general}, \ref{P continuity n-meas} and \ref{P natural extension}.
\end{proposition}
\begin{proof}
We prove that \(\upprevvovkk{}\) satisfies \ref{P compatibility}--\ref{P continuity}, \ref{P iterated general}, \ref{P continuity n-meas} and \ref{P natural extension}.
Axiom \ref{P compatibility} follows immediately from \cite[Proposition~14]{Tjoens2020FoundationsARXIV} which, in particular, says that \(\upprev{} (f \vert \sit{}) = \lupprev{\sit{}}(f(\sit{} \cdot))\) for any situation \(\sit{} \in \situations\) and any \((n+1)\)-measurable gamble \(f\).
Indeed, for any \(h \in \setofgengambles{}(\statespace{})\), \(h(X_{n+1})\) is an \mbox{\((n+1)\)}-measurable gamble, so it follows from \cite[Proposition~14]{Tjoens2020FoundationsARXIV} that \(\smash{\upprevvovkk{} (h(X_{n+1}) \vert \sit{}) = \lupprev{\sit{}}(h)}\).
To prove \ref{P conditional}, observe that \(\liminf \martingale{} \geq_s f\) if and only if \(\liminf \martingale{} \geq_s f\indica{s}\) for all \(f \in \smash{\setofextvariables{}}\) and all \(\martingale{} \in \smash{\setofextsupmartb{}}\).
The desired equality therefore follows directly from the definition of \(\upprevvovkk{}\).
\ref{P iterated general}, and hence also \ref{P iterated}, follows immediately from \cite[Theorem~15]{Tjoens2020FoundationsARXIV}.
Property~\ref{P monotonicity} follows from \cite[Proposition~13(V4)]{Tjoens2020FoundationsARXIV}, which says that $\upprevvovkk{}$ satisfies the stronger property that \(\upprev{}(f \vert s) \leq \upprev{}(g \vert s)\) if $f \leq_s g$ for all \(f,g \in \setofextvariables{}\) and all \(s \in \situations{}\).

Furthermore, \ref{P continuity n-meas} and \ref{P natural extension} follow from, respectively, \cite[Theorem~32]{Tjoens2020FoundationsARXIV} and \cite[Proposition~28]{Tjoens2020FoundationsARXIV}.
Note that, in these results, the definition of \(\setoflimitsoffinmeasextb{}\) seems to be slightly more general, since it there also includes pointwise limits of possibly \emph{extended} real finitary variables. 
However, according to \cite[Proposition~30]{Tjoens2020FoundationsARXIV}, any such limit of extended real finitary variables is in fact also a pointwise limit of $n$-measurable, and therefore finitary, gambles.
So both our definitions of \(\setoflimitsoffinmeasextb{}\) are equivalent, therefore indeed allowing us to apply the results in \cite{Tjoens2020FoundationsARXIV} here.

Finally, Property \ref{P continuity} follows from \cite[Corollary~25]{Tjoens2020FoundationsARXIV}, which says that, for any sequence \(\{f_n\}_{n \in \natz{}}\) in \(\setofextvariablesb{}\) that is uniformly bounded below,
\(\smash{\upprevvovkk{}(\liminf_{n \to +\infty} f_n \vert s) \leq \liminf_{n \to +\infty} \upprevvovkk{}(f_n \vert s)}\).
Indeed, in the special case that \(\{f_n\}_{n \in \natz{}}\) is a sequence of finitary gambles that is uniformly bounded below and converges pointwise to some variable \(f \in \setofextvariablesb{}\), this implies that
\(\upprevvovkk{}(f \vert s) \leq \liminf_{n \to +\infty} \upprevvovkk{}(f_n \vert s) \leq \limsup_{n \to +\infty} \upprevvovkk{}(f_n \vert s)\).
\end{proof}

\begin{theorem}\label{theorem: Vovk is the largest}
There is a unique most conservative upper expectation \(\upprev{\mathrm{A}}\) in \(\setofupprev{}_{1-5}(\tree{})\) and it is the global game-theoretic upper expectation \(\upprevvovkk{}\).
Furthermore, if an upper expectation $\upprev{}$ in \/ \(\setofupprev{}_{1-5}(\tree{})\) satisfies \ref{P iterated general}, \ref{P continuity n-meas} and \ref{P natural extension}, then it is equal to this most conservative upper expectation $\upprev{\mathrm{A}}$, and therefore also equal to $\upprevvovkk$.
\end{theorem}
\begin{proof}
Immediately from Theorem~\ref{Theorem: smallest vovk general} and Proposition~\ref{Prop: vovk satisfies P1-P7}.
\end{proof}


\section{On the Implications of Theorem~\ref*{theorem: Vovk is the largest}}\label{Sect: discussion}

Shafer and Vovk define their game-theoretic upper expectations using supermartingales.
Their definition has a constructive flavour and can be given a clear interpretation in terms of capital processes and betting behaviour.
However, it requires that one allows unbounded and even infinite-valued bets, which we find more questionable from a behavioural point of view.
Furthermore, a complete axiomatisation of \(\upprevvovkk{}\) is, to the best of our knowledge, absent from the literature.
Theorem~\ref{theorem: Vovk is the largest} addresses both of these issues.
First, it provides an abstract axiomatisation for \(\upprevvovkk{}\) using \ref{P iterated general}, \ref{P continuity n-meas} and \ref{P natural extension} in addition to \ref{P compatibility}--\ref{P continuity}.
Most importantly, however, Theorem~\ref{theorem: Vovk is the largest} provides an alternative definition--- and interpretation---for \(\upprevvovkk{}\) as the most conservative upper expectation \(\upprev{\mathrm{A}}\) under a limited set of intuitive properties: \ref{P compatibility}--\ref{P continuity}.
This strengthens the argument in favour of using \(\upprevvovkk{}\) as a global upper expectation, because it can now be motivated from both a game-theoretic point of view and from a purely axiomatic point of view.
Readers need not even be familiar with the concepts of game-theoretic probability in order to use \(\upprevvovkk{}\) as global upper expectation.
Put simply, they only have to agree on the axioms \ref{P compatibility}--\ref{P continuity}.
And even if they wish to impose additional axioms, then \(\upprevvovkk{}\) would still serve as a conservative upper bound for their desired (upper) expectation.

Still, recalling our assumptions about the extended local models $\smash{\extlupprev{s}}$, one could regret the fact that we are only considering a special case of what is generally being done by Shafer and Vovk; an upper expectation $\smash{\extlupprev{s}}$ on $\setofgenextvariables{}(\statespace{})$ satisfying \ref{vovk local coherence: const is const}--\ref{vovk local coherence: continuity wrt non-negative functions} does not necessarily have to be an extension of a coherent upper expectation $\lupprev{s}$ on $\setofgengambles{}(\statespace{})$ satisfying \ref{upper cuts} and \ref{lower cuts}, see \cite[Example~1]{Tjoens2020FoundationsARXIV}.
However, this restriction in generality actually turns out to be rather desirable.
Indeed, as we show in \cite{Tjoens2020FoundationsARXIV}, compatibility of local and global game-theoretic upper expectations---a property that we find essential---can only be obtained if we limit ourselves to extended local models $\extlupprev{s}$ of the form that we have described earlier, that is, boundedly coherent and satisfying \ref{upper cuts} and \ref{lower cuts}.

That said, as mentioned before, we do not a priori impose \ref{lower cuts} on the local models in \cite{Tjoens2020FoundationsARXIV}.
In that paper, we mainly focus on the mathematical properties of global game-theoretic upper expectations and most of these do not require \ref{lower cuts}.
Furthermore, by dropping this axiom, we can allow for more general local models than Shafer and Vovk, who impose \ref{vovk local coherence: const is const}--\ref{vovk local coherence: continuity wrt non-negative functions}.
Nevertheless, in this less mathematical and more philosophical contribution, we do impose \ref{lower cuts} on the local models because, practically speaking, we always want local and global belief models to be compatible; see for example Section~\ref{Sect: Search of a global model}, where we considered it self-evident that \ref{P compatibility} should be desirable.
Note however that Theorem~\ref{theorem: Vovk is the largest} actually does not depend on whether the extended local models $\smash{\extlupprev{s}}$ satisfy \ref{lower cuts}.
On the one hand, because $\upprev{\mathrm{A}}$ only depends on the (boundedly) coherent local models $\lupprev{s}$, and on the other hand, because $\upprevvovkk{}$ only depends on what the values of the extended local models $\extlupprev{s}$ are on the restricted domain $\setofgenextvariablesb{}(\statespace{})$.
The latter can easily be seen from the fact that $\upprevvovkk{}$ is defined using supermartingales that are bounded below; we refer to \cite{Tjoens2020FoundationsARXIV} for more details.

\section{Additional Properties}\label{Sect: additional properties}

Of course, there is more to an uncertainty model than an axiomatisation or an interpretation, however compelling they may be. 
In order to be practically useful, its mathematical properties should also be sufficiently powerful such that inferences can be computed or approximated efficiently.
For instance, the popularity of the Lebesgue integral as a tool for defining expected values, is in part due to its strong continuity properties (e.g. the Dominated Convergence Theorem \cite{shiryaev2016probabilityPartI,royden2010real}).
What may perhaps be somewhat surprising, is that despite the simplicity of our axioms \ref{P compatibility}--\ref{P continuity}, our most conservative model \(\upprev{\mathrm{A}}\) also scores well on this count. 
For example, it satisfies the following generalisation of bounded coherence.
Its proof, as well as the proofs of the other results in this section, follows immediately from the fact that $\upprev{\mathrm{A}}$ coincides with $\upprevvovkk{}$ and the fact that these properties hold for $\upprevvovkk{}$, as is shown in \cite{Tjoens2020FoundationsARXIV}.
\begin{proposition}\label{Prop: coherence for upprev_A}
{\normalfont \cite[Proposition 13 and Theorem 15]{Tjoens2020FoundationsARXIV}} \,
For all \(f,g\in\setofextvariables\), all \(\lambda\in\nnegreals{}\), all \(\mu\in\reals{}\), all \(n \in \natz{}\) and all \(s\in\situations\), \(\upprev{\mathrm{A}}\) satisfies
\begin{enumerate}[leftmargin=*,ref={\upshape V}\arabic*,label={\upshape V\arabic*.},itemsep=3pt, series=sepcoherence ]
\item \label{vovk coherence 1} 
\(\inf_{\omega \in \Gamma(s)} f(\omega) \leq \lowprev{}_{\hspace*{1pt}\raisebox{0pt}{\scriptsize $\mathrm{A}$}}(f \vert s) \leq\upprev{\mathrm{A}}(f \vert s)\leq\sup_{\omega \in \Gamma(s)} f(\omega)\); 
\item \label{vovk coherence 2} 
\(\upprev{\mathrm{A}}(f+g \vert s)\leq\upprev{\mathrm{A}}(f \vert s) + \upprev{\mathrm{A}}(g \vert s)\);
\item \label{vovk coherence 3}
\(\upprev{\mathrm{A}}(\lambda f \vert s)=\lambda \upprev{\mathrm{A}}(f \vert s)\);
\item \label{vovk coherence 6}
\(\upprev{\mathrm{A}}(f + \mu \vert s)=\upprev{\mathrm{A}}(f \vert s) + \mu\).
\item \label{vovk coherence iterated}
\(\upprev{\mathrm{A}}(f \vert X_{1:n}) = \upprev{\mathrm{A}} \bigl(\upprev{\mathrm{A}}(f \vert X_{1:n+1}) \big\vert X_{1:n}\bigr)\).
\end{enumerate}
\end{proposition}

Another important result, better known as Fatou's Lemma, shows that the upper bound imposed by property \ref{P continuity} for a global upper expectation \(\upprev{}(f \vert s)\) of a variable \(f \in \setoflimitsoffinmeasextb{}\) need not be attained by the most conservative upper expectation \(\upprev{\mathrm{A}}\) in \(\setofupprev{}_{1-5}(\tree{})\).
Interestingly enough, it can be replaced by a generally tighter one, which is established in the following lemma.

\begin{lemma}\label{lemma: Fatou general}
{\normalfont \cite[Corollary~25]{Tjoens2020FoundationsARXIV}} \,
For any \(s\in\situations\) and any sequence \(\{f_n\}_{n\in\natz}\) in \(\setofextvariablesb\) that is uniformly bounded below, 
\(\upprev{\mathrm{A}}(\liminf_{n\to+\infty} f_n \vert s)\leq\liminf_{n\to+\infty} \upprev{\mathrm{A}}(f_n \vert s)\).
\end{lemma}

The following theorem states that \(\upprev{\mathrm{A}}\) satisfies continuity with respect to non-decreasing sequences of bounded below variables and with respect to non-increasing sequences of finitary, bounded above variables.

\begin{theorem}\label{theorem: upward convergence game}
{\normalfont \cite[Theorem~23 and Theorem~31]{Tjoens2020FoundationsARXIV}}
\begin{enumerate}[leftmargin=*,ref=\upshape{\roman*},label={\upshape{\roman*.}},itemsep=3pt]
\item\label{theorem: upward convergence game 1}
For any \(s\in\situations\) and any non-decreasing sequence \(\{f_n\}_{n\in\natz}\) in \(\setofextvariablesb\) that converges point-wise to a variable \(f\in\setofextvariablesb\), we have that
$\upprev{\mathrm{A}}(f \vert s)=\lim_{n\to+\infty} \upprev{\mathrm{A}}(f_n \vert s)$.

\item\label{theorem: upward convergence game 2}
For any \(s\in\situations\) and any non-increasing sequence \(\{f_n\}_{n\in\natz}\) of finitary, bounded above variables that converges point-wise to a variable \(f\in\setofextvariables\), we have that
$\upprev{\mathrm{A}}(f \vert s)=\lim_{n\to+\infty} \upprev{\mathrm{A}}(f_n \vert s)$.
\end{enumerate}
\end{theorem} 

Apart from the brief overview above, \(\upprevvovkk{} = \upprev{\mathrm{A}}\) also satisfies a weak and a strong law of large numbers, a law of the iterated logarithm, L\'evy's zero-one law, and many more surprisingly strong properties; we refer the interested reader to Shafer and Vovk's work \cite{shafer2012,Vovk2019finance} and to ours \cite{TJoens2018Continuity,Tjoens2020FoundationsARXIV}.

\section{Measure-Theoretic Belief Models}\label{sect: relation to measure theory}

In contradistinction with the game-theoretic approach, where the central concept is that of a supermartingale, the measure-theoretic approach relies  on \(\sigma\)-additive probability measures in order to describe the global dynamics of an uncertain process.
In a precise probabilities context, this global measure is usually obtained by extending the information that is included in the (precise) local models, now represented by probability mass functions on \(\statespace{}\), using Ionescu Tulcea's extension theorem \cite[Theorem 2.9.2]{shiryaev2016probabilityPartI}.
This consists of two consecutive steps.
First, a probability measure on the algebra generated by all cylinder events is constructed in a straightforward manner, by combining local (conditional) probabilities using the well-known product rule.
Subsequently, this probability measure is extended to a probability measure on the \(\sigma\)-algebra \(\mathscr{F}\) generated by all cylinder events, which, according to Carathéodory's extension theorem \cite[Theorem 1.7]{williams1991probability}, can be done uniquely.
Linear expectations are then typically obtained by Lebesgue integration and by restricting the domain to (random) variables that are measurable with respect to the \(\sigma\)-algebra \(\mathscr{F}\).

This traditional measure-theoretic approach has a few shortcomings however.
First of all, it only considers the special case where each local model is represented by a single probability mass function on \(\statespace{}\).
Hence, we cannot directly apply it here in our more general, imprecise probabilities setting.
Secondly, the domain of the expectation operators is restricted to measurable (random) variables and therefore, no explicit information is given about non-measurable variables.
Thirdly, by condensing all local probability mass functions into a single global probability measure, we disregard part of the information that is incorporated in the local probability mass functions.
More specifically, the part that describes the process' dynamics conditional on the fact that its path will lie in a set of probability zero (with respect to the global probability measure).
Indeed, conditional probabilities and expectations are typically derived from the global probability measure using Bayes' rule, yet, for those cases where the conditioning event has probability zero, the conditional probabilities and expectations become ill-defined or are chosen arbitrarily; see \ref{subsect: Basic measure-theoretic concepts} and \cite{shiryaev2016probabilityPartI,billingsley1995probability}.

In this section, we aim to adapt the traditional measure-theoretic approach to our setting and deal with these problems in a satisfactory manner.
As a first step, we will limit ourselves to the precise case and try to resolve the last two issues mentioned above.
The global probability measure on \(\mathscr{F}\), which is the starting point for the traditional approach, will be replaced by a `conditional probability measure' in the sense of \cite[Definition~6]{8535240}.
Essentially, such a conditional measure specifies a probability measure on \(\mathscr{F}\) for each situation \(s \in \situations{}\), which describes the global dynamics of the uncertain process if we are sure that the path \(\omega\) taken by the process will pass through this situation~\(s\).
This will enable us to also meaningfully define linear expectations conditional on situations that have probability zero.
Subsequently, we will extend the domain of these conditional linear expectation operators to all global variables \(\setofextvariables{}\) by taking a particular upper integral leading to an upper expectation.
As a final step, we allow for imprecision in the local models by considering a set \(\mathbb{P}_s\) of probability mass functions on \(\statespace{}\) in each situation~\(s\).
We can then apply the foregoing course of reasoning for each possible selection \(p \colon s \in \situations{} \mapsto p(\cdot \vert s)\) of local probability mass functions, where, in each situation~\(s\), the selected probability mass function \(p(\cdot \vert s)\) belongs to \(\mathbb{P}_s\).
Hence, instead of a single upper expectation operator, we now obtain a set of them.
The global measure-theoretic upper expectation \(\upprev{\mathrm{meas}}\) will then be defined as the upper envelope of all such `compatible' upper expectation operators, similar to what is done on a local level.

The resulting operator \(\upprev{\mathrm{meas}}\) will satisfy \ref{P compatibility}--\ref{P continuity}, which implies by Theorem~\ref{theorem: Vovk is the largest} that the upper expectation \(\upprev{\mathrm{A}}\)---and therefore also \(\upprevvovkk{}\)---is an upper bound for \(\upprev{\mathrm{meas}}\).
Moreover, if all local models are precise, all upper expectation operators \(\upprev{\mathrm{A}}\), \(\upprevvovkk{}\) and \(\upprev{\mathrm{meas}}\) will be shown to coincide.

\subsection*{Precise Probability Trees}\label{sect: Measure-theoretic expectations in precise probability trees}

In this subsection, we consider a \emph{precise probability tree} $p$, that is, a map \(p \colon s \in \situations{} \mapsto p(\cdot \vert s)\) where each $p(\cdot \vert s)$ is a probability mass function on $\statespace{}$.
Recalling our considerations in Section~\ref{Sect: upprev}, each mass function $p(\cdot \vert s)$ is in a one-to-one correspondence with a linear expectation \(\lupprev{s}\) on \(\setofgengambles{}(\statespace{})\), so we could just as well have started from an imprecise probability tree $\lupprev{}$ where all local models \(\lupprev{s}\) are linear on \(\setofgengambles{}(\statespace{})\).
However, we prefer to start with $p$ because it is in accordance with the traditional measure-theoretic view towards uncertainty, where probability measures or charges---and therefore also probability mass functions---are considered to be primary objects.

As mentioned above, we aim to construct a conditional probability measure that allows us to condition on any situation \(s \in \situations{}\) meaningfully.
To do so, we follow the same reasoning as proposed in \cite[Chapter 3]{8535240}.
First, we define, for any \(x_{1:n} \in \situations{}\), a probability charge \(\mathrm{P}(\cdot \vert x_{1:n})\) on the algebra generated by all cylinder events.
This can be done in a simple and intuitive manner \cite[Lemma 14]{8535240}: for any \(m \in \natz{}\) and any \(C \subseteq \statespace{}^m\), we let
\begin{align}\label{Eq: precise probability on algebra}
\mathrm{P}(C \vert x_{1:n}) \coloneqq \sum_{z_{1:m} \in C} \mathrm{P}(z_{1:m} \vert x_{1:n}),
\text{ where } \mathrm{P}(z_{1:m} \vert x_{1:n}) \coloneqq 
\begin{aligned}
\begin{cases}
\prod_{i=n}^{m-1} p( z_{i+1} \vert z_{1:i}) &\text{ if } n < m \text{ and } z_{1:n} = x_{1:n} \\
1 &\text{ if } n \geq m \text{ and } z_{1:m} = x_{1:m} \\
0 &\text{ otherwise, }
\end{cases}
\end{aligned}
\end{align}
where we use the shorthand notation \(\mathrm{P}(C \vert x_{1:n})\) for any \(C \subseteq \statespace{}^m\) to mean \(\mathrm{P}(\cup_{z_{1:m} \in C} \Gamma(z_{1:m}) \vert x_{1:n})\), and where we did not make any particular distinction between the notations \(z_{1:m}\) and \(\{z_{1:m}\}\).
Since each \(\mathrm{P}(\cdot \vert x_{1:n})\) defined in this way is a (finitely additive) probability charge on the algebra generated by all cylinder events \cite[Chapter 3]{8535240}, and therefore automatically a \(\sigma\)-addive probability measure on this algebra \cite[Theorem 2.3]{billingsley1995probability}, we can use Carathéodory's extension theorem \cite[Theorem 1.7]{williams1991probability} to extend each \(\mathrm{P}(\cdot \vert x_{1:n})\) to a unique \(\sigma\)-additive measure on the \(\sigma\)-algebra \(\mathscr{F}\) generated by all cylinder events.
Then, according to \cite[Theorem~15]{8535240}, we can gather all such probability measures \(\mathrm{P}(\cdot \vert x_{1:n})\) to obtain a single \emph{conditional probability measure} \(\mathrm{P} \colon \mathscr{F} \times \situations{} \to \reals{}\) \cite[Definition~6]{8535240}, which is \emph{coherent} in the sense of \cite[Definition 5]{8535240}.
This notion of coherence is related to the bounded coherence condition in Section~\ref{Sect: upprev}; both notions represent rational behaviour, but, the version used in \cite[Definition 5]{8535240} is adapted to probability measures and involves conditioning on events.

Before defining our global measure-theoretic (upper) expectation, we first recall the following concepts about measurability and integration; we refer to \ref{subsect: Basic measure-theoretic concepts} for a more elaborate introduction.
We say that a global variable $f\in\setofextvariables{}$ is $\mathscr{F}$-measurable if, for all $c\in\reals$, the set $\{\omega\in\Omega\colon f(\omega)\leq c \}$ is $\mathscr{F}$-measurable, meaning that it is an element of the $\sigma$-algebra $\mathscr{F}$.
We will commonly use the properties that $\mathscr{F}$-measurability of variables is preserved under taking maxima or minima (this follows immediately from the fact that the class of all $\mathscr{F}$-measurable sets is closed under countable unions and intersections \cite[Section~II.4]{shiryaev1995probability}) and that $\mathscr{F}$-measurability is preserved under pointwise convergence of variables \cite[Theorem II.4.2]{shiryaev1995probability}.
Moreover, recall that the Lebesgue integral $\int_\Omega f \dif\mathrm{P}_\mathrm{u}$ \cite{shiryaev2016probabilityPartI,billingsley1995probability,royden2010real} with respect to any (unconditional) probability measure $\mathrm{P}_\mathrm{u}$ on $\mathscr{F}$ always exists if $f$ is $\mathscr{F}$-measurable and either bounded or non-negative.
However, for a general $\mathscr{F}$-measurable variable $\smash{f\in\setofextvariables{}}$, the integral only exists if \(\min\{\int_{\Omega} f^{+} \dif \mathrm{P}_\mathrm{u},\int_{\Omega} f^{-} \dif \mathrm{P}_\mathrm{u}\} < +\infty\), where \(f^{+} \coloneqq f^{\vee 0}\) and \(f^{-}\coloneqq -f^{\wedge 0}\).
In that case, we let
$\smash{\int_{\Omega} f \dif \mathrm{P}_\mathrm{u} \coloneqq \int_{\Omega} f^{+} \dif \mathrm{P}_\mathrm{u} - \int_{\Omega}  f^{-} \dif \mathrm{P}_\mathrm{u}}$.
Then, because the Lebesgue integral is trivially real-valued for a bounded $\mathscr{F}$-measurable variable (this can for instance be deduced from \ref{prop measure: linearity} and \ref{prop measure: monotonicity} in \ref{subsect: Basic measure-theoretic concepts}), it can easily be seen that $\int_\Omega f \dif\mathrm{P}_\mathrm{u}$ also always exists for an $\mathscr{F}$-measurable variable $f\in\setofextvariablesb{}$.

We now adopt an approach that is similar to what is done in a traditional context, defining expectations as Lebesgue integrals.
An important difference, however, is that we integrate with respect to a probability measure that depends on the conditioning event.
Concretely, we define the \emph{measure-theoretic expectation} \(\smash{\mathrm{E}_{\mathrm{meas},p}(f \vert s)}\) of an \(\mathscr{F}\)-measurable variable \(f \in \setofextvariables{}\) conditional on the situation \(s \in \situations{}\) by the Lebesgue integral \(\smash{\int_\Omega f \dif \mathrm{P}_{\vert s}}\) with respect to the (unconditional) probability measure \(\mathrm{P}_{\vert s} \coloneqq \mathrm{P}(\cdot \vert s)\), provided that \(\int_\Omega f \dif \mathrm{P}_{\vert s}\) exists.
The expectation \(\smash{\mathrm{E}_{\mathrm{meas},p}( \cdot \vert \Box)}\) then corresponds to the unconditional expectation as introduced in traditional measure theory.
The subscript \(p\) in \(\smash{\mathrm{E}_{\mathrm{meas},p}}\) is used to remind us that \(\smash{\mathrm{E}_{\mathrm{meas},p}}\) depends on the precise probability tree \(p\).
Since \(\smash{\mathrm{E}_{\mathrm{meas},p}}\) is defined in terms of the Lebesgue integral, it inherits several properties of this integral; again, we refer to \ref{subsect: Basic measure-theoretic concepts} for more information.
The following is a subset of these properties---simplified and applied to \(\smash{\mathrm{E}_{\mathrm{meas},p}}\)---that we will need in the main text.
We take $s\in\situations{}$ to be any situation and implicitly assume that the considered expectations exist.  
\begin{enumerate}[leftmargin=21pt,itemsep=3pt]
\item[\ref{prop measure: linearity}.] 
\(\mathrm{E}_{\mathrm{meas},p}(a f + b \vert s) = a \mathrm{E}_{\mathrm{meas},p}(f\vert s) + b\) for all $\mathscr{F}$-measurable variables $f\in\setofextvariables$ and all \(a,b \in \reals{}\).
\item[\ref{prop measure: monotonicity}.] 
\(f \leq g\) \(\Rightarrow \mathrm{E}_{\mathrm{meas},p}(f \vert s) \leq  \mathrm{E}_{\mathrm{meas},p}(g \vert s)\) for all $\mathscr{F}$-measurable variables $f,g\in\setofextvariables$.
\item[\ref{lemma: limits for precise measure-theoretic expectation ii}.]
Consider any non-decreasing sequence \(\{f_n\}_{n \in \natz{}}\) of \/ $\mathscr{F}$-measurable variables in $\setofextvariables{}$.
If there is an \(\mathscr{F}\)-measurable function \(f^\ast\) such that \(\mathrm{E}_{\mathrm{meas},p}(f^\ast \vert s) > -\infty\) and \(f_n \geq f^\ast\) for all \(n \in \natz{}\), then
\begin{align*}
\lim_{n \to +\infty} \mathrm{E}_{\mathrm{meas},p}(f_n  \vert s) = \mathrm{E}_{\mathrm{meas},p}(f \vert s)
\text{ where } \lim_{n \to +\infty} f_n = f.
\end{align*} 

\item[\ref{lemma: limits for precise measure-theoretic expectation iii}.] 
Consider any non-increasing sequence \(\{f_n\}_{n \in \natz{}}\) of \/ $\mathscr{F}$-measurable variables in $\setofextvariables{}$.
If there is an \(\mathscr{F}\)-measurable function \(f^\ast\) such that \(\mathrm{E}_{\mathrm{meas},p}(f^\ast \vert s) < +\infty\) and \(f_n \leq f^\ast\) for all \(n \in \natz{}\), then
\begin{align*}
\lim_{n \to +\infty} \mathrm{E}_{\mathrm{meas},p}(f_n \vert s) = \mathrm{E}_{\mathrm{meas},p}(f \vert s)
\text{ where } \lim_{n \to +\infty} f_n = f.
\end{align*} 
\end{enumerate}

Now that our measure-theoretic definitions are in place, we aim at relating the expectation \(\mathrm{E}_{\mathrm{meas},p}\) to our upper expectation operator \(\smash{\upprev{\mathrm{A}}}\) discussed in the earlier sections of this paper.
For the sake of clarity, we will use \(\smash{\upprev{\mathrm{A},p}}\) to denote our upper expectation \(\smash{\upprev{\mathrm{A}}}\) corresponding to the imprecise probability tree $\lupprev{}$ associated with $p$; so each $\lupprev{s}$ is defined by $\smash{\lupprev{s}(f) \coloneqq \sum_{x\in\statespace{}}f(x)p(x\vert s)}$ for all $f\in\setofgengambles{}(\statespace{})$.
We first focus on the domain of all $\mathscr{F}$-measurable variables in $\setofextvariablesb{}$.
This seems appropriate because, though we cannot guarantee existence of the integral \(\smash{\int_\Omega f \dif \mathrm{P}_{\vert s}}\)---and therefore also \(\mathrm{E}_{\mathrm{meas},p}(f \vert s)\)---for a general $\mathscr{F}$-measurable variable $\smash{f\in\setofextvariables{}}$, it can always be guaranteed for an $\mathscr{F}$-measurable variable $\smash{f\in\setofextvariablesb{}}$, as we have explained earlier.
It will turn out that both operators \(\smash{\mathrm{E}_{\mathrm{meas},p}}\) and \(\smash{\upprev{\mathrm{A},p}}\) coincide on the domain of all \(\mathscr{F}\)-measurable variables in \(\setofextvariablesb{}\).

To prove this, we base ourselves on \cite[Theorem~9.3]{Vovk2019finance}; a result by Shafer and Vovk---yet the essential idea originates from Jean Ville---that shows that game-theoretic and measure-theoretic (upper) expectations coincide on $\mathscr{F}$-measurable gambles, in case we are considering precise probability trees.
Proposition~\ref{prop: unconditional precise measure is equal to game for bounded} below establishes the equality of $\smash{\upprev{\mathrm{A},p}}$ and \(\smash{\mathrm{E}_{\mathrm{meas},p}}\) on this domain and therefore, due to Theorem~\ref{theorem: Vovk is the largest}, also of $\upprevvovkk{}$ and \(\smash{\mathrm{E}_{\mathrm{meas},p}}\).
Though the proof of Proposition~\ref{prop: unconditional precise measure is equal to game for bounded} is based on the same principles and ideas as those that underly the proof of \cite[Theorem~9.3]{Vovk2019finance}, they are stated in a somewhat different setting and, for an unexperienced reader, it may therefore not be entirely clear how both are related to each other.
Moreover, our result here also applies to conditional expectations, whereas \cite[Theorem~9.3]{Vovk2019finance} only considers the unconditional case.
We have therefore decided to include a self-contained proof of Proposition~\ref{prop: unconditional precise measure is equal to game for bounded} in \ref{Sect: Proofs}, starting from Ville's martingale theorem.

\begin{proposition}\label{prop: unconditional precise measure is equal to game for bounded}
For any \(\mathscr{F}\)-measurable \(f \in \setofgambles{}\) and any $s\in\situations{}$, we have that \/ \(\mathrm{E}_{\mathrm{meas},p}(f \vert s) = \upprev{\mathrm{A},p}(f \vert s)\).
\end{proposition}

The next result extends the domain of the equality to all $\mathscr{F}$-measurable variables in \(\setofextvariablesb{}\).

\begin{theorem}\label{theorem: Ville extended variables bounded below}
For any \(\mathscr{F}\)-measurable \(f \in \setofextvariablesb{}\) and any \(s \in \situations{}\), we have that \/ \(\mathrm{E}_{\mathrm{meas},p}(f \vert s) = \upprev{\mathrm{A},p}(f \vert s)\).
\end{theorem} 
\begin{proof}
Consider any \(f \in \setofextvariablesb{}\) that is \(\mathscr{F}\)-measurable and any \(s \in \situations{}\).
Because \(f\) is bounded below and \(\mathscr{F}\)-measurable, \(\smash{\mathrm{E}_{\mathrm{meas},p}(f \vert s)}\) exists.
We can moreover assume that \(f\) is non-negative without loss of generality because it is bounded below and both \(\smash{\mathrm{E}_{\mathrm{meas},p}}\) and \(\upprev{\mathrm{A},p}\) are constant additive with respect to real constants; see \ref{prop measure: linearity} and \ref{vovk coherence 6}.
Consider now the non-decreasing sequence $\{f^{\wedge n}\}_{n\in\natz{}}$ of upper cuts and note that each $f^{\wedge n}$ is bounded and \(\mathscr{F}\)-measurable; this follows from the fact that \(\mathscr{F}\)-measurability is preserved under taking minima.
Using \ref{lemma: limits for precise measure-theoretic expectation ii}---which we are allowed to use because \(\mathrm{E}_{\mathrm{meas},p}(f^{\wedge 0} \vert s)\) is real (since $f^{\wedge 0}$ is bounded) and because \(\{f^{\wedge n}\}_{n \in \natz{}}\) is non-decreasing---we have that \(\mathrm{E}_{\mathrm{meas},p}(f \vert s) = \lim_{n \to +\infty} \mathrm{E}_{\mathrm{meas},p}(f^{\wedge n} \vert s)\).
As a consequence, 
\begin{equation*}
\mathrm{E}_{\mathrm{meas},p}(f \vert s) 
= \lim_{n \to +\infty} \mathrm{E}_{\mathrm{meas},p}(f^{\wedge n} \vert s)
= \lim_{n \to +\infty} \upprev{\mathrm{A},p}(f^{\wedge n} \vert s)
= \upprev{\mathrm{A},p}(f \vert s),
\end{equation*}
where the second equality follows from Proposition~\ref{prop: unconditional precise measure is equal to game for bounded} and the third one from Theorem~\ref{theorem: upward convergence game}(\ref{theorem: upward convergence game 1}).
\end{proof}

Next, we want to drop the condition of $\mathscr{F}$-measurability and generalise towards the domain $\setofextvariables{}$ of all global extended real variables.
So we are looking for an upper expectation operator that appropriately extends \(\mathrm{E}_{\mathrm{meas},p}\) to the entire domain $\smash{\setofextvariables{}\times\situations{}}$.
Inspired by Theorem~\ref{theorem: Ville extended variables bounded below}, we simply suggest the following upper integral:
\begin{align*}
\mupprevpreciseone(f \vert s) 
\coloneqq \inf \Bigl\{ \mathrm{E}_{\mathrm{meas},p}(g \vert s) \colon g \in \setofextvariablesb{} , \, g \text{ is \(\mathscr{F}\)-measurable and } g \geq f \Bigr\} \text{ for all } f \in \setofextvariables{} \text{ and } s \in \situations{}.
\end{align*}
That it is valid to write \(\mathrm{E}_{\mathrm{meas},p}(g \vert s)\) for any \(g \in \setofextvariablesb{}\) that is \(\mathscr{F}\)-measurable, follows from the fact that \(\mathrm{E}_{\mathrm{meas},p}(g \vert s) = \int_\Omega g \dif \mathrm{P}_{\vert s}\) is guaranteed to exist for such a variable.
Let us now show that $\smash{\mupprevpreciseone}$ is indeed an extension of \(\mathrm{E}_{\mathrm{meas},p}\).

\begin{proposition}\label{prop: alternative extension is an extension}
For any \(f \in \setofextvariables{}\) and any \(s \in \situations{}\), we have that \/ \(\mupprevpreciseone(f \vert s) = \mathrm{E}_{\mathrm{meas},p}(f \vert s)\) if \/ \(\mathrm{E}_{\mathrm{meas},p}(f \vert s)\) exists.
\end{proposition}
\begin{proof}
Suppose that \(\mathrm{E}_{\mathrm{meas},p}(f \vert s)\) exists, which means that the variable \(f\) is \(\mathscr{F}\)-measurable and that \(\mathrm{E}_{\mathrm{meas},p}(f^{+} \vert s)\) or \(\mathrm{E}_{\mathrm{meas},p}(f^{-} \vert s)\) is real-valued.
By the monotonicity [\ref{prop measure: monotonicity}] of $\mathrm{E}_{\mathrm{meas},p}$, we immediately have that
\begin{align}\label{Eq: prop: alternative extension is an extension}
\mupprevpreciseone(f \vert s) 
&= \inf \Bigl\{ \mathrm{E}_{\mathrm{meas},p}(g \vert s) \colon g \in \setofextvariablesb{} , \, g \text{ is \(\mathscr{F}\)-measurable and } g \geq f \Bigr\}  
\geq \mathrm{E}_{\mathrm{meas},p}(f \vert s).
\end{align}
If \(\mathrm{E}_{\mathrm{meas},p}(f^{+} \vert s) = +\infty\) and therefore \(\mathrm{E}_{\mathrm{meas},p}(f^{-} \vert s) \in \reals{}\) and \(\mathrm{E}_{\mathrm{meas},p}(f \vert s)=+\infty\), the desired equality follows trivially from Equation~\eqref{Eq: prop: alternative extension is an extension}.
We proceed to show that this is also the case if \(\mathrm{E}_{\mathrm{meas},p}(f^{+} \vert s) < +\infty\).

Consider the non-increasing sequence \(\{f^{\vee -n}\}_{n \in \natz{}}\) of lower cuts of $f$.
Then, for any \(n \in \natz{}\), the global variable \(f^{\vee -n}\) is bounded below and \(\mathscr{F}\)-measurable. 
Hence,
\begin{align*}
\mupprevpreciseone(f \vert s) 
&= \inf \Bigl\{ \mathrm{E}_{\mathrm{meas},p}(g \vert s) \colon g \in \setofextvariablesb{} , \, g \text{ is \(\mathscr{F}\)-measurable and } g \geq f \Bigr\} \\
&\leq \inf_{n \in \natz{}} \mathrm{E}_{\mathrm{meas},p}(f^{\vee -n} \vert s) 
= \lim_{n \to +\infty} \mathrm{E}_{\mathrm{meas},p}(f^{\vee -n} \vert s)
= \mathrm{E}_{\mathrm{meas},p}(f \vert s),
\end{align*}
where the second equality follows from the non-increasing character of the sequence \(\{f^{\vee -n}\}_{n \in \natz{}}\) and the monotonicity [\ref{prop measure: monotonicity}] of \(\mathrm{E}_{\mathrm{meas},p}\), and the final equality follows from \ref{lemma: limits for precise measure-theoretic expectation iii}, which we can use because \(f^{\vee -n} \leq f^{\vee 0} = f^+\) for all \(n \in \natz{}\) and \(\mathrm{E}_{\mathrm{meas},p}(f^{+} \vert s) < +\infty\) by assumption.
Combining this with Equation~\eqref{Eq: prop: alternative extension is an extension}, the desired equality again follows.
\end{proof}

Our next result shows that $\mupprevpreciseone$ is in fact the most conservative extension of $\mathrm{E}_{\mathrm{meas},p}$ that satisfies \ref{P monotonicity}.
\begin{proposition}\label{prop: extension of precise measure is most conservative}
$\mupprevpreciseone$ is the most conservative extension of \/ $\mathrm{E}_{\mathrm{meas},p}$ that satisfies \ref{P monotonicity}.
\end{proposition}
\begin{proof}
Fix any \(f \in \setofextvariables{}\), any \(s \in \situations{}\) and any \(\upprev{p}' \colon \setofextvariables{} \times \situations{} \to \extreals{}\) that satisfies \ref{P monotonicity} and that coincides with \(\mathrm{E}_{\mathrm{meas},p}\) on its domain.
Then,
\begin{align*}
\mupprevpreciseone(f \vert s) 
&= \inf \Bigl\{ \mathrm{E}_{\mathrm{meas},p}(g \vert s) \colon g \in \setofextvariablesb{} , \, g \text{ is \(\mathscr{F}\)-measurable and } g \geq f \Bigr\} \\
&= \inf \Bigl\{ \upprev{p}'(g \vert s) \colon g \in \setofextvariablesb{} , \, g \text{ is \(\mathscr{F}\)-measurable and } g \geq f \Bigr\}
\geq \upprev{p}'(f \vert s),
\end{align*}
where the second equality follows from the fact that \(\smash{\mathrm{E}_{\mathrm{meas},p}(g \vert s)}\) exists for any $\mathscr{F}$-measurable variable $g\in\setofextvariablesb{}$, and the inequality follows from the fact that \(\smash{\upprev{p}'}\) satisfies monotonicity [\ref{P monotonicity}].
Hence, the proposition now follows immediately from the fact that, due to Proposition~\ref{prop: alternative extension is an extension}, $\mupprevpreciseone$ itself is an extension of $\mathrm{E}_{\mathrm{meas},p}$ and the fact that $\mupprevpreciseone$ satisfies \ref{P monotonicity} because $\mathrm{E}_{\mathrm{meas},p}$ is monotone [\ref{prop measure: monotonicity}].
\end{proof}

In Section~\ref{Sect: Search of a global model}, we argued that \ref{P compatibility}--\ref{P continuity} are desirable for a global upper expectation and, under those axioms, we put forward the most conservative upper expectation as our model of choice. 
Now, Proposition~\ref{prop: extension of precise measure is most conservative} guarantees that $\mupprevpreciseone$ is the most conservative extension of $\smash{\mathrm{E}_{\mathrm{meas},p}}$ only under \ref{P monotonicity}.
So it may very well be that $\smash{\mupprevpreciseone}$ does not necessarily satisfy the axioms \ref{P compatibility}--\ref{P continuity}.
Since we consider these axioms to be desirable, that would mean that $\mupprevpreciseone$ is a too conservative extension of $\smash{\mathrm{E}_{\mathrm{meas},p}}$, in the sense that it results in an unwarranted loss of information.
Our following result assures that this is not the case; the extension $\mupprevpreciseone$ coincides with $\upprev{\mathrm{A},p}$ and hence, $\mupprevpreciseone$ is guaranteed to satisfy \ref{P compatibility}--\ref{P continuity}.

\begin{theorem}\label{theorem: all models coincide for precise probability trees}
For any \(f \in \setofextvariables{}\) and any \(s \in \situations{}\), we have that \/ \(\mupprevpreciseone(f \vert s) = \upprev{\mathrm{A},p}(f \vert s)\).
\end{theorem}
\begin{proof}
For any \(f \in \setofextvariables{}\) and any \(s \in \situations{}\), we immediately find that
\begin{align*}
\mupprevpreciseone(f \vert s) 
&= \inf \Bigl\{ \mathrm{E}_{\mathrm{meas},p}(g \vert s)\colon g \in \setofextvariablesb{} , \, g \text{ is \(\mathscr{F}\)-measurable and } g \geq f \Bigr\}  \\
&= \inf \Bigl\{ \upprev{\mathrm{A},p}(g \vert s)\colon g \in \setofextvariablesb{} , \, g \text{ is \(\mathscr{F}\)-measurable and } g \geq f \Bigr\}
\geq \upprev{\mathrm{A},p}(f \vert s),
\end{align*}
where the second equality follows from Theorem~\ref{theorem: Ville extended variables bounded below} and the inequality from the monotonicity [\ref{P monotonicity}] of \(\upprev{\mathrm{A},p}\).
To show that the converse inequality holds, we will use the fact that \(\smash{\upprev{\mathrm{A},p}}\) satisfies \ref{P natural extension} (this follows from Proposition~\ref{Prop: vovk satisfies P1-P7} and Theorem~\ref{theorem: Vovk is the largest}).

Consider any \(g \in \setoflimitsoffinmeasextb{}\).
Since \(g\) is the pointwise limit of a sequence \(\{g_n\}_{n \in \natz{}}\) of finitary gambles \(g_n \in \setoffingambles{}\) and since any finitary gamble is clearly \(\mathscr{F}\)-measurable, \(g\) is the pointwise limit of a sequence of \(\mathscr{F}\)-measurable global gambles.
Then it follows that \(g\) itself is also \(\mathscr{F}\)-measurable ($\mathscr{F}$-measurability is preserved under pointwise convergence). 
Furthermore, by the definition of \(\setoflimitsoffinmeasextb{}\), \(g\) is also bounded below.
Hence, \(\mathrm{E}_{\mathrm{meas},p}(g \vert s)\) exists and due to Theorem~\ref{theorem: Ville extended variables bounded below}, we have that \(\smash{\upprev{\mathrm{A},p}(g \vert s)} = \smash{\mathrm{E}_{\mathrm{meas},p}(g \vert s)}\).
Since $\mupprevpreciseone$ is an extension of $\mathrm{E}_{\mathrm{meas},p}$ by Proposition~\ref{prop: alternative extension is an extension}, this implies that \(\smash{\upprev{\mathrm{A},p}(g \vert s)} = \mupprevpreciseone(g \vert s)\).
We conclude that
\begin{align*}
	\upprev{\mathrm{A},p} (f \vert s) 
	&= \inf{\Big\{ \upprev{\mathrm{A},p}(g \vert s) \colon g \in \setoflimitsoffinmeasextb{} \text{ and }  g \geq f  \Big\}}
	= \inf{\Big\{ \mupprevpreciseone(g \vert s) \colon g \in \setoflimitsoffinmeasextb{} \text{ and }  g \geq f  \Big\}} 
	\geq  \mupprevpreciseone(f \vert s),
\end{align*}
where the first step follows from \ref{P natural extension} and the last step follows from the monotonicity [\ref{P monotonicity}] of \(\mupprevpreciseone\) (as established by Proposition~\ref{prop: extension of precise measure is most conservative}).
\end{proof}

Theorem~\ref{theorem: all models coincide for precise probability trees} does not only show us that \(\mupprevpreciseone\) satisfies \ref{P compatibility}--\ref{P continuity}, but also, and more importantly, that it is the most conservative global upper expectation satisfying \ref{P compatibility}--\ref{P continuity}.
Furthermore, if we combine Theorem~\ref{theorem: all models coincide for precise probability trees} with Proposition~\ref{prop: extension of precise measure is most conservative}, it can also easily be seen that it is the most conservative extension of \(\mathrm{E}_{\mathrm{meas},p}\) satisfying \ref{P compatibility}--\ref{P continuity}.

\begin{corollary}\label{corollary: mupprevpreciseone is most conservative under P1-P5}
\(\mupprevpreciseone\) is the most conservative extension of \/ \(\mathrm{E}_{\mathrm{meas},p}\) that satisfies \ref{P compatibility}--\ref{P continuity}.
\end{corollary}
\begin{proof}
We know from Proposition~\ref{prop: extension of precise measure is most conservative} that \(\mupprevpreciseone\) is the most conservative extension of \(\mathrm{E}_{\mathrm{meas},p}\) that satisfies \ref{P monotonicity}.
This implies in particular that \(\mupprevpreciseone\) is more (or at least as) conservative than any extension of \(\mathrm{E}_{\mathrm{meas},p}\) that satisfies \ref{P compatibility}--\ref{P continuity}.
Since \(\mupprevpreciseone\) itself satisfies \ref{P compatibility}--\ref{P continuity} due to Theorem~\ref{theorem: all models coincide for precise probability trees}, it follows that \(\mupprevpreciseone\) is indeed the most conservative extension of \(\mathrm{E}_{\mathrm{meas},p}\) satisfying \ref{P compatibility}--\ref{P continuity}.
\end{proof}

Another interesting consequence of Theorem~\ref{theorem: all models coincide for precise probability trees} and Proposition~\ref{prop: extension of precise measure is most conservative} is that $\upprev{\mathrm{A},p}$ turns out to coincide with $\mathrm{E}_{\mathrm{meas},p}$ on the entire domain where it is defined, or, in other words, our model $\smash{\upprev{\mathrm{A},p}}$ is an extension of the traditional measure-theoretic expectation $\mathrm{E}_{\mathrm{meas},p}$.

\begin{corollary}\label{corollary: our model coincides with traditional measure}
For any $f\in\setofextvariables{}$ and any $s\in\situations{}$, we have that \/ $\upprev{\mathrm{A},p}(f\vert s) = \mathrm{E}_{\mathrm{meas},p}(f\vert s)$ if \/ $\mathrm{E}_{\mathrm{meas},p}(f\vert s)$ exists.
\end{corollary}
\begin{proof}
Immediate consequence of Theorem~\ref{theorem: all models coincide for precise probability trees} and Proposition~\ref{prop: extension of precise measure is most conservative}.
\end{proof}

Now, before we continue to consider the more general case of imprecise probability trees, we briefly want to come back to our choice of the extension $\smash{\mupprevpreciseone{}}$.
Recall that we defined it as an upper approximation---called an upper integral---with respect to all $\mathscr{F}$-measurable variables in $\setofextvariablesb{}$.
However, one could rightfully wonder why we approximate with respect to this particular domain.
It turned out to be a suitable choice, as Corollary~\ref{corollary: mupprevpreciseone is most conservative under P1-P5} shows, but there was a priori no concrete reason for doing so.
We could just as well have adopted the following alternative upper integral:\footnote{This upper integral seems to be the one used by Shafer and Vovk in \cite[Chapter~9]{Vovk2019finance}.}
\begin{align*}
\mupprevprecisetwo{}(f \vert s) 
\coloneqq \inf \Biggl\{ \mathrm{E}_{\mathrm{meas},p}(g \vert s) \colon g\in\setofextvariables{}, \mathrm{E}_{\mathrm{meas},p}(g \vert s) \text{ exists and } g \geq f \Biggr\} \text{ for all } f \in \setofextvariables{} \text{ and } s \in \situations{}.
\end{align*}
The upper expectation $\smash{\mupprevprecisetwo{}}$ is an extension of $\smash{\mathrm{E}_{\mathrm{meas},p}}$ because \(\smash{\mathrm{E}_{\mathrm{meas},p}}\) is monotone [\ref{prop measure: monotonicity}].
Moreover, it is also clear that $\smash{\mupprevprecisetwo{}(f\vert s)}$ is dominated by $\smash{\mupprevpreciseone{}(f\vert s)}$ for all $f\in\setofextvariables{}$ and all $s\in\situations{}$ because the infimum in $\smash{\mupprevprecisetwo{}}$ is taken over an equal or larger set than the set over which the infimum is taken in $\mupprevpreciseone{}$.
Hence, one could be tempted to adopt the upper integral $\mupprevprecisetwo{}$ instead of $\mupprevpreciseone{}$ with the aim of obtaining a more informative extension---recall that we regard higher upper expectations as less informative.
However, it can straightforwardly be shown that this is only vain hope.

\begin{proposition}\label{prop: two extensions of measure coincide}
We have that \/ $\mupprevpreciseone{}(f\vert s) = \mupprevprecisetwo{}(f\vert s)$ for all $f\in\setofextvariables{}$ and all $s\in\situations{}$.
\end{proposition}
\begin{proof}
Fix any $f\in\setofextvariables{}$ and any $s\in\situations{}$.
That $\mupprevpreciseone{}(f\vert s) \geq \mupprevprecisetwo{}(f\vert s)$ follows immediately from the fact that $\mathrm{E}_{\mathrm{meas},p}(g \vert s)$ exists for each $\mathscr{F}$-measurable variable $g\in\setofextvariablesb{}$ and therefore, that the infimum in $\mupprevprecisetwo{}(f\vert s)$ is taken over a set that is at least as large as the set over which the infimum is taken in $\mupprevpreciseone{}(f\vert s)$.
On the other hand, we have that
\begin{align*}
\mupprevprecisetwo{}(f \vert s) 
&= \inf \Biggl\{ \mathrm{E}_{\mathrm{meas},p}(g \vert s) \colon  g\in\setofextvariables{}, \mathrm{E}_{\mathrm{meas},p}(g \vert s) \text{ exists and } g \geq f \Biggr\} \\
&= \inf \Biggl\{ \mupprevpreciseone{}(g \vert s) \colon  g\in\setofextvariables{},\mathrm{E}_{\mathrm{meas},p}(g \vert s) \text{ exists and } g \geq f \Biggr\} 
\geq \mupprevpreciseone{}(f \vert s),
\end{align*}
where the second equality follows from Proposition~\ref{prop: alternative extension is an extension} and the inequality follows from the monotonicity of $\mupprevpreciseone{}$ (see Proposition~\ref{prop: extension of precise measure is most conservative}).
\end{proof}

Hence, both measure-theoretic extensions $\mupprevpreciseone{}$ and $\mupprevprecisetwo{}$ are equal and therefore, by Theorem~\ref{theorem: all models coincide for precise probability trees}, both of them coincide with our model $\upprev{\mathrm{A},p}$ which, on its turn, coincides with the game-theoretic upper expectation $\upprevvovkk{}$ according to Theorem~\ref{theorem: Vovk is the largest}.
In summary, we conclude that, if the local models are precise, all of the three possible approaches---the game-theoretic, the measure-theoretic and ours---are equivalent.

\subsection*{Imprecise Probability Trees}\label{Sect: Measure-theoretic expectations in imprecise probability trees}

Let us now turn to the general case where we consider an imprecise probability tree \(\tree{}\) for which the local models \(\lupprev{s}\) are (boundedly) coherent upper expectations that are not necessarily self-conjugate.  
As we have explained earlier, these local models can equivalently be represented by sets \(\mathbb{P}_s\) of probability mass functions on \(\statespace\), where \(\smash{\mathbb{P}_s\coloneqq\mathbb{P}_{\mbox{\tiny$\lupprev{s}$}}}\) is related to \(\lupprev{s}\) according to Equations~\eqref{Eq: link upprev and mass functions} and~\eqref{Eq: second link upprev and mass functions} for all \(s \in \situations{}\).
We call a precise probability tree \(p \colon s \in \situations{} \mapsto p(\cdot \vert s)\) \emph{compatible} with the imprecise probability tree \(\tree{}\) and write that \(p \sim \tree{}\) if \(p(\cdot \vert s) \in \mathbb{P}_s\) for all \(s \in \situations{}\).
In other words, a compatible precise probability tree \(p \sim \tree{}\) corresponds to a possible selection, where a local probability mass function \(p(\cdot \vert s)\) is chosen from the set \(\mathbb{P}_s\) for each situation \(s \in \situations{}\).
Now, for any such compatible tree \(p \sim \tree{}\), we can consider the global upper expectation \(\upprev{\mathrm{meas},p}\coloneqq\mupprevpreciseone{}=\mupprevprecisetwo{}\) as defined in the previous section.
Hence, instead of a single upper expectation, we now obtain a set of compatible global upper expectations.
We define the \emph{global measure-theoretic upper expectation} \(\upprev{\mathrm{meas}}\) as the upper envelope over all such compatible upper expectations:
\begin{align*}
\upprev{\mathrm{meas}}(f \vert s) \coloneqq \sup \, \Bigl\{ \upprev{\mathrm{meas},p}(f \vert s) \colon p \sim \tree{} \Bigr\} \text{ for all } f \in \setofextvariables{} \text{ and } s \in \situations{}.
\end{align*}
This definition is in correspondence with the usual measure-theoretic view towards imprecision where it is regarded as partial information about a single, unknown probability measure.
Upper (and lower) expectations are then considered as upper (and lower) bounds on the set of all `possible' linear expectations.
A similar, yet unconditional version of this measure-theoretic upper expectation is also adopted by Shafer and Vovk in \cite[Chapter~9]{Vovk2019finance}.

As we will now show, the measure-theoretic upper expectation \(\upprev{\mathrm{meas}}\) satisfies our axioms \ref{P compatibility}--\ref{P continuity}.
This implies, by Theorem~\ref{theorem: Vovk is the largest}, that \(\upprev{\mathrm{meas}}\) is always dominated by the global upper expectation \(\upprev{\mathrm{A}}\), and therefore also by \(\upprevvovkk{}\).
Moreover, that the operator \(\upprev{\mathrm{meas}}\) is entirely based on a measure-theoretic view towards uncertainty, yet turns out to satisfy our axioms \ref{P compatibility}--\ref{P continuity}, can only be seen as an additional motivation for the use of these axioms.

\begin{theorem}\label{Theorem: measure satisfies P1-P5}
The global measure-theoretic upper expectation \(\mupprev\) is an element of \/ \(\setofupprev{}_{1-5}(\tree{})\). 
\end{theorem}
\begin{proof}
To show that \ref{P compatibility} holds, consider any \(f \in \setofgengambles{}(\statespace{})\) and any \(x_{1:n} \in \situations{}\).
Since \(f(X_{n+1})\) is bounded and clearly \(\mathscr{F}\)-measurable, we have, for any \(p \sim \tree{}\), that \(\mathrm{E}_{\mathrm{meas},p} (f(X_{n+1}) \vert x_{1:n})\) exists and therefore that \(\upprev{\mathrm{meas},p} (f(X_{n+1}) \vert x_{1:n}) = \mathrm{E}_{\mathrm{meas},p} (f(X_{n+1}) \vert x_{1:n})\), due to Proposition~\ref{prop: alternative extension is an extension}.
Moreover, using \(\mathrm{P}_{\vert x_{1:n}}(x_{n+1})\) to denote \(\mathrm{P}_{\vert x_{1:n}}( \cup_{z_{1:n} \in \statespace{}^n} \Gamma(z_{1:n}x_{n+1}))\) for any $x_{n+1}\in\statespace{}$, 
it is clear that \(\smash{\mathrm{E}_{\mathrm{meas},p} (f(X_{n+1}) \vert x_{1:n}) = \sum_{x_{n+1} \in \statespace{}} f(x_{n+1}) \mathrm{P}_{\vert x_{1:n}}(x_{n+1})}\).
It also follows immediately from the construction \eqref{Eq: precise probability on algebra} of \(\mathrm{P}_{\vert x_{1:n}} = \mathrm{P}(\cdot \vert x_{1:n})\) that \(\smash{\mathrm{P}_{\vert x_{1:n}}(x_{n+1}) = \mathrm{P}_{\vert x_{1:n}}(x_{1:n+1}) = p(x_{n+1} \vert x_{1:n})}\) for all $x_{n+1}\in\statespace{}$ and \(p \sim \tree{}\).
Hence,
\begin{align*}
\mupprev{}(f(X_{n+1}) \vert x_{1:n}) 
&= \sup \Bigl\{ \upprev{\mathrm{meas},p}(f(X_{n+1}) \vert x_{1:n}) \colon p \sim \tree{} \Bigr\} \\
&= \sup \Bigl\{ \mathrm{E}_{\mathrm{meas},p}(f(X_{n+1}) \vert x_{1:n}) \colon p \sim \tree{} \Bigr\} \\
&= \sup \Bigg\{ \sum_{x_{n+1} \in \statespace{}} f(x_{n+1}) p(x_{n+1} \vert x_{1:n}) \colon p \sim \tree{} \Bigg\} \\
&= \sup \Bigg\{ \sum_{x_{n+1} \in \statespace{}} f(x_{n+1}) p(x_{n+1} \vert x_{1:n}) \colon p(\cdot \vert x_{1:n}) \in \mathbb{P}_{x_{1:n}} \Bigg\}
= \lupprev{x_{1:n}}(f),
\end{align*}
where the last step follows from the relation between \(\smash{\mathbb{P}_{x_{1:n}}}\) and \(\smash{\lupprev{x_{1:n}}}\) which is given by Equation~\eqref{Eq: second link upprev and mass functions}.

The remaining properties \ref{P conditional}--\ref{P continuity} can be easily derived from the definition of $\upprev{\mathrm{meas}}$ and the fact that, due to Corollary~\ref{corollary: mupprevpreciseone is most conservative under P1-P5}, these properties are satisfied by each compatible $\smash{\upprev{\mathrm{meas},p}}$.
For Properties~\ref{P conditional} and~\ref{P monotonicity} this is immediate.
To see that \ref{P iterated} holds, consider any $f\in\smash{\setoffingambles{}}$ and any $n\in\natz$.
Note that, for any compatible tree $p\sim\tree$ and any $x_{1:n}\in\statespace{}^n$, we have that $\upprev{\mathrm{meas},p}(f\vert x_{1:n}) \leq \upprev{\mathrm{meas},p}(\upprev{\mathrm{meas},p}(f\vert X_{1:n+1}) \vert x_{1:n})$ because $\upprev{\mathrm{meas},p}$ satisfies \ref{P iterated}.
Hence, 
\begin{align*}
\upprev{\mathrm{meas}}(f\vert x_{1:n}) 
= \sup \Bigl\{ \upprev{\mathrm{meas},p}(f \vert x_{1:n}) \colon p \sim \tree{} \Bigr\}
&\leq \sup \Bigl\{\upprev{\mathrm{meas},p}\big(\upprev{\mathrm{meas},p}(f\vert X_{1:n+1}) \vert x_{1:n}\big) \colon p \sim \tree{} \Bigr\}\\
&\leq  \sup \Bigl\{\upprev{\mathrm{meas},p}\big(\upprev{\mathrm{meas}}(f\vert X_{1:n+1}) \vert x_{1:n}\big) \colon p \sim \tree{} \Bigr\} \\
&=  \upprev{\mathrm{meas}}\big(\upprev{\mathrm{meas}}(f\vert X_{1:n+1}) \vert x_{1:n}\big),
\end{align*}
where the second inequality follows from the monotonicity [\ref{P monotonicity}] of $\smash{\upprev{\mathrm{meas},p}}$ and the fact that $\smash{\upprev{\mathrm{meas},p}(f\vert z_{1:n+1})}\leq\upprev{\mathrm{meas}}(f\vert z_{1:n+1})$ for all $p\sim\tree$ and all $z_{1:n+1}\in\statespace{}^{n+1}$.
The inequality above holds for any $x_{1:n}\in\statespace{}^n$, so we indeed find that $\smash{\upprev{\mathrm{meas}}(f\vert X_{1:n}) \leq \upprev{\mathrm{meas}}\big(\upprev{\mathrm{meas}}(f\vert X_{1:n+1}) \vert X_{1:n}\big)}$.
To prove \ref{P continuity}, we consider any \(s \in \situations{}\) and any sequence \(\smash{\{f_n\}_{n \in \natz}}\) of finitary gambles that is uniformly bounded below and that converges pointwise to some variable $f\in\setofextvariables{}$.
For any compatible tree $p\sim\tree{}$, $\upprev{\mathrm{meas},p}$ satisfies \ref{P continuity} and therefore, $\upprev{\mathrm{meas},p}(f \vert s) \leq \limsup_{n \to +\infty} \upprev{\mathrm{meas},p}(f_n \vert s)$.
Then we also have that
\begin{align*}
\upprev{\mathrm{meas}}(f \vert s)
= \sup \Bigl\{\upprev{\mathrm{meas},p}(f \vert s) \colon p\sim\tree{}\Bigr\} 
&\leq \sup \Bigl\{\limsup_{n \to +\infty} \upprev{\mathrm{meas},p}(f_n \vert s) \colon p\sim\tree{}\Bigr\} \\
&\leq \sup \Bigl\{\limsup_{n \to +\infty} \, \Bigl[ \sup \Bigl\{\upprev{\mathrm{meas},p'}(f_n \vert s) \colon p'\sim\tree{}\Bigr\}\Bigr] \colon p\sim\tree{} \Bigr\} \\
&= \limsup_{n \to +\infty} \, \Bigl[ \sup \Bigl\{\upprev{\mathrm{meas},p'}(f_n \vert s) \colon p'\sim\tree{}\Bigr\}\Bigr] 
= \limsup_{n \to +\infty} \upprev{\mathrm{meas}}(f_n \vert s).
\tag*{\qedhere}
\end{align*}
\end{proof}
As mentioned before, the result above implies, in combination with Theorem~\ref{theorem: Vovk is the largest}, that $\upprev{\mathrm{meas}}$ is dominated by our model $\upprev{\mathrm{A}}$.
\begin{corollary}\label{corollary: inequality between mupprev and upprev_A}
We have that \/ \(\mupprev ( f \vert s) \leq \upprev{\mathrm{A}}( f \vert s) = \upprevvovkk{}( f \vert s)\) for all \(f \in \setofextvariables{}\) and all \(s \in \situations{}\).
\end{corollary}
\begin{proof}
Immediate consequence of Theorem~\ref{theorem: Vovk is the largest} and Theorem~\ref{Theorem: measure satisfies P1-P5}.
\end{proof} 

Corollary~\ref{corollary: inequality between mupprev and upprev_A} extends upon an earlier result by Lopatatzidis et al. \cite[Theorem~33]{8535240}, which says that $\mupprev (f\vert s) \leq \upprevvovkk{}(f \vert s)$ for all $s\in\situations{}$ and some specific variables $f$ in $\setoflimitsoffinmeasextb{}$. 
The definition of $\upprevvovkk{}$ in \cite{8535240} slightly differs from ours though.
The version in \cite{8535240} is always larger or equal than---so at least as conservative as---ours because they work with supermartingales that can only take real values; see \cite[section~8]{Tjoens2020FoundationsARXIV} for more details.
Hence, Corollary~\ref{corollary: inequality between mupprev and upprev_A} certainly implies \cite[Theorem~33]{8535240}. 
Another existing result that is strongly related to Corollary~\ref{corollary: inequality between mupprev and upprev_A} is due to Shafer and Vovk \cite[Proposition~9.8]{Vovk2019finance}, which seems to say that $\mupprev (f\vert \Box) \leq \upprevvovkk{}(f \vert \Box)$ for all bounded below variables \(f\in\setofextvariablesb{}\).
Since our result applies to all $f\in\setofextvariables{}$ and all $s\in\situations{}$, our result would in that sense be stronger than theirs. 
However, it is not immediately clear whether we can actually draw such conclusions.
Unlike in our case, Shafer and Vovk~\cite[Section~9.2]{Vovk2019finance} do not require that, for each possible situation, a local model should be specified beforehand.
To put it in their terms: Forecaster need not declare to Skeptic beforehand---that is, before the game starts---which local upper expectations he is going to choose.
Instead, Shafer and Vovk only start off with a set $\mathbb{Q}$ of local upper expectations from which, at each situation, Forecaster can choose.
Since this set $\mathbb{Q}$ is moreover required to be constant throughout the whole game, this crucially implies that we, as subjects outside the game, are unable to specify Forecaster's moves---and hence, the choice of the local upper expectation at each situation.
The global game-theoretic upper expectation will therefore not depend on Forecaster's moves, but rather only on the set $\mathbb{Q}$.
Moreover, both of their global models---their global game-theoretic upper expectation and their global measure-theoretic upper expectation---are defined on variables with a domain that differs from $\samplespace{}$.
Besides the fact that we find it far from trivial how this whole framework relates to ours here, they also impose some relatively strong toplogical conditions on the set $\mathbb{Q}$.
Hence, all together, in order to draw a fair comparison, we feel like a closer look is required.
We hope to clarify this in our future work.


Moving on, the inevitable question that comes to mind when looking at Corollary~\ref{corollary: inequality between mupprev and upprev_A} is whether the inequality can be turned into an equality. We are aware of two partial answers.
On the one hand, Shafer and Vovk seem to establish in \cite[Theorem~9.7]{Vovk2019finance} that $\mupprev{}(f \vert \Box) = \upprevvovkk{}(f \vert \Box)$ for all Suslin gambles $f\in\setofgambles{}$; this includes all $\mathscr{F}$-measurable gambles and is therefore a very strong result. 
However, here too, as before, their framework differs substantially from ours and the result requires additional topological assumptions. 
It is therefore again not immediately clear if, and how, it translates to our context; we leave this for future work.
On the other hand, we also have a result by Lopatatzidis et al. \cite[Theorem~29]{8535240} that shows that $\mupprev{}(f \vert s)=\upprevvovkk{}(f \vert s)$ for all $s\in\situations{}$ and all $n$-measurable gambles~$f$---and therefore all finitary gambles~$f$ in $\setoffingambles{}$.
Here too, however, the result is not immediately applicable because
---as explained directly after Corollary~\ref{corollary: inequality between mupprev and upprev_A}---the version of $\upprevvovkk{}$ that is used in \cite{8535240} differs from ours. 
We therefore give a self-contained proof instead.
Nevertheless, our ideas and techniques are essentially the same as the ones that were used to obtain \cite[Theorem~29]{8535240}.

We start by showing that $\mupprev{}$ satisfies \ref{P iterated general}; the desired equivalence result will then follow as an immediate consequence.

\begin{lemma}\label{lemma: equality between all models for finitary in imprecise case}
The upper expectation $\mupprev{}$ satisfies \ref{P iterated general}.
\end{lemma}
\begin{proof}
Consider any $f\in\setoffingambles{}$.
Let $n\in\natz$ be such that $f$ is $n$-measurable---there is at least one such an $n$ because $f$ is finitary.
Then, to establish the fact that $\mupprev{}$ satisfies \ref{P iterated general}, it suffices to prove that, for all $m\in\natz$,
\begin{align}\label{Eq: prop: equality between all models for finitary in imprecise case}
\text{there is some } p\sim\tree{} \text{ such that }
\mupprev{}(\mupprev{}(f \vert X_{1:m+1}) \vert X_{1:m}) =\upprev{\mathrm{meas},p}(f \vert X_{1:m}).
\end{align}
Indeed, by definition of $\mupprev{}$, $\smash{\upprev{\mathrm{meas},p}(f \vert X_{1:m}) \leq \mupprev{}(f \vert X_{1:m})}$ for all $p\sim\tree$ and all $m\in\natz$, and therefore, expression~\eqref{Eq: prop: equality between all models for finitary in imprecise case} implies that $\mupprev{}(\mupprev{}(f \vert X_{1:m+1})\vert X_{1:m}) \leq \mupprev{}(f \vert X_{1:m})$ for all $m\in\natz$.
The converse inequality then follows from Theorem~\ref{Theorem: measure satisfies P1-P5}, which says that $\mupprev{}$ satisfies \ref{P iterated}.

Let us first show that expression~\eqref{Eq: prop: equality between all models for finitary in imprecise case} holds for any $m \geq n$.
Consider any $\ell\geq n$ and any $x_{1:\ell}\in\statespace{}^{\ell}$.
Since $f$ is $n$-measurable and $\ell\geq n$, $f$ takes a constant value $f(x_{1:n})$ on $\Gamma(x_{1:\ell})$, and therefore, we can write that $f \indica{x_{1:\ell}} = f(x_{1:n})\indica{x_{1:\ell}}$.
Hence, for any $\smash{p\sim\tree}$, since $\smash{\upprev{\mathrm{meas},p}}$ satisfies \ref{P conditional} due to Theorem~\ref{theorem: all models coincide for precise probability trees}, we have that 
\begin{align*}
\upprev{\mathrm{meas},p}(f \vert x_{1:\ell})
\overset{\text{\ref{P conditional}}}{=}\upprev{\mathrm{meas},p}(f \indica{x_{1:\ell}} \vert x_{1:\ell}) 
=\upprev{\mathrm{meas},p}(f(x_{1:n}) \indica{x_{1:\ell}} \vert x_{1:\ell})
\overset{\text{\ref{P conditional}}}{=}\upprev{\mathrm{meas},p}(f(x_{1:n}) \vert x_{1:\ell})
= \prev_{\mathrm{meas},p}(f(x_{1:n}) \vert x_{1:\ell}).
\end{align*}
where the last step follows from Proposition~\ref{prop: alternative extension is an extension} and the fact that $\prev_{\mathrm{meas},p}(f(x_{1:n}) \vert x_{1:\ell})$ exists because $f(x_{1:n})$ is a constant.
Furthermore, since $f(x_{1:n})$ is real-valued [because $f$ is a gamble], \ref{prop measure: linearity} implies that $\prev_{\mathrm{meas},p}(f(x_{1:n}) \vert x_{1:\ell})=f(x_{1:n})$.
Hence, we have that 
\begin{align}\label{Eq: prop: equality between all models for finitary in imprecise case 4}
\upprev{\mathrm{meas},p}(f \vert x_{1:\ell}) = f(x_{1:n})  \text{ for all } p\sim\tree, \text{ all } \ell\geq n \text{ and all } x_{1:\ell}\in\statespace{}^{\ell}.
\end{align}
So, since $m\geq n$, the right hand side of Equation~\eqref{Eq: prop: equality between all models for finitary in imprecise case} is equal to $f(X_{1:n})=f$, whatever the tree $p\sim\tree$.
Moreover, note that Equation~\eqref{Eq: prop: equality between all models for finitary in imprecise case 4} also implies that $\mupprev{}(f \vert x_{1:\ell}) = f(x_{1:n})$ for all $\ell\geq n$ and all $x_{1:\ell}\in\statespace{}^{\ell}$.
Then, since $m\geq n$, and therefore certainly $m+1\geq n$, the left hand side of Equation~\eqref{Eq: prop: equality between all models for finitary in imprecise case} is equal to
\begin{align*}
\mupprev{}(\mupprev{}(f \vert X_{1:m+1})\vert X_{1:m})
= \mupprev{}(f(X_{1:n}) \vert X_{1:m})
= \mupprev{}(f \vert X_{1:m}) = f,
\end{align*}
where the last step follows once more from the fact that $m\geq n$, and that $\mupprev{}(f \vert x_{1:\ell}) = f(x_{1:n})$ for all $\ell\geq n$ and all $x_{1:\ell}\in\statespace{}^{\ell}$.
Hence, irrespectively of the tree $p\sim\tree$, both sides of Equation~\eqref{Eq: prop: equality between all models for finitary in imprecise case} are equal to $f$---and therefore equal to each other---for the case that $m\geq n$.

For the case that $m < n$, we will establish expression~\eqref{Eq: prop: equality between all models for finitary in imprecise case} using a decreasing induction argument in $m$.  
The case where $m=n$ will serve as an induction base; that expression~\eqref{Eq: prop: equality between all models for finitary in imprecise case} holds for this case was just proved above.
Now suppose that expression~\eqref{Eq: prop: equality between all models for finitary in imprecise case} is satisfied with $m=\ell$ where $1\leq\ell\leq n$, meaning that there is some $p\sim\tree$ such that $\smash{\mupprev{}(\mupprev{}(f \vert X_{1:\ell+1}) \vert X_{1:\ell})=\upprev{\mathrm{meas},p}(f \vert X_{1:\ell})}$.
First note that this implies that $\smash{\mupprev{}(f \vert X_{1:\ell})}=\smash{\upprev{\mathrm{meas},p}(f \vert X_{1:\ell})}$ because, on the one hand, $\smash{\mupprev{}(f \vert X_{1:\ell}) \leq \mupprev{}(\mupprev{}(f \vert X_{1:\ell+1}) \vert X_{1:\ell})}$ due to Theorem~\ref{Theorem: measure satisfies P1-P5} which says that $\mupprev{}$ satisfies \ref{P iterated}, and on the other hand, $\upprev{\mathrm{meas},p}(f \vert X_{1:\ell}) \leq \mupprev{}(f \vert X_{1:\ell})$ due to the definition of $\mupprev{}$.
Furthermore, because $f$ is a finitary gamble, and therefore an $\mathscr{F}$-measurable gamble, $\prev_{\mathrm{meas},p}(f \vert X_{1:\ell})$ exists and therefore, by Proposition~\ref{prop: alternative extension is an extension}, 
\begin{align*}
\mupprev{}(f \vert X_{1:\ell}) = \upprev{\mathrm{meas},p}(f \vert X_{1:\ell})=\prev_{\mathrm{meas},p}(f \vert X_{1:\ell}).
\end{align*}
Moreover, using \ref{prop measure: linearity} and \ref{prop measure: monotonicity}, and the fact that $\inf f$ and $\sup f$ are both real (because $f$ is a gamble), it can easily be inferred that $\inf f \leq \prev_{\mathrm{meas},p}(f \vert X_{1:\ell}) \leq \sup f$.
So $\mupprev{}(f \vert X_{1:\ell}) = \prev_{\mathrm{meas},p}(f \vert X_{1:\ell})$ is a gamble, and more specifically an $\ell$-measurable gamble.
Then, fix any $\smash{x_{1:\ell-1}\in\statespace{}^{\ell-1}}$ and note that, since $\smash{\mupprev{}}$ satisfies \ref{P compatibility} and \ref{P conditional} by Theorem~\ref{Theorem: measure satisfies P1-P5}, we can apply Lemma~\ref{Lemma: compatibility} to obtain that 
\begin{align*}
\mupprev{}(\mupprev{}(f \vert X_{1:\ell}) \vert x_{1:\ell-1})
= \mupprev{}(\prev_{\mathrm{meas},p}(f \vert X_{1:\ell}) \vert x_{1:\ell-1})
= \lupprev{x_{1:\ell-1}}(\prev_{\mathrm{meas},p}(f \vert x_{1:\ell-1} \, \cdot)).
\end{align*}
Then, due to the fact that $\mathbb{P}_{x_{1:\ell-1}}$ is related with $\lupprev{x_{1:\ell-1}}$ through Equation~\eqref{Eq: second link upprev and mass functions}, there is some $p'(\, \cdot \, \vert x_{1:\ell-1})\in\mathbb{P}_{x_{1:\ell-1}}$ such that
\begin{align}\label{Eq: prop: equality between all models for finitary in imprecise case 2}
\mupprev{}(\mupprev{}(f \vert X_{1:\ell}) \vert x_{1:\ell-1})
= \lupprev{x_{1:\ell-1}}(\prev_{\mathrm{meas},p}(f \vert x_{1:\ell-1} \, \cdot))
= \sum_{x_\ell \in\statespace{}} p'(x_\ell \vert x_{1:\ell-1})\, \prev_{\mathrm{meas},p}(f \vert x_{1:\ell}).
\end{align} 
Moreover, since $f$ is $n$-measurable, we can write that $f = \sum_{z_{1:n}\in\statespace{}^n} f(z_{1:n}) \indica{z_{1:n}}$, and therefore, the Lebesgue integral $\prev_{\mathrm{meas},p}(f \vert x_{1:\ell})$ for any $x_\ell\in\statespace{}$, reduces---by definition; see \cite[Section 2.6.1]{shiryaev2016probabilityPartI}---to the finite sum 
\begin{align*}
\prev_{\mathrm{meas},p}(f \vert x_{1:\ell}) 
= \sum_{z_{1:n}\in{\statespace{}}^n} f(z_{1:n}) \mathrm{P}_{\vert x_{1:\ell}}(z_{1:n})
&= \sum_{z_{1:\ell}\in{\statespace{}}^\ell}\sum_{x_{\ell+1:n}\in{\statespace{}}^{n-\ell}} f(z_{1:\ell} x_{\ell+1:n}) \mathrm{P}_{\vert x_{1:\ell}}(z_{1:\ell} x_{\ell+1:n}),
\end{align*}
where the last step simply is a convenient change of notation.
Due to Equation~\eqref{Eq: precise probability on algebra}, we have, for all $z_{1:\ell}\in{\statespace{}}^\ell$ and all $x_{\ell+1:n}\in\statespace{}^{n-\ell}$, that $\mathrm{P}_{\vert x_{1:\ell}}(z_{1:\ell} x_{\ell+1:n}) = 0$ if $z_{1:\ell} \not= x_{1:\ell}$ and $\mathrm{P}_{\vert x_{1:\ell}}(z_{1:\ell} x_{\ell+1:n}) = \smash{\prod_{i=\ell}^{n-1}} p( x_{i+1} \vert x_{1:i})$ if $\smash{z_{1:\ell} = x_{1:\ell}}$.
Substituting these expressions in the equality above, gives us
\begin{align*}
\prev_{\mathrm{meas},p}(f \vert x_{1:\ell}) 
&= \sum_{x_{\ell+1:n}\in{\statespace{}}^{n-\ell}} f(x_{1:n}) \mathrm{P}_{\vert x_{1:\ell}}(x_{1:n}) 
= \sum_{x_{\ell+1:n}\in{\statespace{}}^{n-\ell}}\hspace*{-5pt} f(x_{1:n}) \prod_{i=\ell}^{n-1} p( x_{i+1} \vert x_{1:i}).
\end{align*}
Then, plugging this back into Equation~\eqref{Eq: prop: equality between all models for finitary in imprecise case 2}, we obtain that
\begin{align}\label{Eq: prop: equality between all models for finitary in imprecise case 3}
\mupprev{}(\mupprev{}(f \vert X_{1:\ell}) \vert x_{1:\ell-1})
&= \sum_{x_\ell \in\statespace{}} p'(x_\ell \vert x_{1:\ell-1}) \hspace*{-7pt} \sum_{x_{\ell+1:n}\in{\statespace{}}^{n-\ell}} \hspace*{-5pt}  f(x_{1:n}) \prod_{i=\ell}^{n-1} p( x_{i+1} \vert x_{1:i}) \nonumber \\
&= \sum_{x_{\ell:n}\in{\statespace{}}^{n-\ell+1}} \hspace*{-6pt} f(x_{1:n}) p'(x_\ell \vert x_{1:\ell-1}) \prod_{i=\ell}^{n-1} p( x_{i+1} \vert x_{1:i}).
\end{align}
Now let $\smash{p^\ast\sim\tree}$ be any compatible precise probability tree such that $p^\ast(\, \cdot \, \vert x_{1:\ell-1}) = p'(\, \cdot \, \vert x_{1:\ell-1})$ and $p^\ast(\, \cdot \, \vert x_{1:i}) = p( \, \cdot \,  \vert x_{1:i})$ for all $\ell \leq i \leq n-1$ and all $\smash{x_{\ell:i} \in \statespace{}^{i-\ell+1}}$; it is always possible to choose such a compatible tree $p^\ast$ because $p\sim\tree$ and $\smash{p'(\, \cdot \, \vert x_{1:\ell-1}) \in \mathbb{P}_{x_{1:\ell-1}}}$.
Then note that, for any $x_{\ell:n}\in\statespace{}^{n-\ell+1}$, the corresponding probability $\smash{\mathrm{P}^\ast_{\vert x_{1:\ell-1}}(x_{1:n})}$, related with $p^\ast$ through Equation~\eqref{Eq: precise probability on algebra}, is equal to $\smash{p'(x_\ell \vert x_{1:\ell-1}) \prod_{i=\ell}^{n-1} p( x_{i+1} \vert x_{1:i})}$.
Hence, recalling Equation~\eqref{Eq: prop: equality between all models for finitary in imprecise case 3}, we infer that
\begin{align*}
\mupprev{}(\mupprev{}(f \vert X_{1:\ell}) \vert x_{1:\ell-1})
= \sum_{x_{\ell:n}\in{\statespace{}}^{n-\ell+1}} \hspace*{-6pt} f(x_{1:n}) \mathrm{P}^\ast_{\vert x_{1:\ell-1}}(x_{1:n})
&= \sum_{z_{1:\ell-1}\in{\statespace{}}^{\ell-1}}
\sum_{x_{\ell:n}\in{\statespace{}}^{n-\ell+1}} f(z_{1:\ell-1} x_{\ell:n}) \mathrm{P}^\ast_{\vert x_{1:\ell-1}}(z_{1:\ell-1} x_{\ell:n}) \\
&= \prev_{\mathrm{meas},p^\ast}(f \vert x_{1:\ell-1}) 
= \upprev{\mathrm{meas},p^\ast}(f \vert x_{1:\ell-1}), 
\end{align*}
where the second step follows from the fact that $\mathrm{P}^\ast_{\vert x_{1:\ell-1}}(z_{1:\ell-1} x_{\ell:n}) = 0$ if $z_{1:\ell-1} \not= x_{1:\ell-1}$ due to Equation~\eqref{Eq: precise probability on algebra}, where the penultimate step follows from $f = \smash{\sum_{z_{1:n}\in\statespace{}^n} f(z_{1:n}) \indica{z_{1:n}}}$ and the definition of the Lebesgue integral \cite[Section~2.6.1]{shiryaev2016probabilityPartI}, and where the last step follows from Proposition~\ref{prop: alternative extension is an extension}.
Hence, since the equation above holds for any $x_{1:\ell-1}\in\statespace{}^{\ell-1}$,
we have established expression~\eqref{Eq: prop: equality between all models for finitary in imprecise case} for the case where $m = \ell -1$, concluding our proof of the induction step.
\end{proof}

That $\mupprev{}$, $\upprev{\mathrm{A}}$ and $\upprevvovkk{}$ are all equal on the domain $\setoffingambles{}$ now follows straightforwardly by combining Lemma~\ref{lemma: equality between all models for finitary in imprecise case} with Corollary~\ref{corollary: inequality between mupprev and upprev_A}, Theorem~\ref{Theorem: measure satisfies P1-P5} and Proposition~\ref{proposition: largest on n-measurables}.

\begin{proposition}\label{prop: equality between all models for finitary in imprecise case}
We have that \(\mupprev ( f \vert s) = \upprev{\mathrm{A}}( f \vert s) = \upprevvovkk{}( f \vert s)\) for all \(f \in \setoffingambles{}\) and all \(s \in \situations{}\).
\end{proposition}
\begin{proof}
Due to Corollary~\ref{corollary: inequality between mupprev and upprev_A}, it suffices to prove that \(\mupprev ( f \vert s) \geq \upprev{\mathrm{A}}( f \vert s)\) for all \(f \in \setoffingambles{}\) and all \(s \in \situations{}\).
This can easily be done by combining Theorem~\ref{Theorem: measure satisfies P1-P5}, Lemma~\ref{lemma: equality between all models for finitary in imprecise case} and Proposition~\ref{proposition: largest on n-measurables}.
Indeed, Theorem~\ref{Theorem: measure satisfies P1-P5} says that $\mupprev{}$ is an element of $\setofupprev{}_{1-5}(\tree{})\subseteq\setofupprev{}_{1-4}(\tree{})$, and Lemma~\ref{lemma: equality between all models for finitary in imprecise case} says that $\mupprev{}$ satisfies \ref{P iterated general}, so Proposition~\ref{proposition: largest on n-measurables} implies that \(\mupprev ( f \vert s) \geq \upprev{}{'}( f \vert s)\) for all \(f \in \setoffingambles{}\), all \(s \in \situations{}\) and all $\upprev{}{'}\in\setofupprev{}_{1-4}(\tree{})$.
In particular, this implies that \(\mupprev ( f \vert s) \geq \upprev{\mathrm{A}}( f \vert s)\) for all \(f \in \setoffingambles{}\) and all \(s \in \situations{}\) because, by definition, $\upprev{\mathrm{A}}\in\setofupprev{}_{1-5}(\tree{})\subseteq\setofupprev{}_{1-4}(\tree{})$.
\end{proof}

Whether the equality with $\mupprev{}$ can be extended to a more general domain, remains an open question that we would like to pursue in our future research. 
We expect that there is much to learn from the results in \cite[Section~9.2]{Vovk2019finance}, provided that they can be translated to our setting here.

\section{Conclusion}\label{Sect: conclusion}

We have put forward a small set of simple axioms \ref{P compatibility}--\ref{P continuity} and argued that they are desirable for a global upper expectation in the context of discrete-time finite-state uncertain processes.
We have established the existence of a unique most conservative model under these axioms, and have in addition given sufficient conditions to uniquely characterise this most conservative model.
Using this characterisation, we have shown that our most conservative upper expectation coincides with a version of the game-theoretic upper expectation used by Shafer and Vovk, and therefore has particularly powerful mathematical properties, despite the simplicity of our defining axioms.

We also considered an alternative, more traditional measure-theoretic approach, where we defined a global belief model using an upper envelope of upper integrals corresponding to compatible precise probability trees.
We have shown that this model satisfies axioms \ref{P compatibility}--\ref{P continuity} and therefore, that our most conservative model gives guaranteed upper bounds on the value of this measure-theoretic upper expectation.
For precise probability trees, both models, and consequently also the game-theoretic model, coincide.
In the imprecise case, equality was only established for the domain of finitary gambles, and it remains to be seen whether this can be extended to a more general domain.

Given our current findings, we belief there are a number of reasons why our global model $\upprev{\mathrm{A}}$ candidates as an excellent choice when it comes to modelling discrete-time finite-state uncertain processes.
First and foremost, its definition is based on the axioms~\ref{P compatibility}--\ref{P continuity}; a set of properties which, in our opinion, are essential for a global belief model to have.
Our model $\upprev{\mathrm{A}}$ is taken to be the most conservative or, equivalently, the least informative model under these axioms because we do not want to include any further information.
In this way, we obtain a definition that is both convincing and interpretationally clear.
Moreover, if one desires her global belief model to also have other properties, additional to~\ref{P compatibility}--\ref{P continuity}, then our model $\upprev{\mathrm{A}}$ would still serve as a conservative upper bound for her model.
As a second argument, we recall from Section~\ref{Sect: additional properties} that our model $\upprev{\mathrm{A}}$ possesses rather strong continuity properties, as well as several fundamental probabilistic laws.
Thirdly and lastly, our model $\upprev{\mathrm{A}}$ coincides with the game-theoretic upper expectation $\upprevvovkk{}$ and dominates the measure-theoretic upper expectation $\upprev{\mathrm{meas}}$.
Hence, our model $\upprev{\mathrm{A}}$ plays an important role to anyone: for a person with a game-theoretic background, the equivalent models $\upprev{\mathrm{A}}$ and $\upprevvovkk{}$ can be used interchangeably and can therefore benefit from each others properties; for a person with a measure-theoretic background, our model $\upprev{\mathrm{A}}$ always gives a conservative upper bound for the model $\upprev{\mathrm{meas}}$.


One counterargument that could be given to the plea above, is whether axiom \ref{P continuity} is really desirable or even justified for a global belief model.
Our justification for it is based on the fact that we consider upper expectations of non-finitary variables to be abstract idealisations of upper expectations of finitary variables.
Still, this could be seen as a rather artificial and unnecessary assumption, similar to the countable additivity property in measure theory.
However, compared to other common continuity properties, e.g. monotone convergence, property~\ref{P continuity} is fairly weak since it only applies to sequences of finitary gambles and only imposes an inequality on the value of the global upper expectations.
Moreover, that both the game-theoretic and the measure-theoretic model satisfy property \ref{P continuity}, can only be seen as additional motivation to adopt it.

\section*{Acknowledgements}

The research of Jasper De Bock and Gert de Cooman was partially funded by the European Commission's H2020 programme, through the UTOPIAE Marie Curie Innovative Training Network, H2020-MSCA-ITN-2016, Grant Agreement number 722734, as well as by FWO (Research Foundation - Flanders),
through project number 3GO28919, entitled ``Efficient inference in large-scale queueing models using imprecise continuous-time Markov chains''.
We are also indebted two anonymous reviewers that provided helpful feedback and pointed us to some relevant literature.

\bibliographystyle{elsarticle-num-names}
\bibliography{IJA8590}

\appendix

\section{Basic Measure-Theoretic Concepts}\label{subsect: Basic measure-theoretic concepts}

A \emph{measurable space} \((\mathscr{Y},\mathscr{F})\) is a couple, where \(\mathscr{Y}\) is a non-empty set and \(\mathscr{F}\) is a \(\sigma\)-algebra on \(\mathscr{Y}\).
We say that \(A \subseteq \mathscr{Y}\) is \(\mathscr{F}\)-\emph{measurable} if \(A \in \mathscr{F}\).
Such a set \(A \in \mathscr{F}\) is also called an (\(\mathscr{F}\)-measurable) \emph{event}.
We say that an extended real-valued function \(f \colon \mathscr{Y} \to \extreals{}\) is \(\mathscr{F}\)-\emph{measurable} if the set \(f^{-1}(A) \coloneqq \{y \in \mathscr{Y} \colon f(y) \in A \}\) is \(\mathscr{F}\)-measurable for every \(A \in \mathscr{B}(\extreals{})\).
Here, \(\mathscr{B}(\extreals{})\) denotes the Borel algebra on \(\extreals{}\), being the \(\sigma\)-algebra generated by all open---or equivalently, closed---sets in \(\extreals{}\).
Recall that we consider the usual order topology on \(\extreals{}\) and therefore, the Borel algebra \(\mathscr{B}(\extreals{})\) can alternatively be generated from the sets \(\{x \in \extreals{} \colon x \leq c\}\) where \(c \in \reals{}\); see for instance \cite[Section~2.2.2]{shiryaev2016probabilityPartI}.
Then, as we have done in the main text, we can alternatively characterise the \(\mathscr{F}\)-measurable functions as those functions \(f \colon \mathscr{Y} \to \extreals{}\) such that \(\{y \in \mathscr{Y} \colon f(y) \leq c \}\) is \(\mathscr{F}\)-measurable for all \(c \in \reals{}\) \cite[Theorem~13.1.~(i)]{billingsley1995probability}.
Typically, an \(\mathscr{F}\)-measurable (extended) real-valued function \(f\) is also called a \emph{random variable}.

A \emph{probability space} \((\mathscr{Y},\mathscr{F},\mathrm{P})\) is a measurable space \((\mathscr{Y},\mathscr{F})\) equipped with a \(\sigma\)-additive probability measure \(\mathrm{P}\) on \(\mathscr{F}\).
We say that an event \(A \in \mathscr{F}\) is \(\mathrm{P}\)\emph{-null} \/ if \(\mathrm{P}(A) = 0\).
We will also say that a property about the elements in \(\mathscr{Y}\) holds \(\mathrm{P}\)\emph{-almost surely} (\(\mathrm{P}\)-a.s.) if there is a \(\mathrm{P}\)\emph{-null} event \(A \in \mathscr{F}\) such that the property holds for all \(y \in \mathscr{Y} \setminus A\).
Note that the intersection $A \cap B \in \mathscr{F}$ of two \(\mathrm{P}\)-almost sure events $A,B\in \mathscr{F}$, is itself also \(\mathrm{P}\)-almost sure.
Now, for any probability space \((\mathscr{Y},\mathscr{F},\mathrm{P})\) and any extended real-valued function \(f \colon \mathscr{Y} \to \extreals{}\) that is \(\mathscr{F}\)-measurable, the (measure-theoretic) expectation \(\mathrm{E}(f)\) of \(f\) will be defined using the Lebesgue integral \cite{shiryaev2016probabilityPartI,billingsley1995probability,royden2010real}:
\(\mathrm{E}(f)\coloneqq\int_{\mathscr{Y}} f \dif\mathrm{P}\).
As we already mentioned in the main text, this integral is definitely well-defined if $f$ is non-negative and is otherwise only defined if \(\smash{\min\{\int_{\mathscr{Y}} f^{+} \dif \mathrm{P},\int_{\mathscr{Y}} f^{-} \dif \mathrm{P}\}} < +\infty\), with $f^{+}=f^{\vee 0}$ and $f^{-}=-f^{\wedge 0}$, in which case we let $\smash{\int_{\mathscr{Y}} f \dif \mathrm{P}\coloneqq\int_{\mathscr{Y}} f^+ \dif \mathrm{P}-\int_{\mathscr{Y}} f^- \dif \mathrm{P}}$.
When we say that `$\smash{\mathrm{E}(f)}$ exists', we take this to mean that its defining integral $\smash{\int_{\mathscr{Y}} f \dif \mathrm{P}}$ is well-defined.
Furthermore, an extended real-valued function \(f \colon \mathscr{Y} \to \extreals{}\) is called \({\mathrm{P}\textit{-integrable}}\) if it is \(\mathscr{F}\)-measurable and \(\smash{\int_{\mathscr{Y}} \vert f\vert \dif \mathrm{P}=\int_{\mathscr{Y}} f^+ \dif \mathrm{P}+\int_{\mathscr{Y}} f^- \dif \mathrm{P}} < +\infty\).

The (measure-theoretic) conditional expectation \(\mathrm{E}(f \vert \mathscr{G})\) of a non-negative \(\mathscr{F}\)-measurable extended real-valued function \(f\) with respect to a \(\sigma\)-algebra \(\mathscr{G} \subseteq \mathscr{F}\), is any extended real-valued function on \(\mathscr{Y}\) that is \(\mathscr{G}\)-measurable and that satisfies \(\int_{A} f \dif \mathrm{P} = \int_{A} \mathrm{E}(f \vert \mathscr{G}) \dif \mathrm{P}\) or, equivalently, \(\int_{\mathscr{Y}} f \genindica{A} \dif \mathrm{P} = \int_{\mathscr{Y}} \mathrm{E}(f \vert \mathscr{G}) \genindica{A} \dif \mathrm{P}\) for all \(A \in \mathscr{G}\).
If \(f\) is \(\mathscr{F}\)-measurable and non-negative, the conditional expectation \(\mathrm{E}(f \vert \mathscr{G})\) exists and is unique up to a \(\mathrm{P}\)-null set \cite[Section 2.7]{shiryaev2016probabilityPartI}.
The value of \(\mathrm{E}(f \vert \mathscr{G})\) on a \(\mathrm{P}\)-null set can be chosen arbitrarily since it will not change the value of the integral \(\int_{A} \mathrm{E}(f \vert \mathscr{G}) \dif \mathrm{P} = \int_{\mathscr{Y}} \mathrm{E}(f \vert \mathscr{G}) \genindica{A} \dif \mathrm{P}\); see Property~\ref{prop measure: equality} below.
Similarly to the unconditional case, the measure-theoretic conditional expectation \(\mathrm{E}(f \vert \mathscr{G})\) of a general \(\mathscr{F}\)-measurable extended real-valued function \(f\) with respect to a \(\sigma\)-algebra \(\mathscr{G} \subseteq \mathscr{F}\) is only defined if \(\min\{\mathrm{E}(f^+ \vert \mathscr{G}), \mathrm{E}(f^- \vert \mathscr{G})\} < +\infty\) \(\mathrm{P}\)-almost surely, and then \(\mathrm{E}(f \vert \mathscr{G}) \coloneqq \mathrm{E}(f^+ \vert \mathscr{G}) - \mathrm{E}(f^- \vert \mathscr{G})\) up to a \(\mathrm{P}\)-null set.
If \(f\) is \(\mathrm{P}\)-integrable, such a conditional expectation \(\mathrm{E}(f \vert \mathscr{G})\) always exists; see also \cite[Section 34]{billingsley1995probability}.
We moreover have the following convenient properties.

\begin{lemma}\label{lemma: properties of conditional measure-theoretic expectation}
For any probability space \((\mathscr{Y},\mathscr{F},\mathrm{P})\) and any two extended real-valued functions \(f\) and \(g\) that are \(\mathscr{F}\)-measurable, the following properties hold, provided that the considered unconditional or conditional expectations exist.
\begin{enumerate}[leftmargin=*,ref={\upshape M}\arabic*,label={\upshape M\arabic*.},itemsep=3pt, series=propertiesmeasure]
\item\label{prop measure: linearity} \(\mathrm{E}(a f + b) = a \mathrm{E}(f) + b\) for all \(a,b \in \reals{}\);
\item\label{prop measure: monotonicity} \(f \leq g\) \(\Rightarrow \mathrm{E}(f) \leq  \mathrm{E}(g)\);
\item\label{prop measure: equality} \(f = g\) {\normalfont{\(\mathrm{P}\)-almost surely }} \(\Rightarrow \mathrm{E}(f) =  \mathrm{E}(g)\);
\item\label{prop measure: condition on sigma-algebra} \(\mathrm{E}(f \vert \mathscr{F}) = f\) {\(\mathrm{P}\)-almost surely};
\item\label{prop measure: conditional to unconditional} \(\mathrm{E}(f \vert \mathscr{G}^\ast) = \mathrm{E}(f)\) where \(\mathscr{G}^\ast = \{\emptyset,\mathscr{Y}\}\);
\item\label{prop measure: law of iterated} \(\mathrm{E}(\mathrm{E}(f \vert \mathscr{G}_2) \vert \mathscr{G}_1) = \mathrm{E}(f \vert \mathscr{G}_1)\) {\(\mathrm{P}\)-almost surely} for every two \(\sigma\)-algebras \(\mathscr{G}_1 \subseteq  \mathscr{G}_2 \subseteq \mathscr{F}\).
\item\label{prop measure: bounds} 
\(\inf f \leq \mathrm{E}(f\vert\mathscr{G})\) {\(\mathrm{P}\)-almost surely} and \/ \(\mathrm{E}(f\vert\mathscr{G}) \leq \sup f\) {\(\mathrm{P}\)-almost surely} for any \(\sigma\)-algebra \(\mathscr{G} \subseteq \mathscr{F}\);
\end{enumerate}
\end{lemma}
\begin{proof}
Let us first show that \ref{prop measure: conditional to unconditional} follows from \cite[Section 2.7.4(\textbf{E*})]{shiryaev2016probabilityPartI}.
There, it is stated that \(\mathrm{E}(f \vert \mathscr{G}^\ast) = \mathrm{E}(f)\) holds \(\mathrm{P}\)-almost surely.
However, since \(\mathscr{G}^\ast = \{\emptyset,\mathscr{Y}\}\) and \(\mathrm{E}(f \vert \mathscr{G}^\ast)\) is \(\mathscr{G}^\ast\)-measurable, we find that \(\mathrm{E}(f \vert \mathscr{G}^\ast)\) must be constant.
So if \(\mathrm{E}(f \vert \mathscr{G}^\ast) = \mathrm{E}(f)\) holds \(\mathrm{P}\)-almost surely, it also holds on the entire domain \(\mathscr{Y}\).

To prove Property \ref{prop measure: linearity}, consider \cite[Section 2.7.4(\textbf{D*})]{shiryaev2016probabilityPartI} which in particular states that \(\mathrm{E}(a f + b \vert \mathscr{G}^\ast) = a \mathrm{E}(f \vert \mathscr{G}^\ast) + \mathrm{E}(b \vert \mathscr{G}^\ast)\) $\mathrm{P}$-almost surely.
Since $\mathrm{E}(b \vert \mathscr{G}^\ast) = b$ $\mathrm{P}$-almost surely by \cite[Section 2.7.4(\textbf{A*})]{shiryaev2016probabilityPartI}, we have that \(\mathrm{E}(a f + b \vert \mathscr{G}^\ast) = a \mathrm{E}(f \vert \mathscr{G}^\ast) + b\) $\mathrm{P}$-almost surely, which by \ref{prop measure: conditional to unconditional} implies that \(\mathrm{E}(a f + b) = a \mathrm{E}(f) + b\) $\mathrm{P}$-almost surely.
Since both sides are constants, \ref{prop measure: linearity} follows.

Property~\ref{prop measure: monotonicity} follows from \cite[Section 2.7.4(\textbf{B*})]{shiryaev2016probabilityPartI}, which states that $\mathrm{E}(f \vert \mathscr{G}^\ast) \leq \mathrm{E}(g \vert \mathscr{G}^\ast)$ $\mathrm{P}$-almost surely if $f\leq g$ $\mathrm{P}$-almost surely.
Indeed, if $f\leq g$, then certainly $f\leq g$ $\mathrm{P}$-almost surely and then \cite[Section 2.7.4(\textbf{B*})]{shiryaev2016probabilityPartI} together with \ref{prop measure: conditional to unconditional} implies that $\mathrm{E}(f)\leq\mathrm{E}(g)$ $\mathrm{P}$-almost surely; since both sides are constants, \ref{prop measure: monotonicity} follows.

To prove~\ref{prop measure: equality}, note that $f=g$ $\mathrm{P}$-almost surely implies that both $f\leq g$ $\mathrm{P}$-almost surely and $f\geq g$ $\mathrm{P}$-almost surely.
Then \ref{prop measure: equality} follows by applying the previous reasoning to both of these inequalities.
Furthermore, Properties~\ref{prop measure: condition on sigma-algebra} and~\ref{prop measure: law of iterated} are taken directly from \cite[Section 2.7.4]{shiryaev2016probabilityPartI}.

To see that property~\ref{prop measure: bounds} holds, consider \cite[Section 2.7.4(\textbf{A*})]{shiryaev2016probabilityPartI}, which in particular says that $\mathrm{E}(\sup f \vert \mathscr{G}) = \sup f$ $\mathrm{P}$-almost surely.
Combining this with mononicity \cite[Section~2.7.4(\textbf{B*})]{shiryaev2016probabilityPartI}, and taking into account that the intersection of two $\mathrm{P}$-almost sure events is itself also $\mathrm{P}$-almost sure, we indeed find that $\mathrm{E}(f \vert \mathscr{G}) \leq \sup f$ $\mathrm{P}$-almost surely.
The fact that  $\inf f \leq \mathrm{E}(f \vert \mathscr{G})$ $\mathrm{P}$-almost surely can then easily be obtained using the linearity of $\mathrm{E}(\cdot \vert \mathscr{G})$ \cite[Section~2.7.4(\textbf{D*})]{shiryaev2016probabilityPartI}.
\end{proof}

\begin{lemma}\label{lemma: limits for precise measure-theoretic expectation}
Consider any probability space \((\mathscr{Y},\mathscr{F},\mathrm{P})\), any \(\sigma\)-algebra \(\mathscr{G} \subseteq \mathscr{F}\) and any sequence \(\{f_n\}_{n \in \natz{}}\) of extended real-valued functions such that \(f_n\) is \(\mathscr{F}\)-measurable for all \(n \in \natz{}\).
\begin{enumerate}[leftmargin=*,ref={\upshape M}\arabic*,label={\upshape M\arabic*.},itemsep=3pt,resume=propertiesmeasure]
\item\label{lemma: limits for precise measure-theoretic expectation ii} 
If \/ \(\{f_n\}_{n \in \natz{}}\) is non-decreasing and there is an \(\mathscr{F}\)-measurable function \(f^\ast\) such that \/ \(\mathrm{E}(f^\ast) > -\infty\) and \(f_n \geq f^\ast\) for all \(n \in \natz{}\), then
\begin{align*}
\lim_{n \to +\infty} \mathrm{E}(f_n) = \mathrm{E}(f)
\text{ where } \lim_{n \to +\infty} f_n = f.
\end{align*} 

\item\label{lemma: limits for precise measure-theoretic expectation iii} 
If \/ \(\{f_n\}_{n \in \natz{}}\) is non-increasing and there is an \(\mathscr{F}\)-measurable function \(f^\ast\) such that \/ \(\mathrm{E}(f^\ast) < +\infty\) and \(f_n \leq f^\ast\) for all \(n \in \natz{}\), then
\begin{align*}
\lim_{n \to +\infty} \mathrm{E}(f_n) = \mathrm{E}(f)
\text{ where } \lim_{n \to +\infty} f_n = f.
\end{align*} 
\end{enumerate}
\end{lemma}
\begin{proof}
Both properties can directly be obtained by applying \ref{prop measure: conditional to unconditional} to the properties stated in \cite[Theorem 2.7.2]{shiryaev2016probabilityPartI} and observing that both sides of the acquired equations will be constants (allowing us to infer that the equations always hold, instead of~\(\mathrm{P}\)-almost surely as is the case for the more general versions in \cite[Theorem 2.7.2]{shiryaev2016probabilityPartI}).
\end{proof}

A (discrete) \emph{filtration}
\(\{\mathscr{F}_n\}_{n \in \natz{}}\) on a measurable space \((\mathscr{Y},\mathscr{F})\) is a sequence of increasing \(\sigma\)-algebras in \(\mathscr{F}\); so \(\mathscr{F}_0 \subseteq \mathscr{F}_1 \subseteq \cdots \mathscr{F}\).
We will then also use \(\mathscr{F}_\infty\) to denote the smallest \(\sigma\)-algebra \(\sigma(\cup_{n \in \natz{}} \mathscr{F}_n)\) generated by the \(\sigma\)-algebras \(\mathscr{F}_n\).
We say that \((\mathscr{Y},\mathscr{F},\{\mathscr{F}_n\}_{n \in \natz{}})\) is a \emph{filtered measurable space} if \((\mathscr{Y},\mathscr{F})\) is equipped with a filtration \(\{\mathscr{F}_n\}_{n \in \natz{}}\), and moreover say that \((\mathscr{Y},\mathscr{F},\{\mathscr{F}_n\}_{n \in \natz{}},\mathrm{P})\) is a \emph{filtered probability space} if it additionally has a \(\sigma\)-additive measure \(\mathrm{P}\) on \(\mathscr{F}\).
A sequence \(\{\mathscr{N}_n\}_{n \in \natz{}}\) of extended real-valued functions on \(\mathscr{Y}\) is called a \emph{measure-theoretic process} in a filtered probability space \((\mathscr{Y},\mathscr{F},\{\mathscr{F}_n\}_{n \in \natz{}}, \mathrm{P})\) if \(\mathscr{N}_n\) is \(\mathscr{F}_n\)-measurable for all \(n \in \natz\).
It is moreover called a \emph{measure-theoretic martingale} if \(\mathscr{N}_n\) is real-valued and \(\mathrm{E}(\mathscr{N}_{n+1} \vert \mathscr{F}_{n}) = \mathscr{N}_{n}\) \(\mathrm{P}\)-almost surely for all \(n \in \natz{}\), where the existence of $\mathrm{E}(\mathscr{N}_{n+1} \vert \mathscr{F}_{n})$ is an implicit condition.
The following two results are fundamental in establishing a relation between the measure-theoretic and the game-theoretic framework.
\begin{proposition}[L\'evy's zero-one law; {\cite[Theorem 7.4.3]{shiryaev2019probabilityPartII}}]\label{Prop: Levy measure-theoretic}
For any filtered probability space \((\mathscr{Y},\mathscr{F},\{\mathscr{F}_n\}_{n \in \natz{}}, \mathrm{P})\) and any\/ \(\mathrm{P}\)-integrable function \(f\), we have that
\begin{align*}
\lim_{n \to +\infty} \mathrm{E}(f \vert \mathscr{F}_n) = \mathrm{E}(f \vert \mathscr{F}_\infty) \, \text{\normalfont \, \(\mathrm{P}\)-almost surely}
\end{align*} 
\end{proposition}

\begin{proposition}[Ville's theorem; {\cite[Proposition 8.14]{Shafer:2005wx}}]\label{Prop: Ville}
For any filtered probability space \((\mathscr{Y},\mathscr{F},\{\mathscr{F}_n\}_{n \in \natz{}}, \mathrm{P})\) and any \(A \in \mathscr{F}_\infty\), we have that \(\mathrm{P}(A) = 0\) if and only if there is a non-negative measure-theoretic martingale that converges to \(+\infty\) on \(A\).
\end{proposition}


\section{Proof of {}Proposition~\ref{prop: unconditional precise measure is equal to game for bounded}}\label{Sect: Proofs}

Proposition~\ref{prop: unconditional precise measure is equal to game for bounded} will be shown to hold by establishing the equivalence between the global measure-theoretic and global game-theoretic (upper) expectation on a particular restricted domain and then applying Theorem~\ref{theorem: Vovk is the largest}.
Before we relate both global models, however, we first establish the following alternative characterisation for the unique extension $\extlupprev{s}$ of the local linear (upper) expectation $\lupprev{s}$ corresponding to any probability mass function $p_s$ on $\statespace{}$:

\begin{lemma}\label{lemma: extend precise local models to bounded belows}
Consider any probability mass function $p_s$ on $\statespace{}$ and let \/ \(\lupprev{s}\) be the unique linear expectation on \(\setofgengambles{}(\statespace{})\) corresponding to $p_s$.
Then we have that \(\extlupprev{s}(f) = \smash{\sum_{x \in \statespace{}} f(x) p_s(x)}\) for all \(f \in \setofgenextvariables{}(\statespace{})\).
\end{lemma}
\begin{proof}
Let $\upprev{s}\colon\setofgenextvariables{}(\statespace{})\to\extreals{}$ be defined by $\smash{\upprev{s}(f)\coloneqq\smash{\sum_{x \in \statespace{}} f(x) p_s(x)}}$ for all \(f \in \setofgenextvariables{}(\statespace{})\).
We show that $\smash{\extlupprev{s}}$ coincides with $\upprev{s}$.
It is clear that this is the case on the domain $\setofgengambles{}(\statespace{})$ of all gambles on $\statespace{}$; this follows immediately from the fact that $\extlupprev{s}$ is an extension of $\lupprev{s}$ together with the fact that $\upprev{s}(f) = \smash{\sum_{x \in \statespace{}} f(x) p_s(x)} = \lupprev{s}(f)$ for all $f\in\setofgengambles{}(\statespace{})$.
To see that both coincide on $\smash{\setofgenextvariablesb{}(\statespace{})}$, consider any $\smash{f\in\setofgenextvariablesb{}(\statespace{})}$ and observe that
\begin{align*}
\extlupprev{s}(f) 
\overset{\text{\ref{upper cuts}}}{=} \lim_{c\to+\infty} \extlupprev{s}(f^{\wedge c})
= \lim_{c\to+\infty} \sum_{x \in \statespace{}} f^{\wedge c}(x) p_s(x)
= \sum_{x \in \statespace{}} \lim_{c\to+\infty} f^{\wedge c}(x) p_s(x)
&= \sum_{x \in \statespace{}}  p_s(x) \lim_{c\to+\infty} f^{\wedge c}(x)\\
&= \sum_{x \in \statespace{}}  p_s(x) f(x)
= \upprev{s}(f),
\end{align*}
where the second equality follows from the fact that $f^{\wedge c}\in\setofgengambles{}(\statespace{})$ for any $c\in\reals{}$ and, as we have just shown, $\extlupprev{s}$ and $\upprev{s}$ coincide on $\setofgengambles{}(\statespace{})$; 
where the third equality follows from the fact that \(\statespace{}\) is finite and that, for all \(x\in \statespace{}\), \(f^{\wedge c}(x) p_s(x)\) is real and non-decreasing in $c$; 
where the fourth equality follows from the fact that \(p_s(x) \geq 0\) for all \(x \in \statespace{}\) and our convention that $0\cdot(+\infty) = 0$; and where the last equality follows from the definition of \(\upprev{s}\).
In an analogous way, we show that $\extlupprev{s}$ and $\upprev{s}$ coincide on their entire domain $\setofgenextvariables{}(\statespace{})$: consider any \(f\in\setofgenextvariables{}(\statespace{})\) and note that
\begin{align*}
\extlupprev{s}(f) 
\overset{\text{\ref{lower cuts}}}{=} \lim_{c\to-\infty} \extlupprev{s}(f^{\vee c})
= \lim_{c\to-\infty} \sum_{x \in \statespace{}} f^{\vee c}(x) p_s(x)
= \sum_{x \in \statespace{}} \lim_{c\to-\infty} f^{\vee c}(x) p_s(x)
&= \sum_{x \in \statespace{}}  p_s(x) \lim_{c\to-\infty} f^{\vee c}(x) \\
&= \sum_{x \in \statespace{}}  p_s(x) f(x)
= \upprev{s}(f),
\end{align*}
where the second equality follows from the fact that $f^{\vee c}\in\setofgenextvariablesb{}(\statespace{})$ for any $c\in\reals{}$ and, as we have just shown, $\extlupprev{s}$ and $\upprev{s}$ coincide on $\setofgenextvariablesb{}(\statespace{})$; 
where the third equality follows from the fact that \(\statespace{}\) is finite, our convention that $+\infty-\infty=-\infty+\infty=+\infty$ and that, for all \(x\in \statespace{}\), \(f^{\vee c}(x) p_s(x)\) is in $\reals{}\cup\{+\infty\}$ and non-decreasing in $c$; 
the fourth equality follows from the fact that \(p_s(x) \geq 0\) for all \(x \in \statespace{}\) and our convention that $0\cdot(-\infty)=0$; and the last equality follows from the definition of \(\upprev{s}\).
\end{proof}

Henceforth, for any precise probability tree $p$ and any $s\in\situations{}$, we will always let $\extlupprev{s}$  be this unique extension of the linear expectation $\lupprev{s}$ corresponding to the probability mass function $p(\cdot\vert s)$.
We will then say that a map $\martingale{}\colon\situations{}\to\extreals{}$ is a game-theoretic supermartingale \emph{with respect to a precise probability tree} $p$ if $\martingale{}$ is a game-theoretic supermartingale with respect to the local upper expectations $\extlupprev{s}$.

Furthermore, for any precise probability tree \(p \colon s \in \situations{} \mapsto p(\cdot \vert s)\), we let \(\mathrm{P} \colon \mathscr{F} \times \situations{} \to \reals{}\) be the corresponding conditional probability measure as discussed in Section~\ref{sect: Measure-theoretic expectations in precise probability trees}, where \(\mathscr{F}\) is the \(\sigma\)-algebra generated by all cylinder events.
Recall that, for any \(s \in \situations{}\), the map \(\mathrm{P}(\cdot \vert s) = \mathrm{P}_{\vert s}\) is then a probability measure on \(\mathscr{F}\).
This allows us to apply the concepts and results in \ref{subsect: Basic measure-theoretic concepts} here, by considering the probability space \((\Omega, \mathscr{F}, \mathrm{P}_{\vert s})\), for any \(s \in \situations{}\).
For notational convenience, we will let \(\mathrm{E}_{\vert s}(f)\) be the Lebesgue integral \(\smash{\int_\Omega f \dif \mathrm{P}_{\vert s}}\) with respect to the measure \(\mathrm{P}_{\vert s}\) for all \(\mathscr{F}\)-measurable \(f \in \smash{\setofextvariables{}}\) and all \(s \in \situations{}\) such that \(\smash{\int_\Omega f \dif \mathrm{P}_{\vert s}}\) exists.
So \(\smash{\mathrm{E}_{\vert s}(f)}\) is an alternative notation for \(\mathrm{E}_{\mathrm{meas},p}(f \vert s)\).
We introduce this notation because it allows us to write \(\mathrm{E}_{\vert s}(f \vert \mathscr{G})\) to denote a \(\mathscr{G}\)-measurable function representing the measure-theoretic expectation of \(f\) conditional on a \(\sigma\)-algebra \(\mathscr{G} \subseteq \mathscr{F}\), as defined in \ref{subsect: Basic measure-theoretic concepts}.
Furthermore, we equip the measurable space \((\Omega, \mathscr{F})\) with the filtration \(\{\mathscr{F}_n\}_{n \in \natz{}}\) where, for any \(n \in \natz{}\), \(\mathscr{F}_n\) is the \(\sigma\)-algebra generated by all cylinder events \(\Gamma(x_{1:n})\) where \(x_{1:n} \in \statespace{}^n\).
Note that, for any \(n \in \natz{}\), an \(\mathscr{F}_n\)-measurable function is then an \(n\)-measurable variable because the cylinder events \(\Gamma(x_{1:n})\) form the atoms of \(\mathscr{F}_n\) and \(\statespace{}\) is finite.
Hence, any measure-theoretic process \(\{\mathscr{N}_n\}_{n \in \natz{}}\) in \((\Omega, \mathscr{F}, \{\mathscr{F}_n\}_{n \in \natz{}})\) is a sequence of \(n\)-measurable variables.
This allows us to write \(\mathscr{N}_n(x_{1:n})\) for any \(n \in \natz{}\) and any \(x_{1:n} \in \statespace{}^n\) to mean the constant value of \(\mathscr{N}_n\) on all \(\omega \in \Gamma(x_{1:n})\).

\begin{lemma}\label{lemma: measure martingale to game martingale}
Consider any $x_{1:n}\in\situations{}$, any precise probability tree $p$ and let \/ $\mathrm{P}$ be the corresponding conditional probability measure.
Then, for any non-negative measure-theoretic martingale \(\{\mathscr{N}_i\}_{i \in \natz{}}\) in the filtered probability space \((\Omega,\mathscr{F},\{\mathscr{F}_m\}_{m \in \natz{}},\mathrm{P}_{\vert x_{1:n}})\), there is a non-negative game-theoretic (super)martingale\footnote{As shown in the proof, the inequality \(\smash{\martingale{}(z_{1:m}) \geq \extlupprev{z_{1:m}}(\martingale{}(z_{1:m} \cdot))}\) is actually an equality, for all \(z_{1:m} \in \situations{}\).
Combining this with the non-negativity of $\martingale{}$ and the characterisation of the upper expectations $\smash{\extlupprev{z_{1:m}}}$ that we established in Lemma~\ref{lemma: extend precise local models to bounded belows}, it can then be deduced that $-\martingale{}$ is also a game-theoretic supermartingale.
In this case, we can therefore actually call \(\martingale{}\) a \emph{game-theoretic martingale}.}
\(\martingale{}\) with respect to the tree $p$ such that \/ \(\liminf \martingale{} \geq_{x_{1:n}} \liminf_{i \to +\infty} \mathscr{N}_i\) and moreover \(\martingale{}(x_{1:n}) = \mathscr{N}_0(\Box)\). 
\end{lemma}
\begin{proof}
Fix any non-negative measure-theoretic martingale \(\{\mathscr{N}_i\}_{i \in \natz{}}\) and let \(\martingale{} \colon \situations{} \to \extreals{}\) be defined by 
\begin{align*}
\martingale{}(z_{1:m}) \coloneqq
\begin{cases}
\mathscr{N}_m(z_{1:m}) &\text{ if } n\leq m, \, z_{1:n}=x_{1:n} \text{ and } \mathrm{P}_{\vert x_{1:n}}(z_{1:m}) > 0; \\
+\infty &\text{ if } n\leq m, \, z_{1:n}=x_{1:n} \text{ and } \mathrm{P}_{\vert x_{1:n}}(z_{1:m}) = 0; \\
\mathscr{N}_n(x_{1:n}) &\text{otherwise, }
\end{cases}
\quad \text{ for all } z_{1:m} \in \situations{}.
\end{align*}
We show that \(\martingale{}\) is a non-negative game-theoretic supermartingale such that \(\martingale{}(x_{1:n}) = \mathscr{N}_0(\Box)\) and \(\liminf \martingale{} \geq_{x_{1:n}} \liminf_{i \to +\infty} \mathscr{N}_i\). 

Let us first show that \(\martingale{}(z_{1:m}) \geq \extlupprev{z_{1:m}}(\martingale{}(z_{1:m} \cdot))\) for all \(z_{1:m} \in \situations{}\) and therefore, that $\martingale{}$ is a game-theoretic supermartingale.
Recall that, because \(\{\mathscr{N}_i\}_{i \in \natz{}}\) is a non-negative measure-theoretic martingale, we have that \(\{\mathscr{N}_i\}_{i \in \natz{}}\) is a sequence of non-negative real ${i\text{-measurable}}$ variables such that \(\mathrm{E}_{\vert x_{1:n}}(\mathscr{N}_{i+1} \vert \mathscr{F}_{i}) = \mathscr{N}_{i}\) \(\mathrm{P}_{\vert x_{1:n}}\)-almost surely for all \(i \in \natz{}\) (note that the considered expectations exist because \(\{\mathscr{N}_i\}_{i \in \natz{}}\) is non-negative).
For all \(i \in \natz{}\) and all \(A \in \mathscr{F}_i\), since \(\smash{\mathrm{E}_{\vert x_{1:n}}(\mathscr{N}_{i+1}\vert\mathscr{F}_{i})}\) and \(\mathscr{N}_{i}\) only differ on a \(\mathrm{P}_{\vert x_{1:n}}\)-null set, we have that \(\smash{\int_{A}\mathrm{E}_{\vert x_{1:n}}(\mathscr{N}_{i+1} \vert \mathscr{F}_{i}) \dif \mathrm{P}_{\vert x_{1:n}}} = \int_{A} \mathscr{N}_{i} \dif \mathrm{P}_{\vert x_{1:n}}\) because of \ref{prop measure: equality}.
In particular, this implies that 
\begin{equation*}
\smash{\int_{\Gamma(z_{1:m})}\mathrm{E}_{\vert x_{1:n}}(\mathscr{N}_{m+1}\vert\mathscr{F}_{m}) \dif \mathrm{P}_{\vert x_{1:n}}} = \int_{\Gamma(z_{1:m})} \mathscr{N}_{m} \dif \mathrm{P}_{\vert x_{1:n}} \text{ for any } m\in\natz \text{ and any } z_{1:m}\in\statespace{}^{m}.
\end{equation*}
Moreover, since \(\Gamma(z_{1:m}) \in \mathscr{F}_m\), it follows from the definition of the measure-theoretic conditional expectation that 
\(\int_{\Gamma(z_{1:m})} \mathscr{N}_{m+1} \dif \mathrm{P}_{\vert x_{1:n}} 
= \int_{\Gamma(z_{1:m})} \mathrm{E}_{\vert x_{1:n}}(\mathscr{N}_{m+1}\vert\mathscr{F}_{m}) \, \dif\mathrm{P}_{\vert x_{1:n}}\).
Combining both equalities, we find that
\begin{align*}
\int_{\Gamma(z_{1:m})} \hspace{-8pt}\mathscr{N}_{m+1} \dif \mathrm{P}_{\vert x_{1:n}} 
= \int_{\Gamma(z_{1:m})} \hspace{-8pt}\mathscr{N}_{m} \dif \mathrm{P}_{\vert x_{1:n}}  
= \mathscr{N}_{m}(z_{1:m}) \int_{\Gamma(z_{1:m})} \hspace{-9pt} \dif \mathrm{P}_{\vert x_{1:n}} 
&= \mathscr{N}_{m}(z_{1:m}) \mathrm{P}_{\vert x_{1:n}}(z_{1:m}),  
\end{align*}
where the second equality follows from \ref{prop measure: linearity} and the fact that \(\mathscr{N}_{m}\) is constant and real on the cylinder event \(\Gamma(z_{1:m})\).
Moreover, because \(\mathscr{N}_{m+1}\) is \((m+1)\)-measurable and real-valued, the term on the left hand side of the equation above reduces to the finite sum $\sum_{z_{m+1} \in \statespace{}} \mathscr{N}_{m+1}(z_{1:m+1}) \mathrm{P}_{\vert x_{1:n}}(z_{1:m+1})$, allowing us to conclude that
\begin{align}\label{Eq: lemma: measure martingale to game martingale}
 \mathscr{N}_{m}(z_{1:m}) \mathrm{P}_{\vert x_{1:n}}(z_{1:m}) 
 = \sum_{z_{m+1} \in \statespace{}} \mathscr{N}_{m+1}(z_{1:m+1}) \mathrm{P}_{\vert x_{1:n}}(z_{1:m+1}) \text{ for any } m\in\natz \text{ and any } z_{1:m}\in\statespace{}^m.  
\end{align}

Consider now first any $m\geq n$ and any $z_{1:m}\in\statespace{}^m$ such that $z_{1:n}=x_{1:n}$ and $\mathrm{P}_{\vert x_{1:n}}(z_{1:m})>0$. It then follows from the definition of $\martingale{}$ that
\begin{align*}
\martingale{}(z_{1:m}) \mathrm{P}_{\vert x_{1:n}}(z_{1:m})
=  \mathscr{N}_{m}(z_{1:m}) \mathrm{P}_{\vert x_{1:n}}(z_{1:m}) 
&\overset{\text{\eqref{Eq: lemma: measure martingale to game martingale}}}{=} 
\sum_{z_{m+1} \in \statespace{}} \mathscr{N}_{m+1}(z_{1:m+1}) \mathrm{P}_{\vert x_{1:n}}(z_{1:m+1}) \\
&\hspace{4pt} = 
\sum_{z_{m+1} \in \statespace{}} \martingale{}(z_{1:m+1}) \mathrm{P}_{\vert x_{1:n}}(z_{1:m+1}) \\
&\hspace{4pt} = \sum_{z_{m+1} \in \statespace{}} \martingale{}(z_{1:m+1}) \mathrm{P}_{\vert x_{1:n}}(z_{1:m}) p(z_{m+1} \vert z_{1:m}) \\
&\hspace{4pt} = \mathrm{P}_{\vert x_{1:n}}(z_{1:m}) \sum_{z_{m+1} \in \statespace{}} \martingale{}(z_{1:m+1})  p(z_{m+1} \vert z_{1:m})
= \mathrm{P}_{\vert x_{1:n}}(z_{1:m}) \, \extlupprev{z_{1:m}}(\martingale{}(z_{1:m} \cdot)),
\end{align*}
where the third equality follows because \(\mathscr{N}_{m+1}(z_{1:m+1})\) only differs from \(\martingale{}(z_{1:m+1})\) if \(\mathrm{P}_{\vert x_{1:n}}(z_{1:m+1}) = 0\) [because $n \leq m+1$ and $z_{1:n}=x_{1:n}$] and our convention that $0 \cdot (+\infty) = 0$, where the fourth equality follows from Equation~\eqref{Eq: precise probability on algebra} together with the fact that $n < m+1$ and $z_{1:n}=x_{1:n}$, and where the last equality follows from the expression for $\extlupprev{z_{1:m}}$ that we established in Lemma~\ref{lemma: extend precise local models to bounded belows}.
Recall that \(\mathrm{P}_{\vert x_{1:n}}(z_{1:m}) > 0\), so we can devide both sides by \(\mathrm{P}_{\vert x_{1:n}}(z_{1:m})\) to obtain that \(\martingale{}(z_{1:m}) = \smash{\extlupprev{z_{1:m}}(\martingale{}(z_{1:m} \cdot))}\).
As a consequence, the condition that \(\martingale{}(z_{1:m}) \geq \smash{\lupprev{z_{1:m}}(\martingale{}(z_{1:m} \cdot))}\) is satisfied for all \(z_{1:m} \in \situations{}\) such that $m\geq n$, $z_{1:n}=x_{1:n}$ and $\mathrm{P}_{\vert x_{1:n}}(z_{1:m})>0$.

Second, consider any \(z_{1:m} \in \situations{}\) such that $m\geq n$, $z_{1:n}=x_{1:n}$ and $\mathrm{P}_{\vert x_{1:n}}(z_{1:m})=0$.
Then we also have that \(\mathrm{P}_{\vert x_{1:n}}(z_{1:m+1}) = 0\) for any \(z_{m+1} \in \statespace{}\) because $\Gamma(z_{1:m+1}) \subset \Gamma(z_{1:m})$, and it therefore follows from the definition of \(\martingale{}\) that \(\martingale{}(z_{1:m}) = \martingale{}(z_{1:m+1}) = +\infty\) for all \(z_{m+1}\in\statespace{}\).
Hence, using Lemma~\ref{lemma: extend precise local models to bounded belows}, we find that
\begin{equation*}
\extlupprev{z_{1:m}}(\martingale{}(z_{1:m}\cdot))
= \sum_{z_{m+1}\in\statespace{}} \martingale{}(z_{1:m+1})  p(z_{m+1} \vert z_{1:m}) = +\infty = \martingale{}(z_{1:m}),
\end{equation*}
which establishes that indeed \(\martingale{}(z_{1:m}) \geq \extlupprev{z_{1:m}}(\martingale{}(z_{1:m} \cdot))\).
So we have that \(\martingale{}(z_{1:m}) \geq \smash{\extlupprev{z_{1:m}}(\martingale{}(z_{1:m} \cdot))}\) for all \(z_{1:m} \in \situations{}\) such that $m\geq n$ and $z_{1:n}=x_{1:n}$.

Thirdly, consider any other situation \(z_{1:m} \in \situations{}\).
That is, consider any \(z_{1:m} \in \situations{}\) such that $n>m$ or $z_{1:n}\neq x_{1:n}$.
It then follows from the definition of $\martingale{}$ that $\martingale{}(z_{1:m})=\mathscr{N}_n(x_{1:n})$ and, as we will now show, that also $\martingale{}(z_{1:m+1})=\mathscr{N}_n(x_{1:n})$ for all $z_{m+1}\in\statespace{}$.
Our assumption about $z_{1:m}$ tells us that there are two cases: either $n\leq m$ and $z_{1:n}\not= x_{1:n}$, either $n>m$. 
If $n\leq m$ and $z_{1:n}\not= x_{1:n}$, then also $n\leq m+1$ and therefore, due to the definition of $\martingale{}$, we find that $\martingale{}(z_{1:m+1})=\mathscr{N}_n(x_{1:n})$ for all $z_{m+1}\in\statespace{}$. 
If $n> m$, then either $n>m+1$ or $n=m+1$.
If $n>m+1$ or, equivalently, $n\not\leq m+1$, then it is once more clear from the definition of $\martingale{}$ that $\martingale{}(z_{1:m+1})=\mathscr{N}_n(x_{1:n})$ for all $z_{m+1}\in\statespace{}$. 
If $n=m+1$, then for any $z_{m+1}\in\statespace{}$, we either have that $z_{1:m+1} = x_{1:n}$ or that $z_{1:m+1} \not= x_{1:n}$.
In the latter case, it is clear, again due to the definition of $\martingale{}$, that $\martingale{}(z_{1:m+1})=\mathscr{N}_n(x_{1:n})$.
Otherwise, so if $z_{1:m+1} = x_{1:n}$, the definition of $\martingale{}$ also implies that $\martingale{}(z_{1:m+1})=\martingale{}(x_{1:n})=\mathscr{N}_n(x_{1:n})$ because $\mathrm{P}_{\vert x_{1:n}}(x_{1:n})=1>0$ according to Equation~\eqref{Eq: precise probability on algebra}.  
We conclude that, indeed, $\martingale{}(z_{1:m})=\mathscr{N}_n(x_{1:n})$ and $\martingale{}(z_{1:m+1})=\mathscr{N}_n(x_{1:n})$ for all $z_{m+1}\in\statespace{}$. 
It now follows trivially from Lemma~\ref{lemma: extend precise local models to bounded belows} that \(\martingale{}(z_{1:m}) = \smash{\extlupprev{z_{1:m}}(\martingale{}(z_{1:m} \cdot))}\), and therefore definitely that \(\martingale{}(z_{1:m}) \geq \smash{\extlupprev{z_{1:m}}(\martingale{}(z_{1:m} \cdot))}\).

So we have proved that \(\martingale{}(z_{1:m}) \geq \smash{\extlupprev{z_{1:m}}(\martingale{}(z_{1:m} \cdot))}\) for all \(z_{1:m} \in \situations{}\) and therefore, that $\martingale{}$ is indeed a game-theoretic supermartingale with respect to the tree~$p$.
Moreover, \(\martingale{}\) is non-negative because \(\{\mathscr{N}_i\}_{i \in \natz{}}\) is non-negative.
That \(\liminf_{i \to +\infty} \martingale{}(\omega^i) \geq \liminf_{i \to +\infty} \mathscr{N}_i(\omega)\) holds for all \(\omega \in \Gamma(x_{1:n})\), follows immediately from the fact that \(\martingale{}(\omega^i) \geq \mathscr{N}_i(\omega^i) = \mathscr{N}_i(\omega)\) for any \(\omega \in \Gamma(x_{1:n})\) and any $i \geq n$, where the inequality is simply implied by the definition of $\martingale{}$ together with the fact that $\omega^n = x_{1:n}$ and $i \geq n$.
It still remains to show that $\martingale{}(x_{1:n})=\mathscr{N}_0(\Box)$.

We only need to prove that $\mathscr{N}_0(\Box) = \mathscr{N}_n(x_{1:n})$, since the desired equality then trivially follows from the definition of $\martingale{}$ and the fact that $\mathrm{P}_{\vert x_{1:n}}(x_{1:n})=1>0$ due to Equation~\eqref{Eq: precise probability on algebra}.
To do so, consider any $z_{1:m}\in\situations{}$ such that $m<n$ and $z_{1:m}=x_{1:m}$.
Note that, by Equation~\eqref{Eq: precise probability on algebra}, we then have that $\mathrm{P}_{\vert x_{1:n}}(z_{1:m})=1$.
Moreover, it should also be clear from Equation~\eqref{Eq: precise probability on algebra} that $\mathrm{P}_{\vert x_{1:n}}(z_{1:m+1})=1$ if $z_{m+1}=x_{m+1}$ [because then $z_{1:m+1} = x_{1:m+1}$ and $m+1\leq n$] and that $\mathrm{P}_{\vert x_{1:n}}(z_{1:m+1})=0$ otherwise [because then $z_{1:m+1} \not= x_{1:m+1}$].
Hence, plugging this back into \eqref{Eq: lemma: measure martingale to game martingale}, we find that $\mathscr{N}_{m}(z_{1:m}) = \mathscr{N}_{m+1}(z_{1:m}x_{m+1})$.  
Since this holds for any $z_{1:m}\in\situations{}$ such that $m<n$ and $z_{1:m}=x_{1:m}$, it follows that $\mathscr{N}_{m}(x_{1:m}) = \mathscr{N}_{m+1}(x_{1:m+1})$ for all $m<n$. This clearly implies that $\mathscr{N}_{0}(\Box) = \mathscr{N}_{1}(x_1) = \cdots = \mathscr{N}_{n}(x_{1:n})$.
\end{proof}

In the following proof, we will write, for any two $f,g\in\setofextvariables{}$, any $s\in\situations{}$ and any (unconditional) probability measure $\mathrm{P}'$ on $\mathscr{F}$, that $f =_s g$ $\mathrm{P}$-almost surely---and similarly for $\geq_s$ and $\leq_s$---if the event $\{\omega\in\Gamma(s) \colon f(\omega)\not= g(\omega)\}$ is $\mathrm{P}$-null.
Note that then $f =_s g$ $\mathrm{P}$-almost surely if $f = g$ $\mathrm{P}$-almost surely. 

\begin{proof}[Proof of Proposition~\ref{prop: unconditional precise measure is equal to game for bounded}]
Fix any \(\mathscr{F}\)-measurable \(f' \in \setofgambles{}\) and any $\sit\in\situations{}$.
It suffices to show that \(\smash{\mathrm{E}_{\mathrm{meas},p}(f' \vert \sit)} = \upprevvovkk{}(f' \vert \sit)\); 
the desired equality is then automatically implied by Theorem~\ref{theorem: Vovk is the largest}.
First observe that, because \(f'\) is bounded and \(\mathscr{F}\)-measurable, \(\mathrm{E}_{\mathrm{meas},p}(f' \vert \sit) = \mathrm{E}_{\vert\sit}(f')\) exists.
We will now prove that \(\mathrm{E}_{\mathrm{meas},p}(f \vert \sit) = \upprevvovkk{}(f \vert \sit)\) for the non-negative $\mathscr{F}$-measurable gamble $f\coloneqq f' - \inf f'$ (the variable $f$ is indeed a gamble because $f'$ is a gamble and therefore $\inf f' \in\reals$), which then implies that \(\mathrm{E}_{\mathrm{meas},p}(f' \vert \sit) = \upprevvovkk{}(f' \vert \sit)\) because \(\mathrm{E}_{\mathrm{meas},p}(\cdot\vert\sit)\) and \(\smash{\upprevvovkk{}}\) both satisfy the constant additivity property; see \ref{prop measure: linearity} and \ref{vovk coherence 6}.

We first show that \(\smash{\upprevvovkk{}(f\vert\sit) \leq \mathrm{E}_{\mathrm{meas},p}(f \vert \sit)}\).
To do so, we will prove that there is some \(c \in \reals{}\) such that, for all \(\epsilon>0\), there is a bounded below game-theoretic supermartingale \(\martingale{}_\epsilon\) with respect to the tree $p$ such that \(\martingale{}_\epsilon(\sit) = \mathrm{E}_{\mathrm{meas},p}(f \vert \sit) + \epsilon c\) and \(\liminf \martingale{}_\epsilon \geq_{\sit} f\).
Indeed, the desired inequality then follows immediately from the definition of $\upprevvovkk{}$.

Consider the filtered probability space \((\Omega,\mathscr{F},\{\mathscr{F}_m\}_{m \in \natz{}},\mathrm{P}_{\vert x_{1:n}})\) and the corresponding measure-theoretic expectation \(\mathrm{E}_{\vert x_{1:n}} \hspace*{-2pt} = \mathrm{E}_{\mathrm{meas},p}(\cdot \vert x_{1:n})\).
Since \(f\) is bounded and \(\mathscr{F}\)-measurable, it is surely \(\mathrm{P}_{\vert x_{1:n}}\hspace*{-2pt}\)-integrable (the Lebesgue integral of a bounded $\mathscr{F}$-measurable function is real-valued), and therefore, by Proposition~\ref{Prop: Levy measure-theoretic} in \ref{subsect: Basic measure-theoretic concepts}, we have that
\begin{align*}
\lim_{m \to +\infty} \mathrm{E}_{\vert x_{1:n}} (f \vert \mathscr{F}_m) = \mathrm{E}_{\vert x_{1:n}} (f \vert \mathscr{F}_\infty) \, \text{\normalfont \, \(\mathrm{P}_{\vert x_{1:n}}\)\hspace*{-2pt}-almost surely}.
\end{align*}
Note that \(\mathscr{F}_\infty = \sigma(\cup_{m \in \natz{}} \mathscr{F}_m) = \mathscr{F}\) is the smallest \(\sigma\)-algebra generated by all cylinder events, which, by \ref{prop measure: condition on sigma-algebra}, implies that \(\mathrm{E}_{\vert x_{1:n}} (f \vert \mathscr{F}_\infty) = f\), \(\mathrm{P}_{\vert x_{1:n}}\)-almost surely. 
Hence, since the intersection of two \(\mathrm{P}_{\vert x_{1:n}}\)\hspace*{-2pt}-almost sure events is itself also \(\mathrm{P}_{\vert x_{1:n}}\)\hspace*{-2pt}-almost sure, we have that \(\lim_{m \to +\infty} \mathrm{E}_{\vert x_{1:n}} (f \vert \mathscr{F}_m) = f\) \(\mathrm{P}_{\vert x_{1:n}}\)\hspace*{-2pt}-almost surely.
Moreover, note that the conditional expectations \(\mathrm{E}_{\vert x_{1:n}} (f \vert \mathscr{F}_m)\) can be chosen in such a way that the sequence \(\{\mathrm{E}_{\vert x_{1:n}} (f \vert \mathscr{F}_m)\}_{m \in \natz{}}\) is a non-negative measure-theoretic martingale in the filtered probability space \((\Omega,\mathscr{F},\{\mathscr{F}_m\}_{m \in \natz{}},\mathrm{P}_{\vert x_{1:n}})\).
Indeed, each \(\mathrm{E}_{\vert x_{1:n}} (f \vert \mathscr{F}_m)\) is real-valued and non-negative \(\mathrm{P}_{\vert x_{1:n}}\)\hspace*{-2pt}-almost surely because of \ref{prop measure: bounds} and the fact that \(f\) is bounded and non-negative. So, since the value of the measure-theoretic conditional expectation can be chosen arbitrarily on a null set, \(\mathrm{E}_{\vert x_{1:n}} (f \vert \mathscr{F}_m)\) can be chosen such that it is real-valued and non-negative everywhere.
Moreover, because of \ref{prop measure: law of iterated}, we have that
\(\smash{\mathrm{E}_{\vert x_{1:n}} \bigl( \mathrm{E}_{\vert x_{1:n}} (f \vert \mathscr{F}_{m+1}) \vert \mathscr{F}_m \bigr) 
 = \mathrm{E}_{\vert x_{1:n}} (f \vert \mathscr{F}_m)}\) \(\mathrm{P}_{\vert x_{1:n}}\)\hspace*{-2pt}-almost surely for all \(m \in \natz{}\), where $\mathrm{E}_{\vert x_{1:n}} \bigl( \mathrm{E}_{\vert x_{1:n}} (f \vert \mathscr{F}_{m+1}) \vert \mathscr{F}_m \bigr)$ exists because $\mathrm{E}_{\vert x_{1:n}} (f \vert \mathscr{F}_{m+1})$ is $\mathscr{F}$-measurable and non-negative.
So let \(\{\mathrm{E}_{\vert x_{1:n}} (f \vert \mathscr{F}_m)\}_{m \in \natz{}}\) be a version of the conditional expectations that forms a non-negative measure-theoretic martingale.
We can then use Lemma~\ref{lemma: measure martingale to game martingale} to infer the existence of a non-negative game-theoretic supermartingale \(\martingale{}_0\) with respect to the tree $p$ such that \(\liminf \martingale{}_0 \geq_{x_{1:n}} \liminf_{m \to +\infty} \mathrm{E}_{\vert x_{1:n}} (f \vert \mathscr{F}_m)\) and \(\martingale{}_0(x_{1:n}) = \mathrm{E}_{\vert x_{1:n}} (f \vert \mathscr{F}_0)(\Box)\). 
Then, because \(\lim_{m \to +\infty} \mathrm{E}_{\vert x_{1:n}} (f \vert \mathscr{F}_m) = f\) \(\mathrm{P}_{\vert x_{1:n}}\)\hspace*{-2pt}-almost surely and therefore also \(\lim_{m \to +\infty} \mathrm{E}_{\vert x_{1:n}} (f \vert \mathscr{F}_m) =_{x_{1:n}} \hspace*{-2pt}f\) \(\mathrm{P}_{\vert x_{1:n}}\)\hspace*{-2pt}-almost surely, we have that \(\liminf \martingale{}_0 \geq_{x_{1:n}} f\) \(\mathrm{P}_{\vert x_{1:n}}\)\hspace*{-2pt}-almost surely. 
Moreover, we have that 
$\martingale{}_0(x_{1:n}) = \mathrm{E}_{\vert x_{1:n}} (f \vert \mathscr{F}_0)(\Box) = \mathrm{E}_{\vert x_{1:n}} (f)$, due to the fact that \(\mathscr{F}_0 = \{\emptyset,\Omega\}\) and property~\ref{prop measure: conditional to unconditional}.
So, we can conclude that \(\martingale{}_0\) is a non-negative game-theoretic supermartingale with respect to $p$ such that \(\martingale{}_0(x_{1:n}) = \mathrm{E}_{\vert x_{1:n}}(f)= \mathrm{E}_{\mathrm{meas},p}(f\vert x_{1:n})\) and \(\liminf \martingale{}_0 \geq_{x_{1:n}} f\) \(\mathrm{P}_{\vert x_{1:n}}\)-almost surely.

As our final step towards obtaining $\martingale{}_\epsilon$, consider Proposition~\ref{Prop: Ville} and note that it ensures that there is a non-negative measure-theoretic supermartingale \(\{\mathscr{N}_m\}_{m \in \natz{}}\) in \((\Omega,\mathscr{F},\{\mathscr{F}_m\}_{m \in \natz{}},\mathrm{P}_{\vert x_{1:n}})\) that converges to \(+\infty\) on all paths \(\omega \in \Gamma(x_{1:n})\) such that \(\liminf \martingale{}_0(\omega) < f(\omega)\). 
Indeed, the set of all such paths \(\omega\) has probability zero because \(\liminf \martingale{}_0 \geq_{x_{1:n}} \hspace*{-2pt}f\) \(\mathrm{P}_{\vert x_{1:n}}\)-almost surely.
Let $c\coloneqq\mathscr{N}_0(\Box)$, which is real-valued because \(\{\mathscr{N}_m\}_{m \in \natz{}}\) is a measure-theoretic supermartingale.
By Lemma~\ref{lemma: measure martingale to game martingale}, we find that there is a non-negative game-theoretic supermartingale \(\martingale{}'\) with respect to $p$ such that $\liminf\martingale{}' \geq_{x_{1:n}} \liminf_{m\to+\infty}\mathscr{N}_m$ and $\martingale{}'(x_{1:n}) = \mathscr{N}_0(\Box) = c$.
Since \(\{\mathscr{N}_m\}_{m \in \natz{}}\) converges to $+\infty$ on all paths \(\omega\in\Gamma(x_{1:n})\) such that \(\liminf \martingale{}_0(\omega) < f(\omega)\), it follows that $\martingale{}'$ also converges to $+\infty$ on all such paths \(\omega\).
Consider now any \(\epsilon > 0\) and let \(\martingale{}_\epsilon\) be the process defined by \(\martingale{}_\epsilon(s) \coloneqq \martingale{}_0(s) + \epsilon \martingale{}'(s)\) for all \(s \in \situations{}\).
Then \(\martingale{}_\epsilon\) is clearly non-negative---and therefore bounded below---and it is a game-theoretic supermartingale because of \cite[Lemma~12]{Tjoens2020FoundationsARXIV}.\footnote{Alternatively, instead of using \cite[Lemma~12]{Tjoens2020FoundationsARXIV}, one could also easily deduce this using the alternative expression for the local models \(\extlupprev{s}\) that we established in Lemma~\ref{lemma: extend precise local models to bounded belows}.}
Furthermore, note that \(\liminf \martingale{}_\epsilon(\omega) \geq f(\omega)\) for all \(\omega \in \Gamma(x_{1:n})\).
Indeed, if \(\liminf \martingale{}_0(\omega) \geq f(\omega)\) for some \(\omega \in \Gamma(x_{1:n})\), then also \(\liminf \martingale{}_\epsilon(\omega) \geq f(\omega)\) because \(\epsilon\) and \(\martingale{}'\) are non-negative.
If \(\liminf \martingale{}_0(\omega) < f(\omega)\) for some \(\omega \in \Gamma(x_{1:n})\), then \(\martingale{}'\), and therefore also \(\epsilon \martingale{}'\), converges to \(+\infty\), which, together with the non-negativity of \(\martingale{}_0\), implies that \(\martingale{}_\epsilon\) converges to \(+\infty\) on \(\omega\).
Hence, also in this case, we have that \(\liminf \martingale{}_\epsilon(\omega) \geq f(\omega)\) so we can conclude that \(\liminf \martingale{}_\epsilon \geq_{x_{1:n}} f\).
Moreover, recall that \(\martingale{}'(x_{1:n}) = c \in\reals{}\) and that \(\martingale{}_0(x_{1:n}) = \mathrm{E}_{\vert x_{1:n}}(f)\), so we have that \(\martingale{}_\epsilon(x_{1:n}) = \martingale{}_0(x_{1:n}) + \epsilon \martingale{}' (x_{1:n}) = \mathrm{E}_{\vert x_{1:n}}(f) + \epsilon c = \mathrm{E}_{\mathrm{meas},p}(f\vert x_{1:n})+ \epsilon c\).
Hence, $\martingale{}_\epsilon$ satisfies all the desired conditions and we conclude that indeed \(\smash{\upprevvovkk{}(f\vert\sit) \leq \mathrm{E}_{\mathrm{meas},p}(f \vert \sit)}\).

Then we are left to show the remaining inequality \(\upprevvovkk{}(f \vert \sit) \geq \mathrm{E}_{\mathrm{meas},p}(f \vert \sit)\).
However, this can be easily deduced from the already obtained inequality and the self-conjugacy of \(\mathrm{E}_{\mathrm{meas},p}\).
Indeed, \(\mathrm{E}_{\vert \sit}(-f)\) exists because \(-f\) is \(\mathscr{F}\)-measurable and bounded, so we can apply \ref{prop measure: linearity} to find that \(\mathrm{E}_{\mathrm{meas},p}(f \vert \sit) = \mathrm{E}_{\vert \sit}(f) 
= - \mathrm{E}_{\vert \sit}(- f)
= - \mathrm{E}_{\mathrm{meas},p}(- f \vert \sit)\).
Since we have already shown that \(\smash{\upprevvovkk{}(g \vert \sit) \leq \mathrm{E}_{\mathrm{meas},p}(g \vert \sit)}\) for all $\mathscr{F}$-measurable \(g \in \setofgambles{}\), we have in particular that \(\smash{\upprevvovkk{} (-f \vert \sit) \leq \mathrm{E}_{\mathrm{meas},p}(-f \vert \sit)}\), which implies that \(\smash{\mathrm{E}_{\mathrm{meas},p}(f \vert \sit)} = \smash{- \mathrm{E}_{\mathrm{meas},p}(- f \vert \sit)} \leq \smash{- \upprevvovkk{} (-f \vert \sit)} = \smash{ \lowprevvovkk{} (f \vert \sit)} \leq \smash{\upprevvovkk{}(f \vert \sit)}\), where the last inequality follows from \ref{vovk coherence 1} and the fact that $\upprev{\mathrm{A}}=\upprevvovkk{}$---and therefore also $\lowprev{}_{\hspace*{1pt}\raisebox{-1pt}{\scriptsize$\mathrm{A}$}}=\lowprevvovkk{}$---because of Theorem~\ref{theorem: Vovk is the largest}.
\end{proof}

\end{document}